\documentclass[draft]{birkart}

\usepackage{amssymb,amsmath,amsfonts,amsthm,
mathrsfs}

 \newtheorem{thm}{Theorem}[section]
 \newtheorem{cor}[thm]{Corollary}
 \newtheorem{lem}[thm]{Lemma}
 \newtheorem{prop}[thm]{Proposition}
 \theoremstyle{definition}
 \newtheorem{defn}[thm]{Definition}
 \theoremstyle{remark}
 \newtheorem{rem}[thm]{Remark}
 \newtheorem*{ex}{Example}
 \numberwithin{equation}{section}

\newcommand{\wt}[1]{\widetilde{#1}}

\newcommand{\Cinf}{\ensuremath{\mathcal{C}^\infty}}
\newcommand{\Cinfc}{\ensuremath{\mathcal{C}^\infty_{\text{c}}}}
\newcommand{\D}{\ensuremath{{\mathcal D}}}
\renewcommand{\S}{\mathscr{S}}
\newcommand{\E}{\ensuremath{{\mathcal E}}}

\newcommand{\LL}{\mathcal{L}}
\newcommand{\Lb}{\mathcal{L}_{\rm{b}}}
\newcommand{\mO}{\mathcal{O}}
\newcommand{\mM}{\mathcal{M}}
\newcommand{\mF}{\mathcal{F}}
\newcommand{\mP}{\mathcal{P}}
\newcommand{\MPhi}{\mathcal{M}_\Phi}

\newcommand{\rfunc}{{\rm{basic}}}
 
\newcommand{\dslash}{d\hspace{-0.4em}{ }^-\hspace{-0.2em}}

\newcommand{\mb}[1]{\ensuremath{\mathbb{#1}}}
\newcommand{\N}{\mb{N}}

\newcommand{\R}{\mb{R}}
\newcommand{\C}{\mb{C}}

\newcommand{\G}{\ensuremath{{\mathcal G}}}
\newcommand{\Gt}{\ensuremath{{\mathcal G}_\tau}}

\newcommand{\Gc}{\ensuremath{{\mathcal G}_\mathrm{c}}}
\newcommand{\Gcinf}{\ensuremath{{\mathcal G}^\infty_\mathrm{c}}}
\newcommand{\GS}{\G_{{\, }\atop{\hskip-4pt\scriptstyle\S}}\!}
\newcommand{\EM}{\ensuremath{{\mathcal E}_{M}}}

\newcommand{\Et}{\ensuremath{{\mathcal E}_{\tau}}}

\newcommand{\Nt}{\ensuremath{{\mathcal N}_{\tau}}}
\newcommand{\Neg}{\mathcal{N}}

\newcommand{\Ginf}{\ensuremath{\G^\infty}}

\newcommand{\lara}[1]{\langle #1 \rangle}
\newcommand{\WF}{\mathrm{WF}}

\newcommand{\singsupp}{\mathrm{sing\, supp}}
\newcommand{\supp}{\mathrm{supp}}
\newcommand{\Char}{\ensuremath{\text{Char}}}
\newcommand{\zs}{\setminus 0}
\newcommand{\CO}[1]{\ensuremath{T^*(#1) \zs}}
\newcommand{\ssc}{\mathrm{sc}}

\newfont{\bigmath}{cmr12 at 13pt}

\newcommand{\val}{\mathrm{v}} 
\newcommand{\esp}{\mathrm{e}}

\newcommand{\eps}{\varepsilon}

\newcommand{\om}{\omega}
\newcommand{\Om}{\Omega}

\newcommand{\Syscu}{{\wt{\underline{\mathcal{S}}}}_{\,\ssc}}
\renewcommand{\d}{\ensuremath{\partial}}
\newcommand{\diff}[1]{\frac{d}{d#1}}

\begin{document}
%
\title[Generalized Fourier Integral Operators]
 {Generalized Fourier Integral Operators\\ on spaces of Colombeau type}
\author[C. Garetto]{Claudia Garetto}

\address{%
Institut f\"ur Grundlagen der Bauingenieurwissenschaften\\
Leopold-Franzens-Uni\-ver\-si\-t\"at Innsbruck\\
Technikerstr. 13\\
A 6020 Innsbruck\\
Austria}

\email{claudia@mat1.uibk.ac.at}

\thanks{This work was completed with the support of FWF (Austria), grants T305-N13 and  Y237-N13.}
\subjclass{Primary 35S30; Secondary 46F30}
\keywords{Fourier integral operators, Colombeau algebras}

\date{December 31, 2007}

\begin{abstract}
Generalized Fourier integral operators (FIOs) acting on Colombeau algebras are defined.
This is based on a theory of generalized oscillatory integrals (OIs) whose phase functions
as well as amplitudes may be generalized functions of Colombeau type. 
The mapping properties of these FIOs are studied as the composition with a generalized pseudodifferential operator. Finally, the microlocal Colombeau regularity for OIs and the influence of the FIO action on generalized wave front sets are investigated. This theory of generalized FIOs is motivated by the
need of a general framework for partial differential operators with non-smooth coefficients and
distributional data.  
\end{abstract}

\maketitle

\section{Introduction}
 
This work is part of a program that aims to solve linear partial differential equations with non-smooth
coefficients and highly singular data and investigate the qualitative properties of the solutions.
A well established theory with powerful analytic methods is available in the case of operators with
(relatively) smooth coefficients \cite{Hoermander:V1-4}, but cannot be applied to many models from physics which involve non-smooth variations of the physical parameters. These models require indeed partial differential operators where the smoothness assumption on the coefficients is dropped. Furthermore, in case of nonlinear operations (cf. \cite{HdH:01, HS:68, O:92}), the theory of distribution does not provide a general framework in which solutions exist.

An alternative framework is provided by the theory of Colombeau algebras of generalized functions
\cite{Colombeau:85, GKOS:01, O:92}. We recall that the space of distributions $\D'(\Om)$ is embedded via convolution with a mollifier in the Colombeau algebra $\G(\Om)$ of generalized functions on $\Om$ and interpreting the non-smooth coefficients and data as elements of the Colombeau algebra, existence and uniqueness has been established
for many classes of equations by now
\cite{Biagioni:90, BO:92, BO:92b, CO:90, GH:03, HO:03, LO:91, O:88, O:92, OR:98a, OR:98b, Ste:98}.
In order to study the regularity of solutions,
microlocal techniques have to be introduced into this setting, in particular, pseudodifferential
operators with generalized amplitudes and generalized wave front sets. This has been done in the papers
\cite{GGO:03, GH:05, GH:05b, Hoermann:99, GH:04, HK:01, HOP:05, NPS:98, Pilipovic:94},
with a special attention for elliptic equations and hypoellipticity.

The interest for hyperbolic equations, regularity of solutions and inverse problems (determining the non-smooth coefficients from the data  is an important problem in geophysics \cite{dHSto:02}), leads in the case of differential operators with Colombeau coefficients, to a theory of Fourier integral operators with
generalized amplitudes {\em and} generalized phase functions. This has been initiated in \cite{GHO:06} and has provided some first results on propagation of singularities in the dual $\LL(\Gc(\Om),\wt{\C})$ of the Colombeau algebra $\Gc(\Om)$. We recall that within the Colombeau algebra ${\mathcal G}(\Omega)$, regularity theory is based on
the subalgebra ${\mathcal G}^\infty(\Omega)$ of regular generalized functions, whose intersection with ${\mathcal D}'(\Omega)$ coincides with $\Cinf(\Omega)$. Since $\Ginf(\Om)\subseteq \G(\Om)\subseteq \LL(\Gc(\Om),\wt{\C})$, two different regularity theories coexist in the dual: one based on $\G(\Om)$ and one based on $\Ginf(\Om)$.  

This work can be considered as a compendium of \cite{GHO:06}, in the sense that collects (without proof) the main results achieved in \cite{GHO:06} and studies the composition between a generalized Fourier integral operator and a generalized pseudodifferential operator in addition. 

We can now describe the contents in more detail. Section 2 provides the needed background of Colombeau theory. In particular, topological concepts, generalized symbols and the definition of $\G$- and $\Ginf$-wave front set are recalled. In Subsection \ref{asymp_new} we elaborate and state in full generality the notion of asymptotic expansion of a generalized symbol introduced for the first time in \cite{GGO:03} and we prove a new and technically useful characterization. Section 3 develops the foundations for generalized Fourier integral operators: oscillatory integrals with generalized phase functions and amplitudes. They are then supplemented by an additional parameter in Section 4, leading to the notion of a Fourier integral operator with generalized amplitude and phase function. We study the mapping properties of such operators on Colombeau algebras, the extension to the dual $\LL(\Gc(\Om),\wt{\C})$ and we present suitable assumptions on phase function and amplitude which lead to $\Ginf$-mapping properties. The core of the work is Section 5, where, by making use of some technical preliminaries, we study in Theorem \ref{theo_comp} the composition $a(x,D)F_\omega(b)$ of a generalized pseudodifferential operator $a(x,D)$ with a generalized Fourier integral operator of the form 
\[
F_\omega(b)(u)(x)=\int_{\R^n}\esp^{i\omega(x,\eta)}b(x,\eta)\widehat{u}(\eta)\, \dslash\eta.
\]
The final Section 6 collects the first results of microlocal analysis for generalized Fourier integral operators obtained in \cite[Section 4]{GHO:06}. A deeper investigation of the microlocal properties of generalized Fourier integral operators is current topic of research.

\section{Basic notions: Colombeau and duality theory}
\label{section_basic}
This section gives some background of Colombeau and duality theory for the techniques used in the sequel of the current work. As main sources we refer to \cite{Garetto:05b, Garetto:05a, GGO:03, GH:05, GKOS:01}.

\subsection{Nets of complex numbers}
Before dealing with the major points of the Colombeau construction we begin by recalling some definitions concerning elements of $\mathbb{C}^{(0,1]}$.

A net $(u_\eps)_\eps$ in $\C^{(0,1]}$ is said to be \emph{strictly nonzero} if there exist $r>0$ and $\eta\in(0,1]$ such that $|u_\eps|\ge \eps^r$ for all $\eps\in(0,\eta]$.\\
The regularity issues discussed in Sections 3 and 4 will make use of the following concept of \emph{slow scale net (s.s.n)}. A slow scale net is a net $(r_\eps)_\eps\in\C^{(0,1]}$ such that 
\[
\forall q\ge 0\, \exists c_q>0\, \forall\eps\in(0,1]\qquad\qquad|r_\eps|^q\le c_q\eps^{-1}.
\]
Throughout this paper we will always consider slow scale nets $(r_\eps)_\eps$ of positive real numbers with $\inf_{\eps\in(0,1]} r_\eps\neq 0$.
A net $(u_\eps)_\eps$ in $\C^{(0,1]}$ is said to be \emph{slow scale-strictly nonzero} is there exist a slow scale net $(s_\eps)_\eps$ and $\eta\in(0,1]$ such that $|u_\eps|\ge 1/s_\eps$ for all $\eps\in(0,\eta]$.

\subsection{$\wt{\C}$-modules of generalized functions based on a locally convex topological vector space $E$}
\label{subsection_G_E}
The most common algebras of generalized functions of Colombeau type as well as the spaces of generalized symbols we deal with are introduced and investigated under a topological point of view by referring to the following models.

Let $E$ be a locally convex topological vector space topologized through the family of seminorms $\{p_i\}_{i\in I}$. The elements of 
\[
\begin{split} 
\mM_E &:= \{(u_\eps)_\eps\in E^{(0,1]}:\, \forall i\in I\,\, \exists N\in\N\quad p_i(u_\eps)=O(\eps^{-N})\, \text{as}\, \eps\to 0\},\\
\mM^\ssc_E &:=\{(u_\eps)_\eps\in E^{(0,1]}:\, \forall i\in I\,\, \exists (\omega_\eps)_\eps\, \text{s.s.n.}\quad p_i(u_\eps)=O(\omega_\eps)\, \text{as}\, \eps\to 0\},\\
\mM^\infty_E &:=\{(u_\eps)_\eps\in E^{(0,1]}:\, \exists N\in\N\,\, \forall i\in I\quad p_i(u_\eps)=O(\eps^{-N})\, \text{as}\, \eps\to 0\},\\
\Neg_E &:= \{(u_\eps)_\eps\in E^{(0,1]}:\, \forall i\in I\,\, \forall q\in\N\quad p_i(u_\eps)=O(\eps^{q})\, \text{as}\, \eps\to 0\},
\end{split}
\]
 
are called $E$-moderate, $E$-moderate of slow scale type, $E$-regular and $E$-negligible, respectively. We define the space of \emph{generalized functions based on $E$} as the factor space $\G_E := \mM_E / \Neg_E$. 

The ring of \emph{complex generalized numbers}, denoted by $\wt{\C}:=\EM/\Neg$, is obtained by taking $E=\C$. $\wt{\C}$ is not a field since by Theorem 1.2.38 in \cite{GKOS:01} only the elements which are strictly nonzero (i.e. the elements which have a representative strictly nonzero) are invertible and vice versa. Note that all the representatives of $u\in\wt{\C}$ are strictly nonzero once we know that there exists at least one which is strictly nonzero. When $u$ has a representative which is slow scale-strictly nonzero we say that it is \emph{slow scale-invertible}.

For any locally convex topological vector space $E$ the space $\G_E$ has the structure of a $\wt{\C}$-module. The ${\C}$-module $\G^\ssc_E:=\mM^\ssc_E/\Neg_E$ of \emph{generalized functions of slow scale type} and the $\wt{\C}$-module $\Ginf_E:=\mM^\infty_E/\Neg_E$ of \emph{regular generalized functions} are subrings of $\G_E$ with more refined assumptions of moderateness at the level of representatives. We use the notation $u=[(u_\eps)_\eps]$ for the class $u$ of $(u_\eps)_\eps$ in $\G_E$. This is the usual way we adopt to denote an equivalence class.

The family of seminorms $\{p_i\}_{i\in I}$ on $E$ determines a \emph{locally convex $\wt{\C}$-linear} topology on $\G_E$ (see \cite[Definition 1.6]{Garetto:05b}) by means of the \emph{valuations}
\[
\val_{p_i}([(u_\eps)_\eps]):=\val_{p_i}((u_\eps)_\eps):=\sup\{b\in\R:\qquad p_i(u_\eps)=O(\eps^b)\, \text{as $\eps\to 0$}\}
\] 
and the corresponding \emph{ultra-pseudo-seminorms} $\{\mP_i\}_{i\in I}$, where $\mP_i(u)=\esp^{-\val_{p_i}(u)}$. For the sake of brevity we omit to report definitions and properties of valuations and ultra-pseudo-seminorms in the abstract context of $\wt{\C}$-modules. Such a theoretical presentation can be found in \cite[Subsections 1.1, 1.2]{Garetto:05b}. We recall that on $\wt{\C}$ the valuation and the ultra-pseudo-norm obtained through the absolute value in $\C$ are denoted by $\val_{\wt{\C}}$ and $|\cdot|_{\esp}$ respectively. Concerning the space $\Ginf_E$ of regular generalized functions based on $E$ the moderateness properties of $\mM_E^\infty$ allows to define the valuation 
\[
\val^\infty_E ((u_\eps)_\eps):=\sup\{b\in\R:\, \forall i\in I\qquad p_i(u_\eps)=O(\eps^b)\, \text{as $\eps\to 0$}\}
\]
which extends to $\Ginf_E$ and leads to the ultra-pseudo-norm $\mP^\infty_E(u):=\esp^{-\val_E^\infty(u)}$.

The Colombeau algebra $\G(\Om)=\EM(\Om)/\Neg(\Om)$ can be obtained as a ${\wt{\C}}$-module of $\G_E$-type by choosing $E=\E(\Om)$. Topologized through the family of seminorms $p_{K,i}(f)=\sup_{x\in K, |\alpha|\le i}|\partial^\alpha f(x)|$ where $K\Subset\Om$, the space $\E(\Om)$ induces on $\G(\Om)$ a metrizable and complete locally convex $\wt{\C}$-linear topology  which is determined by the ultra-pseudo-seminorms $\mP_{K,i}(u)=\esp^{-\val_{p_{K,i}}(u)}$.  
From a structural point of view $\Om\to\G(\Om)$ is a fine sheaf of differential algebras on $\R^n$.

The Colombeau algebra $\Gc(\Om)$ of generalized functions with compact support is topologized by means of a strict inductive limit procedure. More precisely, setting $\G_K(\Om):=\{u\in\Gc(\Om):\, \supp\, u\subseteq K\}$ for $K\Subset\Om$, $\Gc(\Om)$ is the strict inductive limit of the sequence of locally convex topological $\wt{\C}$-modules $(\G_{K_n}(\Om))_{n\in\N}$, where $(K_n)_{n\in\N}$ is an exhausting sequence of compact subsets of $\Om$ such that $K_n\subseteq K_{n+1}$. We endow $\G_K(\Om)$ with the topology induced by $\G_{\mathcal{D}_{K'}(\Om)}$ where $K'$ is a compact subset containing $K$ in its interior. For more details concerning the topological structure of $\Gc(\Om)$ see \cite[Example 3.7]{Garetto:05a}.  

Regularity theory in the Colombeau context as initiated in \cite{O:92} is based on the subalgebra $\Ginf(\Om)$ of all elements $u$ of $\G(\Om)$ having a representative $(u_\eps)_\eps$ belonging to the set
\begin{multline*}
\EM^\infty(\Om):=\{(u_\eps)_\eps\in\E[\Om]:\ \forall K\Subset\Om\, \exists N\in\N\, \forall\alpha\in\N^n\\ \sup_{x\in K}|\partial^\alpha u_\eps(x)|=O(\eps^{-N})\, \text{as $\eps\to 0$}\}.
\end{multline*}
$\Ginf(\Om)$ can be seen as the intersection $\cap_{K\Subset\Om}\Ginf(K)$, where $\Ginf(K)$ is the space of all $u\in\G(\Om)$ having a representative $(u_\eps)_\eps$ satisfying the condition: $\exists N\in\N$ $\forall\alpha\in\N^n$,\ $\sup_{x\in K}|\partial^\alpha u_\eps(x)|=O(\eps^{-N})$. The ultra-pseudo-seminorms $\mP_{\Ginf(K)}(u):=\esp^{-\val_{\Ginf(K)}}$, where 
\[
\val_{\Ginf(K)}:=\sup\{b\in\R:\, \forall\alpha\in\N^n\quad \sup_{x\in K}|\partial^\alpha u_\eps(x)|=O(\eps^b)\}
\]
equip $\Ginf(\Om)$ with the topological structure of a \emph{Fr\'echet $\wt{\C}$-module}.\\
Finally, let us consider the algebra $\Gcinf(\Om):=\Ginf(\Om)\cap\Gc(\Om)$. On $\Ginf_K(\Om):=\{u\in\Ginf(\Om):\, \supp\, u\subseteq K\}$ with $K\Subset\Om$, we define the ultra-pseudo-norm $\mP_{\Ginf_K(\Om)}(u)=\esp^{-\val^\infty_K(u)}$ where $\val^\infty_K(u):=\val^\infty_{\mathcal{D}_{K'}(\Om)}(u)$ and $K'$ is any compact set containing $K$ in its interior. At this point, given an exhausting sequence $(K_n)_n$ of compact subsets of $\Om$, the strict inductive limit procedure equips $\Gcinf(\Om)=\cup_n \Ginf_{K_n}(\Om)$ with a complete and separated locally convex $\wt{\C}$-linear topology (see \cite[Example 3.13]{Garetto:05a}.  

\subsection{Topological dual of a Colombeau algebra}
A duality theory for $\wt{\C}$-modules had been developed in \cite{Garetto:05b} in the framework of topological and locally convex topological $\wt{\C}$-modules. Starting from an investigation of $\LL(\G,\wt{\C})$, the $\wt{\C}$-module of all $\wt{\C}$-linear and continuous functionals on $\G$, it provides the theoretical tools for dealing with the topological duals of the Colombeau algebras $\Gc(\Om)$ and $\G(\Om)$. In the paper $\LL(\G(\Om),\wt{\C}$ and $\LL(\Gc(\Om),\wt{\C})$ are endowed with the \emph{topology of uniform convergence on bounded subsets}. This is determined by the ultra-pseudo-seminorms 
\[
\mP_{B^\circ}(T)=\sup_{u\in B}|T(u)|_\esp,
\]
where $B$ is varying in the family of all bounded subsets of $\G(\Om)$ and $\Gc(\Om)$ respectively. For general results concerning the relation between boundedness and ultra-pseudo-seminorms in the context of locally convex topological $\wt{\C}$-modules we refer to \cite[Section 1]{Garetto:05a}. For the choice of topologies illustrated in this section Theorem 3.1 in \cite{Garetto:05a} shows the following chains of continuous embeddings:
\begin{equation}
\label{chain_1}
\Ginf(\Om)\subseteq\G(\Om)\subseteq\LL(\Gc(\Om),\wt{\C}),
\end{equation}
\begin{equation}
\label{chain_2}
\Gcinf(\Om)\subseteq\Gc(\Om)\subseteq\LL(\G(\Om),\wt{\C}),
\end{equation}  
\begin{equation}
\label{chain_3}
\LL(\G(\Om),\wt{\C})\subseteq\LL(\Gc(\Om),\wt{\C}).
\end{equation}  
In \eqref{chain_1} and \eqref{chain_2} the inclusion in the dual is given via integration $\big(u\to\big( v\to\int_\Om u(x)v(x)dx\big)\big)$ (for definitions and properties of the integral of a Colombeau generalized functions see \cite{GKOS:01}) while the embedding in \eqref{chain_3} is determined by the inclusion $\Gc(\Om)\subseteq\G(\Om)$. Since $\Om\to\LL(\Gc(\Om),\wt{\C})$ is a sheaf we can define the \emph{support of a functional $T$} (denoted by $\supp\, T$). In analogy with distribution theory, from Theorem 1.2 in \cite{Garetto:05a} we have that $\LL(\G(\Om),\wt{\C})$ can be identified with the set of functionals in $\LL(\Gc(\Om),\wt{\C})$ having compact support. 

By \eqref{chain_1} it is meaningful to measure the regularity of a functional in the dual $\LL(\Gc(\Om),\wt{\C})$ with respect to the algebras $\G(\Om)$ and $\Ginf(\Om)$. We define the \emph{$\G$-singular support} of $T$ (${\rm{singsupp}}_\G\, T$) as the complement of the set of all points $x\in\Om$ such that the restriction of $T$ to some open neighborhood $V$ of $x$ belongs to $\G(V)$. Analogously replacing $\G$ with $\Ginf$ we introduce the notion of \emph{$\Ginf$-singular support} of $T$ denoted by ${\rm{singsupp}}_{\Ginf} T$. This investigation of regularity is connected with the notions of generalized wave front sets considered in Subsection \ref{sub_sec_micro} and will be focused on the functionals in $\LL(\Gc(\Om),\wt{\C})$ and $\LL(\G(\Om),\wt{\C})$ which have a ``basic'' structure. In detail, we say that $T\in\LL(\Gc(\Om),\wt{\C})$ is ${\rfunc}$ if there exists a net $(T_\eps)_\eps\in\D'(\Om)^{(0,1]}$ fulfilling the following condition: for all $K\Subset\Om$ there exist $j\in\N$, $c>0$, $N\in\N$ and $\eta\in(0,1]$ such that
\[
\forall f\in\D_K(\Om)\, \forall\eps\in(0,\eta]\qquad\quad
|T_\eps(f)|\le c\eps^{-N}\sup_{x\in K,|\alpha|\le j}|\partial^\alpha f(x)|
\]
and $Tu=[(T_\eps u_\eps)_\eps]$ for all $u\in\Gc(\Om)$.\\
In the same way a functional $T\in\LL(\G(\Om),\wt{\C})$ is said to be $\rfunc$ if there exists a net  $(T_\eps)_\eps\in\E'(\Om)^{(0,1]}$ such that there exist $K\Subset\Om$, $j\in\N$, $c>0$, $N\in\N$ and $\eta\in(0,1]$ with the property 
\[
\forall f\in\Cinf(\Om)\, \forall\eps\in(0,\eta]\qquad\quad
|T_\eps(f)|\le c\eps^{-N}\sup_{x\in K,|\alpha|\le j}|\partial^\alpha f(x)|
\]
and $Tu=[(T_\eps u_\eps)_\eps]$ for all $u\in\G(\Om)$.\\
Clearly the sets $\Lb(\Gc(\Om),\wt{\C})$ and $\Lb(\G(\Om),\wt{\C})$ of $\rfunc$ functionals are $\wt{\C}$-linear subspaces of $\LL(\Gc(\Om),\wt{\C})$ and $\LL(\G(\Om),\wt{\C})$ respectively. In addition if $T$ is a $\rfunc$ functional in $\LL(\Gc(\Om),\wt{\C})$ and $u\in\Gc(\Om)$ then $uT\in\LL(\G(\Om),\wt{\C})$ is $\rfunc$. We recall that nets $(T_\eps)_\eps$ which define basic maps as above were already considered in \cite{Delcroix:05,DelSca:00} with slightly more general notions of moderateness and different choices of notations and language.


\subsection{Generalized symbols}
\label{subsec_gen_symb}
For the convenience of the reader we recall a few basic notions concerning the sets of symbols employed in the course of this work. More details can be found in \cite{GGO:03, GH:05} where a theory of generalized pseudodifferential operators acting on Colombeau algebras is developed.
\subsubsection*{Definitions.}
Let $\Om$ be an open subset of $\R^n$, $m\in\R$ and $\rho,\delta\in[0,1]$. $S^m_{\rho,\delta}(\Om\times\R^p)$ denotes the set of symbols of order $m$ and type $(\rho,\delta)$ as introduced by H\"ormander in \cite{Hoermander:71}. The subscript $(\rho,\delta)$ is omitted when $\rho=1$ and $\delta=0$. If $V$ is an open conic set of $\Om\times\R^{p}$ we define $S^m_{\rho,\delta}(V)$ as the set of all $a\in\Cinf(V)$ such that for all $K\Subset V$, 
\[
\sup_{(x,\xi)\in K^{c}}\lara{\xi}^{-m+\rho|\alpha|-\delta|\beta|}|\partial^\alpha_\xi\partial^\beta_x a(x,\xi)|<\infty,
\]
where $K^{c}:=\{(x,t\xi):\, (x,\xi)\in K\ t\ge 1\}$.
We also make use of the space $S^1_{\rm{hg}}(\Om\times\R^p\setminus 0)$ of all $a\in S^1(\Om\times\R^p\setminus 0)$ homogeneous of degree $1$ in $\xi$. Note that the assumption of homogeneity allows to state the defining conditions above in terms of the seminorms
\[
\sup_{x\in K,\xi\in\R^p\setminus 0}|\xi|^{-1+\alpha}|\partial^\alpha_\xi\partial^\beta_x a(x,\xi)|
\]
where $K$ is any compact subset of $\Om$.

The space of \emph{generalized symbols} $\wt{\mathcal{S}}^m_{\rho,\delta}(\Om\times\R^p)$ is the $\wt{\C}$-module of $\G_E$-type obtained by taking $E=S^{m}_{\rho,\delta}(\Om\times\R^p)$ equipped with the family of seminorms
\[
|a|^{(m)}_{\rho,\delta,K,j}=\sup_{x\in K,\xi\in\R^n}\sup_{|\alpha+\beta|\le j}|\partial^\alpha_\xi\partial^\beta_x a(x,\xi)|\lara{\xi}^{-m+\rho|\alpha|-\delta|\beta|},\qquad\quad K\Subset\Om,\, j\in\N.
\]
The valuation corresponding to $|\cdot|^{(m)}_{\rho,\delta,K,j}$ gives the ultra-pseudo-seminorm $\mP^{(m)}_{\rho,\delta,K,j}$. $\wt{\mathcal{S}}^m_{\rho,\delta}(\Om\times\R^p)$ topologized through the family $\{\mP^{(m)}_{\rho,\delta,K,j}\}_{K\Subset\Om,j\in\N}$ of ultra-pseudo-seminorms is a Fr\'echet $\wt{\C}$-module. In analogy with $\wt{\mathcal{S}}^m_{\rho,\delta}(\Om\times\R^p)$ we use the notation $\wt{\mathcal{S}}^m_{\rho,\delta}(V)$ for the $\wt{\C}$-module $\G_{S^m_{\rho,\delta}(V)}$.

$\wt{\mathcal{S}}^m_{\rho,\delta}(\Om_x\times\R^p_\xi)$ has the structure of a sheaf with respect to $\Om$. So it is meaningful to talk of the support with respect to $x$ of a generalized symbol $a$ ($\supp_x\, a$).\\  
We define the \emph{conic support} of $a\in\wt{\mathcal{S}}^m_{\rho,\delta}(\Om\times\R^p)$ (${\rm{cone\, supp}}\, a$) as the complement of the set of points $(x_0,\xi_0)\in\Om\times\R^p$ such that there exists a relatively compact open neighborhood $U$ of $x_0$, a conic open neighborhood $\Gamma$ of $\xi_0$ and a representative $(a_\eps)_\eps$ of $a$ satisfying the condition
\begin{equation}
\label{cond_conic_supp}
\forall\alpha\in\N^p\, \forall\beta\in\N^n\, \forall q\in\N\quad \sup_{x\in U,\xi\in\Gamma}\lara{\xi}^{-m+\rho|\alpha|-\delta|\beta|}|\partial^\alpha_\xi\partial^\beta_x a_\eps(x,\xi)|=O(\eps^q)\, \text{as $\eps\to 0$}.
\end{equation}
By definition ${\rm{cone\, supp}}\, a$ is a closed conic subset of $\Om\times\R^p$. The generalized symbol $a$ is $0$ on $\Om\setminus\pi_x(\rm{cone\, supp}\, a)$.
\subsubsection*{Slow scale symbols.}
In the paper the classes of the factor space $\G^{\,\ssc}_{S^m_{\rho,\delta}(\Om\times\R^p)}$ are called \emph{generalized symbols of slow scale type}. For simplicity we introduce the notation $\wt{S}^{\,m,\ssc}_{\rho,\delta}(\Om\times\R^p)$. Substituting $S^m_{\rho,\delta}(\Om\times\R^p)$ with $S^m_{\rho,\delta}(V)$ we obtain the set $\wt{S}^{m,\ssc}_{\rho,\delta}(V):=\G^{\,\ssc}_{S^m_{\rho,\delta}(V)}$ of slow scale symbols on the open set $V\subseteq\Om\times(\R^p\setminus 0)$. 

\subsubsection*{Generalized symbols of order $-\infty$.}
Different notions of regularity are related to the sets $\wt{\mathcal{S}}^{-\infty}(\Om\times\R^p)$ and $\wt{\mathcal{S}}^{-\infty,\ssc}(\Om\times\R^p)$ of generalized symbols of order $-\infty$.\\ The space $\wt{\mathcal{S}}^{-\infty}(\Om\times\R^p)$ of generalized symbols of order $-\infty$ is defined as the $\wt{\C}$-module $\G_{S^{-\infty}(\Om\times\R^p)}$. Its elements are equivalence classes $a$ whose representatives $(a_\eps)_\eps$ have the property $|a_\eps|^{(m)}_{K,j}=O(\eps^{-N})$ as $\eps\to 0$, where $N$ depends on the order $m$ of the symbol, on the order $j$ of the derivatives and on the compact set $K\subseteq\Om$.\\ $\wt{\mathcal{S}}^{-\infty,\ssc}(\Om\times\R^p)$ is defined by substituting $O(\eps^{-N})$ with $O(\lambda_\eps)$ in the previous estimate, where $(\lambda_\eps)_\eps$ is a slow scale net depending as above on the order $m$ of the symbol, on the order $j$ of the derivatives and on the compact set $K\subseteq\Om$. It follows that $(a_\eps)_\eps$ is $\Ginf$-regular, in the sense that 
\[
|a_\eps|^{(m)}_{K,j}=O(\eps^{-1})
\]
as $\eps\to 0$ for all $m,j$ and $K\Subset\Om$.
 
\subsubsection*{Generalized microsupports.} The $\G$- and $\Ginf$-regularity of generalized symbols on $\Om\times\R^n$ is measured in conical neighborhoods by means of the following notions of microsupports.

Let $a\in\wt{\mathcal{S}}^l_{\rho,\delta}(\Om\times\R^n)$ and $(x_0,\xi_0)\in\CO{\Om}$. The symbol $a$ is $\G$-smoothing at $(x_0,\xi_0)$ if there exist a representative $(a_\eps)_\eps$ of $a$, a relatively compact open neighborhood $U$ of $x_0$ and a conic neighborhood $\Gamma\subseteq\R^n\setminus 0$ of $\xi_0$ such that
\begin{multline}
\label{est_micro_G}
\forall m\in\R\, \forall\alpha,\beta\in\N^n\, \exists N\in\N\, \exists c>0\, \exists\eta\in(0,1]\, \forall(x,\xi)\in U\times\Gamma\, \forall\eps\in(0,\eta]\\
|\partial^\alpha_\xi\partial^\beta_x a_\eps(x,\xi)|\le c\lara{\xi}^m\eps^{-N}.
\end{multline}
The symbol $a$ is $\Ginf$-smoothing at $(x_0,\xi_0)$ if there exist a representative $(a_\eps)_\eps$ of $a$, a relatively compact open neighborhood $U$ of $x_0$, a conic neighborhood $\Gamma\subseteq\R^n\setminus 0$ of $\xi_0$ and a natural number $N\in\N$ such that  
\begin{multline}
\label{est_micro_Ginf}
\forall m\in\R\, \forall\alpha,\beta\in\N^n\, \exists c>0\, \exists\eta\in(0,1]\, \forall(x,\xi)\in U\times\Gamma\, \forall\eps\in(0,\eta]\\
|\partial^\alpha_\xi\partial^\beta_x a_\eps(x,\xi)|\le c\lara{\xi}^m\eps^{-N}.
\end{multline}
We define the \emph{$\G$-microsupport} of $a$, denoted by $\mu\, \supp_\G(a)$, as the complement in $\CO{\Om}$ of the set of points $(x_0,\xi_0)$ where $a$ is $\G$-smoothing and the \emph{$\Ginf$-microsupport} of $a$, denoted by $\mu\, \supp_{\Ginf}(a)$, as the complement in $\CO{\Om}$ of the set of points $(x_0,\xi_0)$ where $a$ is $\Ginf$-smoothing. 
\subsubsection*{Continuity results.}
By simple reasoning at the level of representatives one proves that the usual operations between generalized symbols, as product and derivation, are continuous. In particular the $\wt{\C}$-bilinear map
\begin{equation}
\label{bil_product}
\Gc(\Om)\times\wt{\mathcal{S}}^m_{\rho,\delta}(\Om\times\R^p)\to\wt{\mathcal{S}}^m_{\rho,\delta}(\Om\times\R^p):(u,a)\to a(y,\xi)u(y)
\end{equation}
is continuous. If $l<-p$ each $b\in\wt{\mathcal{S}}^l_{\rho,\delta}(\Om\times\R^p)$ can be integrated on $K\times\R^p$, $K\Subset\Om$, by setting
\[
\int_{K\times\R^p}b(y,\xi)\, dy\, d\xi :=\biggl[\biggl(\int_{K\times\R^p}b_\eps(y,\xi)\, dy\, d\xi\biggr)_\eps\biggr].
\]
Moreover if $\supp_y b\Subset\Om$ we define the integral of $b$ on $\Om\times\R^p$ as
\[
\int_{\Om\times\R^p}b(y,\xi)\, dy\, d\xi :=\int_{K\times\R^p}b(y,\xi)\, dy\, d\xi,
\]
where $K$ is any compact set containing $\supp_y b$ in its interior. Integration defines a continuous $\wt{\C}$-linear functional on this space of generalized symbols with compact support in $y$ as it is proven in \cite[Proposition 1.1, Remark 1.2]{GHO:06}.
 
\subsection{Asymptotic expansions in $\wt{\mathcal{S}}^m_{\rho,\delta}(\Om\times\R^p)$ and $\wt{\mathcal{S}}^{m,\ssc}_{\rho,\delta}(\Om\times\R^p)$}
\label{asymp_new}
In this subsection we elaborate and state in full generality the notion of asymptotic expansion of a generalized symbol introduced for the first time in \cite{GGO:03}. We also provide a technical result which will be useful in Section \ref{section_comp}. We begin by working on moderate nets of symbols and we recall that a net $(C_\eps)_\eps\in\C^{(0,1]}$ is said to be of slow scale type if there exists a slow scale net $(\omega_\eps)_\eps$ such that $|C_\eps|=O(\omega_\eps)$.
\begin{defn}
\label{asymp_mod}
Let $\{m_j\}_{j\in\mathbb{N}}$ be sequences of real numbers with
$m_j\searrow -\infty$, $m_0=m$. 
\begin{itemize}
\item[(i)] Let $\{(a_{j,\epsilon})_\epsilon\}_{j\in\mathbb{N}}$
be a sequence of elements $(a_{j,\epsilon})_\epsilon\in\mM_{S^{m_j}_{\rho,\delta}(\Om\times\R^p)}$. 
We say that the formal series $\sum_{j=0}^\infty(a_{j,\epsilon})_\epsilon$ is the asymptotic expansion of
$(a_\epsilon)_\epsilon\in\mathcal{E}[\Omega\times\mathbb{R}^n]$,
$(a_\epsilon)_\epsilon\sim\sum_j(a_{j,\epsilon})_\epsilon$ for short, iff for all $r\ge 1$
\[
\biggl(a_\epsilon-\sum_{j=0}^{r-1}a_{j,\epsilon}\biggr)_\epsilon\in \mM_{S^{m_r}_{\rho,\delta}(\Omega\times\mathbb{R}^p)}. 
\]
\item[(ii)] Let $\{(a_{j,\epsilon})_\epsilon\}_{j\in\mathbb{N}}$
be a sequence of elements $(a_{j,\epsilon})_\epsilon\in\mM^\ssc_{S^{m_j}_{\rho,\delta}(\Om\times\R^p)}$. 
We say that the formal series $\sum_{j=0}^\infty(a_{j,\epsilon})_\epsilon$ is the asymptotic expansion of
$(a_\epsilon)_\epsilon\in\mathcal{E}[\Omega\times\mathbb{R}^n]$,
$(a_\epsilon)_\epsilon\sim_\ssc\sum_j(a_{j,\epsilon})_\epsilon$ for short, iff for all $r\ge 1$
\[
\biggl(a_\epsilon-\sum_{j=0}^{r-1}a_{j,\epsilon}\biggr)_\epsilon\in \mM^\ssc_{S^{m_r}_{\rho,\delta}(\Omega\times\mathbb{R}^p)}. 
\]
\end{itemize}
\end{defn}
\begin{thm}
\label{theo_asymp_expan}
\leavevmode
\begin{itemize}
\item[(i)] Let $\{(a_{j,\epsilon})_\epsilon\}_{j\in\mathbb{N}}$
be a sequence of elements $(a_{j,\epsilon})_\epsilon\in\mM_{S^{m_j}_{\rho,\delta}(\Om\times\R^p)}$ with $m_j\searrow -\infty$ and $m_0=m$. Then, there exists $(a_\eps)_\eps\in\mM_{S^{m}_{\rho,\delta}(\Om\times\R^p)}$ such that $(a_\epsilon)_\epsilon\sim\sum_j(a_{j,\epsilon})_\epsilon$. Moreover, if $(a'_\epsilon)_\epsilon\sim\sum_j(a_{j,\epsilon})_\epsilon$ then $(a_\eps-a'_\eps)_\eps\in\mM_{S^{-\infty}(\Om\times\R^p)}$.
\item[(ii)] Let $\{(a_{j,\epsilon})_\epsilon\}_{j\in\mathbb{N}}$
be a sequence of elements $(a_{j,\epsilon})_\epsilon\in\mM^{\ssc}_{S^{m_j}_{\rho,\delta}(\Om\times\R^p)}$ with $m_j\searrow -\infty$ and $m_0=m$. Then, there exists $(a_\eps)_\eps\in\mM^\ssc_{S^{m}_{\rho,\delta}(\Om\times\R^p)}$ such that $(a_\epsilon)_\epsilon\sim_\ssc\sum_j(a_{j,\epsilon})_\epsilon$. Moreover, if $(a'_\epsilon)_\epsilon\sim_\ssc\sum_j(a_{j,\epsilon})_\epsilon$ then $(a_\eps-a'_\eps)_\eps\in\mM^\ssc_{S^{-\infty}(\Om\times\R^p)}$.
\end{itemize}
\end{thm}
\begin{proof}
The proof follows the classical line of arguments, but we will have to keep track of the
$\epsilon$-dependence carefully.
We consider a sequence of relatively compact open sets $\{V_l\}$ contained in $\Omega$, such that for all $l\in\mathbb{N}$, $V_l\subset K_l=\overline{V_l}\subset V_{l+1}$ and $\bigcup_{l\in\mathbb{N}}V_l=\Omega$. Let $\psi\in\mathcal{C}^\infty(\mathbb{R}^p)$, $0\le\psi(\xi)\le 1$, such that $\psi(\xi)=0$ for $|\xi|\le 1$ and $\psi(\xi)=1$ for $|\xi|\ge 2$.\\
$(i)$ We introduce
\[
b_{j,\epsilon}(x,\xi)=\psi(\lambda_{j,\eps}\xi)a_{j,\epsilon}(x,\xi),
\]
where $\lambda_{j,\eps}$ will be positive constants with $\lambda_{j+1,\eps}<\lambda_{j,\eps}<1$, $\lambda_{j,\eps}\to 0$ if $j\to\infty$. We can define
\begin{equation}
\label{a_symb}
a_\epsilon(x,\xi)=\sum_{j\in\mathbb{N}}b_{j,\epsilon}(x,\xi).
\end{equation}
This sum is locally finite and therefore $(a_\eps)_\eps\in \E[\Om\times\R^p]$. We observe that
$\partial^\alpha(\psi(\lambda_{j,\eps}\xi))=\partial^\alpha\psi(\lambda_{j,\eps}\xi)\lambda_{j,\eps}^{|\alpha|}$,
${\rm supp}\,(\partial^\alpha\psi(\lambda_{j,\eps}\xi))\subseteq\{\xi:\ 1/\lambda_{j,\eps}\le|\xi|\le 2/\lambda_{j,\eps}\}$, and that $1/\lambda_{j,\eps}\le|\xi|\le 2/\lambda_{j,\eps}$ implies $\lambda_{j,\eps}\le 2/|\xi|\le 4/(1+|\xi|)$. We first estimate $b_{j,\eps}$. Fixing $K\Subset\Omega$ and
$\alpha\in\N^p$, $\beta\in\mathbb{N}^n$, we obtain for $j\in\mathbb{N}$, $\epsilon\in(0,1]$, $x\in K$, $\xi\in\mathbb{R}^p$,
\begin{multline*}
|\partial^\alpha_\xi\partial^\beta_x b_{j,\epsilon}(x,\xi)|\le\sum_{\gamma\le\alpha}\binom{\alpha}{\gamma}\lambda_{j,\eps}^{|\alpha-\gamma|}|\partial^{\alpha-\gamma}\psi(\lambda_{j,\eps}\xi)||a_{j,\eps}|^{(m_j)}_{\rho,\delta,K,\gamma,\beta}\langle\xi\rangle^{m_j-\rho|\gamma|+\delta|\beta|}\\
\le\sum_{\gamma\le\alpha}c(\psi,\gamma)4^{|\alpha-\gamma|}\langle\xi\rangle^{-|\alpha-\gamma|}|a_{j,\eps}|^{(m_j)}_{\rho,\delta,K,\gamma,\beta}\langle\xi\rangle^{m_j-\rho|\gamma|+\delta|\beta|}\\
\le C_{j,\alpha,\beta,K,\eps}\langle\xi\rangle^{m_j-\rho|\alpha|+\delta|\beta|},
\end{multline*}
where
\[
C_{j,\alpha,\beta,K,\eps}:=\sum_{\gamma\le\alpha}c(\psi,\gamma)4^{|\alpha-\gamma|}|a_{j,\eps}|^{(m_j)}_{\rho,\delta,K,\gamma,\beta}.
\]
Since $(C_{j,\alpha,\beta,K,\eps})_\eps$ is a moderate net of positive numbers, we have that $(b_{j,\eps})_\eps\in\mM_{S^{m_j}_{\rho,\delta}(\Om\times\R^p)}$. At this point we choose $\lambda_{j,\eps}$ such that for $|\alpha+\beta|\le j$, $l\le j$
\begin{equation}
\label{cj}
C_{j,\alpha,\beta,K_l,\eps}\lambda_{j,\eps}\le 2^{-j}.
\end{equation}
Our aim is to prove that $a_\epsilon(x,\xi)$ defined in \eqref{a_symb} belongs to $\mM_{S^{m}_{\rho,\delta}(\Om\times\R^p)}$. Since there exists $N_j\in N$ and $\eta_j\in(0,1]$ such that 
\[
C_{j,\alpha,\beta,K_l,\eps}\le \eps^{-N_j}
\]
for $l\le j$ and $|\alpha+\beta|\le j$, we take $\lambda_{j,\eps}=2^{-j}\eps^{N_j}$ on the interval $(0,\eta_j]$.
We observe that
\begin{equation}
\label{lo}
\begin{array}{cc}
\forall K\Subset\Omega,\ \exists l\in\mathbb{N}:\ K\subset V_l\subset K_l,\\[0.2cm]
\forall\alpha_0\in\N^p,\ \forall\beta_0\in\mathbb{N}^n,\ \exists j_0\in\mathbb{N},\ j_0\ge l:\ |\alpha_0|+|\beta_0|\le j_0,\quad m_{j_0}+1\le m,
\end{array}
\end{equation}
and we write $(a_\epsilon)_\epsilon$ as the sum of the following two terms:
\[
\sum_{j=0}^{j_0-1}b_{j,\epsilon}(x,\xi)+\sum_{j=j_0}^{+\infty}b_{j,\epsilon}(x,\xi)=f_\epsilon(x,\xi)+s_\epsilon(x,\xi).
\]
For $x\in K$, we have that
\begin{multline*}
|\partial^{\alpha_0}_\xi\partial^{\beta_0}_x f_\epsilon(x,\xi)|\le \sum_{j=0}^{j_0-1}|b_{j,\eps}|^{(m_j)}_{\rho,\delta,K,\alpha_0,\beta_0}\langle\xi\rangle^{m_j-\rho|\alpha_0|+\delta|\beta_0|}\\
\le \biggl(\sum_{j=0}^{j_0-1}|b_{j,\eps}|^{(m_j)}_{\rho,\delta,K,\alpha_0,\beta_0}\biggr)\lara{\xi}^{m-\rho|\alpha_0|+\delta|\beta_0|}, 
\end{multline*}
where 
\[
\biggl(\sum_{j=0}^{j_0-1}|b_{j,\eps}|^{(m_j)}_{\rho,\delta,K,\alpha_0,\beta_0}\biggr)_\eps\in\EM.
\]
We now turn to $s_\epsilon(x,\xi)$. From the estimates on $b_{j,\eps}$ and \eqref{cj}, we get for $x\in K$ and $\epsilon\in(0,1]$,
\begin{multline*}
|\partial^{\alpha_0}_\xi\partial^{\beta_0}_x s_\epsilon(x,\xi)|\le\sum_{j=j_0}^{+\infty}C_{j,\alpha_0,\beta_0,K_l}\langle\xi\rangle^{m_j-\rho|\alpha_0|+\delta|\beta_0|}\\
\le \sum_{j=j_0}^{+\infty}2^{-j}\lambda_{j,\eps}^{-1}\langle\xi\rangle^{-1}\langle\xi\rangle^{ m_j+1-\rho|\alpha_0|+\delta|\beta_0|}\le \sum_{j=j_0}^{+\infty}2^{-j}\lambda_{j,\eps}^{-1}\langle\xi\rangle^{-1}\langle\xi\rangle^{ m-\rho|\alpha_0|+\delta|\beta_0|}.
\end{multline*}
Since $\psi(\xi)$ is identically equal to $0$ for $|\xi|\le 1$, we can assume in our estimates
that $\langle\xi\rangle^{-1}\le\lambda_{j,\eps}$, and therefore from \eqref{lo}, we conclude that
\[
|\partial^{\alpha_0}_\xi\partial^{\beta_0}_x s_\epsilon(x,\xi)|\le 2\langle\xi\rangle^{m-\rho|\alpha_0|+\delta|\beta_0|}, 
\]
for all $x\in K$, $\xi\in\R^p$ and $\eps\in(0,1]$.

In order to prove that $(a_\epsilon)_\epsilon\sim\sum_j(a_{j,\epsilon})_\epsilon$ we fix $r\ge 1$ and we write
\begin{multline*}
a_\epsilon(x,\xi)-\sum_{j=0}^{r-1}a_{j,\epsilon}(x,\xi)
=\sum_{j=0}^{r-1}(\psi(\lambda_{j,\eps}\xi)-1)a_{j,\epsilon}(x,\xi)+\sum_{j=r}^{+\infty}\psi(\lambda_{j,\eps}\xi)a_{j,\epsilon}(x,\xi)\\
=g_\epsilon(x,\xi)+t_\epsilon(x,\xi).
\end{multline*}
Recall that $\psi\in\mathcal{C}^\infty(\mathbb{R}^p)$ was chosen such that
$\psi-1\in\mathcal{C}^\infty_c(\mathbb{R}^p)$ and
${\rm{supp}}(\psi-1)\subseteq\{\xi:\ |\xi|\le 2\}$. Thus, for $0\le j\le r-1$,
\[
\text{supp}(\psi(\lambda_{j,\eps}\xi)-1)\subseteq\{\xi:\ |\lambda_{j,\eps}\xi|\le 2\}\subseteq \{\xi:\ |\xi|\le 2\lambda_{r-1,\eps}^{-1}\}.
\]
As a consequence, for fixed $K\Subset\Omega$ and for all $\epsilon\in(0,1]$,
\begin{multline*}
|\partial^\alpha_\xi\partial^\beta_x g_\epsilon(x,\xi)|\le\sum_{j=0}^{r-1}\sum_{\alpha'\le\alpha}\binom{\alpha}{\alpha'}\lambda_{j,\eps}^{|\alpha'|}c(\psi,\alpha')\lara{2\lambda_{r-1,\eps}^{-1}}^{m_j-m_r+\rho|\alpha'|}|a_{j,\eps}|^{(m_j)}_{\rho,\delta,K,\alpha-\alpha',\beta}\\
\cdot\lara{\xi}^{m_r-\rho|\alpha|+\delta|\beta|},\\
\end{multline*}
where, from our assumptions on $(a_{j,\eps})_\eps$ and $(\lambda_{j,\eps})_\eps$, the nets $(|a_{j,\eps}|^{(m_j)}_{\rho,\delta,K,\alpha-\alpha',\beta})_\eps$ and $(\lara{2\lambda_{r-1,\eps}^{-1}}^{m_j-m_r+\rho|\alpha'|})_\eps$ are both moderate. Repeating the same arguments used in the construction
of $(a_\epsilon)_\epsilon$ we have that $(t_\epsilon)_\epsilon$ belongs to $\mM_{S^{m_r}_{\rho,\delta}(\Om\times\R^p)}$. It is clear that $(a_\eps)_\eps$ is uniquely determined by $\sum_j (a_{j,\eps})_\eps$ modulo $\mM_{S^{-\infty}(\Om\times\R^p)}$.\\
$(ii)$ In the slow scale case one easily sees that $(b_{j,\eps})_\eps\in \mM^\ssc_{S^{m_j}_{\rho,\delta}(\Om\times\R^p)}$. Moreover, since there exists a slow scale net $\omega_j(\eps)$ and $\eta_j\in(0,1]$ such that 
\[
C_{j,\alpha,\beta,K_l,\eps}\le \omega_j(\eps)
\]
for $l\le j$ and $|\alpha+\beta|\le j$, we can take $\lambda_{j,\eps}=2^{-j}\omega_j^{-1}(\eps)$ on the interval $(0,\eta_j]$. It follows that $(a_\eps)_\eps\in\mM^\ssc_{S^m_{\rho,\delta}(\Om\times\R^p)}$ and that both the nets $(g_\eps)_\eps$ and $(t_\eps)_\eps$ belong to $\mM^\ssc_{S^{m_r}_{\rho,\delta}(\Om\times\R^p)}$. 
\end{proof}
\begin{prop}
\label{prop_asym_Shubin}
\leavevmode
\begin{itemize}
\item[(i)] Let $\{(a_{j,\epsilon})_\epsilon\}_{j\in\mathbb{N}}$
be a sequence of elements $(a_{j,\epsilon})_\epsilon\in\mM_{S^{m_j}_{\rho,\delta}(\Om\times\R^p)}$ with $m_j\searrow -\infty$ and $m_0=m$. Let $(a_\eps)_\eps\in\E[\Om\times\R^p]$ such that for all $K\Subset\Om$, for all $\alpha,\beta$ there exists $\mu\in\R$ and $(C_\eps)_\eps\in \EM$ such that 
\begin{equation}
\label{est_1_Sh}
|\partial^\alpha_\xi\partial^\beta_x a_\eps(x,\xi)|\le C_\eps\lara{\xi}^\mu,
\end{equation}
for all $x\in K$, $\xi\in\R^p$, $\eps\in(0,1]$. Furthermore, assume that for any $r\ge 1$ and $K\Subset\Om$ there exists $\mu_r=\mu_r(K)$ and $(C_{r,\eps})_\eps=(C_{r,\eps}(K))_\eps\in\EM$ such that $\mu_r\to-\infty$ as $r\to +\infty$ and
\begin{equation}
\label{est_2_Sh}
\biggl|a_\eps(x,\xi)-\sum_{j=0}^{r-1}a_{j,\eps}(x,\xi)\biggr|\le C_{r,\eps}\lara{\xi}^{\mu_r}
\end{equation}
for all $x\in K$, $\xi\in\R^p$, $\eps\in(0,1]$. Then, $(a_\eps)_\eps\sim\sum_j(a_{j,\epsilon})_\epsilon$. 
\item[(ii)] $(i)$ holds with $(a_{j,\epsilon})_\epsilon\in\mM^\ssc_{S^{m_j}_{\rho,\delta}(\Om\times\R^p)}$, the nets $(C_\eps)_\eps$ and $(C_{r,\eps})_\eps$ of slow scale type and $(a_\eps)_\eps\sim_\ssc\sum_j(a_{j,\epsilon})_\epsilon$ in the sense of Definition \ref{asymp_mod}$(ii)$.
\end{itemize}
\end{prop}
The proof of Proposition \ref{prop_asym_Shubin} requires the following lemma.
\begin{lem}
\label{lemma_Shubin}
Let $K_1$ and $K_2$ be two compact sets in $\R^p$ such that $K_1\subset{\rm{Int}}\, K_2$. Then there exists a constant $C>0$ such that for any smooth function $f$ on a neighborhood of $K_2$, the following estimate holds:
\[
\biggl(\sup_{x\in K_1}\sum_{|\alpha|= 1}|D^\alpha f(x)|\biggr)^2\le C\sup_{x\in K_2}|f(x)|\biggl(\sup_{x\in K_2}|f(x)|+\sup_{x\in K_2}\sum_{|\alpha|=2}|D^\alpha f(x)|\biggr).
\]
\end{lem}
\begin{proof}[Proof of Proposition \ref{prop_asym_Shubin}]
$(i)$ By Theorem \ref{theo_asymp_expan} we know that there exists $(b_\eps)_\eps\in\mM_{S^{m}_{\rho,\delta}(\Om\times\R^p)}$ such that $(b_\eps)_\eps\sim\sum_j(a_{j,\eps})_\eps$.We consider the difference $d_\eps=a_\eps-b_\eps$. From \eqref{est_1_Sh} and the moderateness of $(b_\eps)_\eps$ we have that for all $\alpha,\beta$ and $K\Subset\Om$ there exist $(C'_\eps)_\eps$ and $\mu'$ such that 
\begin{equation}
\label{formula_raf_1}
|\partial^\alpha_\xi\partial^\beta_x d_\eps(x,\xi)|\le C'_\eps\lara{\xi}^{\mu'},
\end{equation}
for all $x\in K$, $\xi\in\R^p$ and $\eps\in(0,1]$. Combining $(b_\eps)_\eps\sim\sum_j(a_{j,\eps})_\eps$ with \eqref{est_2_Sh} we obtain that for all $r>0$ and for all $K\Subset\Om$ there exists $(C_{r,\eps}(K))_\eps\in\EM$ such that 
\[
|d_\eps(x,\xi)|\le C_{r,\eps}(K)\lara{\xi}^{-r},\qquad\qquad x\in K,\ \xi\in\R^p,\, \eps\in(0,1].
\]
Set $d_{\xi,\eps}(x,\theta)=d_\eps(x,\xi+\theta)$. Then, $\partial^\alpha_\theta \partial^\beta_x d_{\xi,\eps}(x,\theta)|_{\theta=0}=\partial^\alpha_\xi\partial^\beta_x d(x,\xi)$, and applying Lemma \ref{lemma_Shubin} with $K_1=K\times 0$ and $K_2=K'\times\{|\theta|\le 1\}$, where $K\subset{\rm{Int}}K'\subset K'\Subset\Om$, we obtain
\begin{multline}
\label{formula_raf_2}
\biggl(\sup_{x\in K}\sum_{|\alpha+\beta|= 1}|\partial^\alpha_\xi\partial^\beta_x d_\eps(x,\xi)|\biggr)^2\le C\sup_{x\in K', |\theta|\le 1}|d_\eps(x,\xi+\theta)|\cdot\\
\cdot\biggl(\sup_{x\in K', |\theta|\le 1}|d_\eps(x,\xi+\theta)|+\sup_{x\in K', |\theta|\le 1}\sum_{|\alpha+\beta|=2}|\partial^\alpha_\xi\partial^\beta_x d_\eps(x,\xi+\theta)|\biggr)\\
\le CC_{r,\eps}(K')\sup_{|\theta|\le 1}\lara{\xi+\theta}^{-r}\biggl(C_{r,\eps}(K')\sup_{|\theta|\le 1}\lara{\xi+\theta}^{-r}+C'_\eps(K')\sup_{|\theta|\le 1}\lara{\xi+\theta}^{\mu'(K,2)}\biggr)\\
\le C''_{r,\eps}(K)\lara{\xi}^{-r},
\end{multline}
where $C''_{r,\eps}(K)\in\EM$. By induction one can prove that for all $r>0$, for all $K\Subset\Om$ and for all $\alpha\in\N^p$, $\beta\in\N^n$, there exists a moderate net $(c_\eps)_\eps$ such that the estimate
\[
|\partial^\alpha_\xi\partial^\beta_x d_\eps(x,\xi)|\le c_\eps\lara{\xi}^{-r}
\]
is valid for all $x\in K$, $\xi\in\R^p$ and $\eps\in(0,1]$. This means that $(d_\eps)_\eps\in\mM_{S^{-\infty}(\Om\times\R^p)}$ and therefore $(a_\eps)_\eps\sim \sum_j(a_{j,\eps})_\eps$.\\
$(ii)$ It is clear that when we work with nets of slow scale type then $(d_\eps)_\eps\in\mM^\ssc_{S^{-\infty}(\Om\times\R^p)}$ and $(a_\eps)_\eps\sim_\ssc\sum_j(a_{j,\epsilon})_\epsilon$. 
\end{proof}
\begin{rem}
Proposition \ref{prop_asym_Shubin} can be stated for nets of symbols in $\mM_{S^m_{\rho,\delta}(\Om\times\R^p\setminus 0)}$ and $\mM^\ssc_{S^m_{\rho,\delta}(\Om\times\R^p\setminus 0)}$. The proof make use of \eqref{formula_raf_1} when $|\xi|\le 1$ and \eqref{formula_raf_2} when $|\xi|>1$.
\end{rem}
\begin{defn}
\label{def_asymp_gen}
Let $\{m_j\}_{j\in\mathbb{N}}$ with $m_j\searrow -\infty$ and $m_0=m$. 
\begin{itemize}
\item[(i)]
Let $\{a_j\}_{j\in\mathbb{N}}$ be a sequence
of symbols $a_j\in{\widetilde{\mathcal{S}}}^{\, m_j}_{\rho,\delta}
(\Omega\times\mathbb{R}^p)$. We say that the formal series $\sum_j a_j$ is the asymptotic expansion of $a\in\wt{\mathcal{S}}^m_{\rho,\delta}(\Om\times\R^p)$, $a\sim\sum_j a_j$ for short, iff there exist a representative $(a_\epsilon)_\epsilon$ of $a$ and, for every $j$, representatives $(a_{j,\epsilon})_\epsilon$ of $a_j$, such that $(a_\epsilon)_\epsilon\sim\sum_j(a_{j,\epsilon})_\epsilon$.
\item[(ii)] Let $\{a_j\}_{j\in\mathbb{N}}$ be a sequence
of symbols $a_j\in{\widetilde{\mathcal{S}}}^{\, m_j,\ssc}_{\rho,\delta}
(\Omega\times\mathbb{R}^p)$. We say that the formal series $\sum_j a_j$ is the asymptotic expansion of $a\in\wt{\mathcal{S}}^{m,\ssc}_{\rho,\delta}(\Om\times\R^p)$, $a\sim\sum_j a_j$ for short, iff there exist a representative $(a_\epsilon)_\epsilon$ of $a$ and, for every $j$, representatives $(a_{j,\epsilon})_\epsilon$ of $a_j$, such that $(a_\epsilon)_\epsilon\sim_\ssc\sum_j(a_{j,\epsilon})_\epsilon$.
\end{itemize}
\end{defn}
\subsection{Generalized pseudodifferential operators}
Let $\Om$ be an open subset of $\R^n$ and $a\in \widetilde{\mathcal{S}}^m_{\rho,\delta}(\Om\times\R^n)$. The generalized oscillatory integral (see \cite{GGO:03})
\[
\int_{\Om\times\R^n}\esp^{i(x-y)\xi}a(x,\xi){u}(y)\, dy\, \dslash\xi :=\biggl(\int_{\Om\times\R^n}\esp^{i(x-y)\xi}a_\eps(x,\xi){u_\eps}(y)\, dy\, \dslash\xi\biggr)_\eps+\Neg(\Om),
\] 
defines the action of the pseudodifferential operator $a(x,D)$ with generalized symbol $a\in \widetilde{\mathcal{S}}^m_{\rho,\delta}(\Om\times\R^n)$ on $u\in\Gc(\Om)$. The operator $a(x,D)$ maps $\Gc(\Om)$ continuously into $\G(\Om)$ and can be extended to a continuous $\wt{\C}$-linear map from $\LL(\G(\Om),\wt{\C})$ to $\LL(\Gc(\Om),\wt{\C})$. If $a$ is of slow scale type then $a(x,D)$ maps $\Gcinf(\Om)$ continuously into $\Ginf(\Om)$.
Pseudodifferential operators with generalized symbol of order $-\infty$ are regularizing, in the sense that $a(x,D)$ maps $\Lb(\G(\Om),\wt{\C})$ to $\G(\Om)$ if $a\in\wt{\mathcal{S}}^{-\infty}(\Om\times\R^n)$ and 
$\Lb(\G(\Om),\wt{\C})$ to $\Ginf(\Om)$ if $a\in\wt{\mathcal{S}}^{-\infty,\ssc}(\Om\times\R^n)$. Clearly, all the previous results can be stated for pseudodifferential operators given by a generalized amplitude $a(x,y,\xi)\in\widetilde{\mathcal{S}}^m_{\rho,\delta}(\Om\times\Om\times\R^n)$. For a complete overview on generalized pseudodifferential operators acting on spaces of Colombeau type we advise the reader to refer to \cite{Garetto:06a, GGO:03, GH:05}

\subsection{Generalized elliptic symbols}
One of the main issues in developing a theory of generalized symbols has been the search for a notion of generalized elliptic symbol. This is obviously related to the construction of a generalized pseudodifferential parametrix by means of which to investigate problems of $\G$- and $\Ginf$-regularity. In the sequel we recall some of the results obtain in this direction in \cite{GGO:03, GH:05}, which will be employed in Section \ref{section_comp}. We work at the level of representatives and we set $\rho=1$, $\delta=0$. We leave to the reader the proof of the next proposition which is based on \cite[Section 6]{GGO:03}.
\begin{prop}
\label{prop_ellip}
Let $(a_\eps)_\eps\in\mM_{S^m(\Om\times\R^n\setminus 0)}$ such that 
\begin{itemize}
\item[(e1)] for all $K\Subset\Om$ there exists $s\in\R$, $(R_\eps)_\eps\in\EM$ strictly nonzero and $\eta\in(0,1]$ such that
\[
|a_\eps(x,\xi)|\ge \eps^s\lara{\xi}^m,
\]
for all $x\in K$, $|\xi|\ge R_\eps$ and $\eps\in(0,\eta]$.
\end{itemize}
Then,
\begin{itemize}
\item[(i)] for all $K\Subset\Om$, for all $\alpha,\beta\in\N^n$ there exist $N\in\N$, $(R_\eps)_\eps\in\EM$ strictly nonzero and $\eta\in(0,1]$ such that 
\[
|\partial^\alpha_\xi\partial^\beta_x a_\eps(x,\xi)|\le \eps^{-N}\lara{\xi}^{-|\alpha|}|a_\eps(x,\xi)|
\]
for all $x\in K$, $|\xi|\ge R_\eps$ and $\eps\in(0,\eta]$;
\item[(ii)] $(i)$ holds for the net $(a_\eps^{-1})_\eps$;
\item[(iii)] if $(a'_\eps)_\eps\in\mM_{S^{m'}(\Om\times\R^n\setminus 0)}$ with $m'<m$ then $(e1)$ holds for the net $(a_\eps+a'_\eps)_\eps$.
\end{itemize}
Let $(a_\eps)_\eps\in\mM^\ssc_{S^m(\Om\times\R^n\setminus 0)}$ such that 
\begin{itemize}
\item[(e2)] for all $K\Subset\Om$ there exists $(s_\eps)_\eps$ with $(s^{-1}_\eps)_\eps$ s.s.n.,  $(R_\eps)_\eps$ s.s.n. and $\eta\in(0,1]$ such that
\[
|a_\eps(x,\xi)|\ge s_\eps\lara{\xi}^m,
\]
for all $x\in K$, $|\xi|\ge R_\eps$ and $\eps\in(0,\eta]$.
\end{itemize}
Then,
\begin{itemize}
\item[(iv)] for all $K\Subset\Om$, for all $\alpha,\beta\in\N^n$ there exist $(c_\eps)_\eps$, $(R_\eps)_\eps$ s.s.n and $\eta\in(0,1]$ such that 
\[
|\partial^\alpha_\xi\partial^\beta_x a_\eps(x,\xi)|\le c_\eps\lara{\xi}^{-|\alpha|}|a_\eps(x,\xi)|
\]
for all $x\in K$, $|\xi|\ge R_\eps$ and $\eps\in(0,\eta]$;
\item[(v)] $(i)$ holds for the net $(a_\eps^{-1})_\eps$;
\item[(vi)] if $(a'_\eps)_\eps\in\mM^\ssc_{S^{m'}(\Om\times\R^n\setminus 0)}$ with $m'<m$ then $(e2)$ holds for the net $(a_\eps+a'_\eps)_\eps$.
\end{itemize}
\end{prop}

\begin{prop}
\label{prop_ellip_2}
Let $(a_\eps)_\eps$ be a net of elliptic symbols of $S^m(\Om\times\R^n\setminus 0)$.
\begin{itemize}
\item[(i)] If $(a_\eps)_\eps\in\mM_{S^m(\Om\times\R^n\setminus 0)}$ fulfills condition $(e1)$ then there exists $(p_\eps)_\eps\in\mM_{S^{-m}(\Om\times\R^n\setminus 0)}$ such that for all $\eps\in(0,1]$ 
\[
p_\eps a_\eps =1+r_\eps,
\]
where $(r_\eps)_\eps\in\mM_{S^{-\infty}(\Om\times\R^n\setminus 0)}$.
\item[(ii)] If $(a_\eps)_\eps\in\mM^\ssc_{S^m(\Om\times\R^n\setminus 0)}$ fulfills condition $(e2)$ then there exists $(p_\eps)_\eps\in\mM^\ssc_{S^{-m}(\Om\times\R^n\setminus 0)}$ such that for all $\eps\in(0,1]$
\[
p_\eps a_\eps =1+r_\eps,
\]
where $(r_\eps)_\eps\in\mM^\ssc_{S^{-\infty}(\Om\times\R^n\setminus 0)}$.
\end{itemize}
\end{prop}
\begin{proof}
As in \cite[Proposition 6.4]{GGO:03} we define $p_\eps$ as
\[
\sum_ja_\epsilon^{-1}(x,\xi)\varphi\big(\frac{\xi}{R_{j,\epsilon}}\big)\psi_j(x),
\]
where $\psi_j$ is a partition of unity subordinated to a covering of relatively compact subsets $\Om_j$ of $\Om$, $(R_{j,\eps})_\eps$ is the radius corresponding to $\overline{\Om_j}$ and $\varphi$ is a smooth function on $\R^n$ such that $\varphi(\xi)=0$ for $|\xi|\le 1$ and $\varphi(\xi)=1$ for $|\xi|\ge 2$. From Proposition \ref{prop_ellip} we have that $(e1)$ yields $(p_\eps)_\eps\in\mM_{S^{-m}(\Om\times\R^n\setminus 0)}$ and $(e2)$ yields $(p_\eps)_\eps\in\mM^\ssc_{S^{-m}(\Om\times\R^n\setminus 0)}$. Let $K\Subset\Om$. By construction, for all $x\in K$,
\[
p_\eps(x,\xi)a_\eps(x,\xi)=1+\biggl(\sum_{j=0}^{j_0}\varphi\big(\frac{\xi}{R_{j,\epsilon}})\psi_j(x)-1\biggr)= 1+\sum_{j=0}^{j_0}\biggl(\varphi\big(\frac{\xi}{R_{j,\epsilon}})-1\biggr)\psi_j(x),
\]
and the following estimates hold for all $l>0$ and $\alpha\in\N^n\setminus 0$:
\begin{multline*}
\sup_{\xi\neq 0}\lara{\xi}^l|\varphi\big(\frac{\xi}{R_{j,\epsilon}})-1|\le\sup_{|\xi|\le 2R_{j,\epsilon}}\lara{\xi}^l|\varphi\big(\frac{\xi}{R_{j,\epsilon}})-1|\le c_\varphi\lara{2R_{j,\eps}}^l,\\
\sup_{\xi\neq 0}\lara{\xi}^l|\partial^\alpha_\xi\varphi\big(\frac{\xi}{R_{j,\epsilon}})|(R_{j,\eps})^{-|\alpha|}\le\sup_{R_{j,\eps}\le|\xi|\le 2R_{j,\epsilon}}\lara{\xi}^l|\partial^\alpha_\xi\varphi\big(\frac{\xi}{R_{j,\epsilon}})|(R_{j,\eps})^{-|\alpha|}\\
\le c_\varphi\lara{2R_{j,\eps}}^l(R_{j,\eps})^{-|\alpha|}. 
\end{multline*}
We deduce that $(p_\eps a_\eps-1)_\eps$ belongs to $\mM_{S^{-\infty}(\Om\times\R^n\setminus 0)}$ under the hypothesis $(e1)$ on $(a_\eps)_\eps$ and that $(p_\eps a_\eps-1)_\eps$ belongs to $\mM^\ssc_{S^{-\infty}(\Om\times\R^n\setminus 0)}$ under the hypothesis $(e2)$ on $(a_\eps)_\eps$.
\end{proof}

\subsection{Microlocal analysis in the Colombeau context: generalized wave front sets in $\LL(\Gc(\Om),\wt{\C})$}
\label{sub_sec_micro}
In this subsection we recall the basic notions of microlocal analysis which involve the duals of the Colombeau algebras $\Gc(\Om)$ and $\G(\Om)$ and have been developed in \cite{Garetto:06a}. In this generalized context the role which is classically played by $\S(\R^n)$ is given to the Colombeau algebra $\GS(\R^n):=\G_{\S(\R^n)}$. $\GS(\R^n)$ is topologized as in Subsection \ref{subsection_G_E} and its dual $\LL(\GS(\R^n),\wt{\C})$ is endowed with the topology of uniform convergence on bounded subsets. In the sequel $\Gt(\R^n)$ denotes the Colombeau algebra of tempered generalized functions defined as the quotient $\Et(\R^n)/\Nt(\R^n)$, where $\Et(\R^n)$ is the algebra of all  \emph{$\tau$-moderate} nets $(u_\eps)_\eps\in\Et[\R^n]:=\mO_M(\R^n)^{(0,1]}$ such that 
\[
\forall \alpha\in\N^n\, \exists N\in\N\qquad \sup_{x\in\R^n}(1+|x|)^{-N}|\partial^\alpha u_\eps(x)|=O(\eps^{-N})\qquad \text{as}\ \eps\to 0
\]
and $\Nt(\R^n)$ is the ideal of all \emph{$\tau$-negligible} nets $(u_\eps)_\eps\in\Et[\R^n]$ such that
\[
\forall \alpha\in\N^n\, \exists N\in\N\, \forall q\in\N\quad \sup_{x\in\R^n}(1+|x|)^{-N}|\partial^\alpha u_\eps(x)|=O(\eps^{q})\ \text{as}\ \eps\to 0.
\]
Theorem 3.8 in \cite{Garetto:05b} shows that we have the chain of continuous embeddings
\[
\GS(\R^n)\subseteq\Gt(\R^n)\subseteq\LL(\GS(\R^n),\wt{\C}).
\]
Moreover, since for any $u\in\Gc(\Om)$ with $\supp\, u\subseteq K\Subset\Om$ and any $K'\Subset\Om$ with $K\subset{\rm{Int}}\,K'$ one can find a representative $(u_\eps)_\eps$ with $\supp\, u_\eps\subseteq K'$ for all $\eps\in(0,1]$, we have that $\Gc(\Om)$ is continuously embedded into $\GS(\R^n)$.
 
\subsubsection*{The Fourier transform on $\GS(\R^n)$, $\LL(\GS(\R^n),\wt{\C})$ and $\LL(\G(\Om),\wt{\C})$.}
The Fourier transform on $\GS(\R^n)$ is defined by the corresponding transformation at the level of representatives, as follows:
\[
\mF:\GS(\R^n)\to\GS(\R^n):u\to [(\widehat{u_\eps})_\eps].
\] 
$\mF$ is a $\wt{\C}$-linear continuous map from $\GS(\R^n)$ into itself which extends to the dual in a natural way. In detail, we define the Fourier transform of $T\in\LL(\GS(\R^n),\wt{\C})$ as the functional in $\LL(\GS(\R^n),\wt{\C})$ given by
\[
\mF(T)(u)=T(\mF u).
\]
As shown in \cite[Remark 1.5]{Garetto:06a} $\LL(\G(\Om),\wt{\C})$ is embedded in $\LL(\GS(\R^n),\wt{\C})$ by means of the map
\[
\LL(\G(\Om),\wt{\C})\to\LL(\GS(\R^n),\wt{\C}):T\to \big(u\to T(({u_\eps}_{\vert_\Om})_\eps+\Neg(\Om))\big).
\]
In particular, when $T$ is a $\rfunc$ functional in $\LL(\G(\Om),\wt{\C})$ we have from \cite[Proposition 1.6, Remark 1.7]{Garetto:06a} that the Fourier transform of $T$ is the tempered generalized function obtained as the action of $T(y)$ on $\esp^{-iy\xi}$, i.e., $\mF(T)=T(\esp^{-i\cdot\xi})=(T_\eps(\esp^{-i\cdot\xi}))_\eps+\Nt(\R^n)$. 

\subsubsection*{Generalized wave front sets of a functional in $\LL(\Gc(\Om),\wt{\C})$.}
The notions of $\G$-wave front set and $\Ginf$-wave front set of a functional in $\LL(\Gc(\Om),\wt{\C})$ have been introduced in \cite{Garetto:06a} as direct analogues of the distributional wave front set in \cite{Hoermander:71}. They employ a subset of the space $\G^\ssc_{S^m(\Om\times\R^n)}$ of generalized symbols of slow scale type denoted by $\Syscu^m(\Om\times\R^n)$ (see \cite[Definition 1.1]{GH:05}) and a suitable notion of slow scale micro-ellipticity \cite[Definition 1.2]{GH:05}. In detail, $(x_0,\xi_0)\not\in\WF_\G\, T$ (resp. $(x_0,\xi_0)\not\in\WF_{\Ginf}\, T$) if there exists $a(x,D)$ properly supported with $a\in\Syscu^0(\Om\times\R^n)$ such that $a$ is slow scale micro-elliptic at $(x_0,\xi_0)$ and $a(x,D)T\in\G(\Om)$ (resp. $a(x,D)T\in\Ginf(\Om)$).\\
When $T$ is a basic functional of $\LL(\Gc(\Om),\wt{\C})$, Proposition 3.14 in \cite{Garetto:06a} proves that one can limit to classical properly supported pseudodifferential operators in the definition of $\WF_\G\, T$ and $\WF_{\Ginf}\, T$. More precisely,
\begin{equation}
\label{WGcl}
{\rm{W}}_{{\rm{cl}},\G}(T):=\bigcap_{AT\in\G(\Om)}\Char(A)
\end{equation}
and
\begin{equation}
\label{WGinfcl}
{\rm{W}}_{{\rm{cl}},\Ginf}(T):=\bigcap_{AT\in\Ginf(\Om)}\Char(A)
\end{equation}
where the intersections are taken over all the classical properly supported operators $A\in\Psi^0(\Om)$ such that $AT\in\G(\Om)$ in \eqref{WGcl} and $AT\in\Ginf(\Om)$ in \eqref{WGinfcl}. $\WF_\G T$ and $\WF_{\Ginf}T$ are both closed conic subsets of $\CO{\Om}$ and, as proved in \cite[Proposition 3.5]{Garetto:06a},  
\[
\pi_\Om(\WF_\G T)=\singsupp_\G T
\]
and
\[
\pi_{\Om}(\WF_{\Ginf}T)=\singsupp_{\Ginf} T.
\]
\subsubsection*{Characterization of $\WF_\G T$ and $\WF_{\Ginf} T$ when $T$ is a basic functional.} 
We will employ a useful characterization of the $\G$-wave front set and the $\Ginf$-wave front set valid for functionals which are basic. It involves the sets of generalized functions $\G_{\S,0}(\Gamma)$ and $\Ginf_{\S\hskip-2pt,0}(\Gamma)$, defined on the conic subset $\Gamma$ of $\R^n\setminus 0$, as follows:
\begin{multline*}
\G_{\S,0}(\Gamma):=\{u\in\Gt(\R^n):\ \exists (u_\eps)_\eps\in u\ \forall l\in\R\, \exists N\in\N\\ \sup_{\xi\in\Gamma}\lara{\xi}^l|u_\eps(\xi)|=O(\eps^{-N})\, \text{as $\eps\to 0$}\},
\end{multline*}
\begin{multline*}
\Ginf_{\S\hskip-2pt,0}(\Gamma):=\{u\in\Gt(\R^n):\ \exists (u_\eps)_\eps\in u\ \exists N\in\N\, \forall l\in\R\\ \sup_{\xi\in\Gamma}\lara{\xi}^l|u_\eps(\xi)|=O(\eps^{-N})\, \text{as $\eps\to 0$}\}.
\end{multline*}
Let $T\in\LL(\Gc(\Om),\wt{\C})$. Theorem 3.13 in \cite{Garetto:06a} shows that:
\begin{itemize}
\item[(i)] $(x_0,\xi_0)\not\in\WF_\G T$ if and only if there exists a conic neighborhood $\Gamma$ of $\xi_0$ and a cut-off function $\varphi\in\Cinfc(\Om)$ with $\varphi(x_0)=1$ such that $\mF(\varphi T)\in\G_{\S,0}(\Gamma)$.
\item[(ii)] $(x_0,\xi_0)\not\in\WF_{\Ginf} T$ if and only if there exists a conic neighborhood $\Gamma$ of $\xi_0$ and a cut-off function $\varphi\in\Cinfc(\Om)$ with $\varphi(x_0)=1$ such that $\mF(\varphi T)\in\Ginf_{\S\hskip-2pt,0}(\Gamma)$.
\end{itemize}

\section{Generalized oscillatory integrals: definition}
This section is devoted to a notion of oscillatory integral where both the amplitude and the phase function are generalized objects of Colombeau type. 

In the sequel $\Om$ is an arbitrary open subset of $\R^n$.
We recall that $\phi(y,\xi)$ is a \emph{phase function} on $\Om\times\R^p$ if it is a smooth function on $\Om\times\R^p\setminus 0$, real valued, positively homogeneous of degree $1$ in $\xi$ with $\nabla_{y,\xi}\phi(y,\xi)\neq 0$ for all $y\in\Om$ and $\xi\in\R^p\setminus 0$. We denote the set of all phase functions on $\Om\times\R^p$ by $\Phi(\Om\times\R^p)$ and the set of all nets in $\Phi(\Om\times\R^p)^{(0,1]}$ by $\Phi[\Om\times\R^p]$. The notations concerning classes of symbols have been introduced in Subsection \ref{subsec_gen_symb}. The proofs of the statements collected in this section can be found in \cite{GHO:06}. In the paper \cite{GHO:06} the authors deal with generalized symbols in $\wt{S}^m_{\rho,\delta}(\Om\times\R^p)$ as well as  with regular generalized symbols. This last class of symbols is modelled on the subalgebra $\Ginf(\Om)$ of regular generalized functions and contains the generalized symbols of slow scale type as a submodule. Even though many statements of Section 3, 4 and 6 hold for regular symbols as well, for the sake of simplicity and in order to have uniformity of assumptions between phase functions and symbols, we limit in this work to consider $\wt{\mathcal{S}}^m_{\rho,\delta}(\Om\times\R^p)$ and the smaller class $\wt{S}^{m,\ssc}_{\rho,\delta}(\Om\times\R^p)$ of generalized symbols of slow scale type. 
\begin{defn}
\label{def_phase_moderate}
An element of $\mathcal{M}_\Phi(\Om\times\R^p)$ is a net $(\phi_\eps)_\eps\in\Phi[\Om\times\R^p]$ satisfying the conditions:
\begin{itemize}
\item[(i)] $(\phi_\eps)_\eps\in\mathcal{M}_{S^1_{\rm{hg}}(\Om\times\R^p\setminus 0)}$,
\item[(ii)] for all $K\Subset\Om$ the net $$\biggl(\inf_{y\in K,\xi\in\R^p\setminus 0}\biggl|\nabla \phi_\eps\biggl(y,\frac{\xi}{|\xi|}\biggr)\biggr|^2\biggr)_\eps$$ is strictly nonzero. 
\end{itemize}
On $\MPhi(\Om\times\R^p)$ we introduce the equivalence relation $\sim$ as follows: $(\phi_\eps)_\eps\sim(\omega_\eps)_\eps$ if and only if $(\phi_\eps-\omega_\eps)\in\Neg_{S^1_{\rm{hg}}(\Om\times\R^p\setminus 0)}$. The elements of the factor space  $$\wt{\Phi}(\Om\times\R^p):={\mathcal{M}_\Phi(\Om\times\R^p)}/{\sim}.$$
will be called \emph{generalized phase functions}. 
\end{defn}
We shall employ the equivalence class notation $[(\phi_\eps)_\eps]$ for $\phi\in\wt{\Phi}(\Om\times\R^p)$. When $(\phi_\eps)_\eps$ is a net of phase functions, i.e. $(\phi_\eps)_\eps\in\Phi[\Om\times\R^p]$, Lemma 1.2.1 in \cite{Hoermander:71} shows that there exists a family of partial differential operators $(L_{\phi_\eps})_\eps$ such that ${\ }^tL_{\phi_\eps}\esp^{i\phi_\eps}=\esp^{i\phi_\eps}$ for all $\eps\in(0,1]$. $L_{\phi_\eps}$ is of the form 
\begin{equation}
\label{def_L_phi_cl}
\sum_{j=1}^p a_{j,\eps}(y,\xi)\frac{\partial}{\partial_{\xi_j}} +\sum_{k=1}^n b_{k,\eps}(y,\xi)\frac{\partial}{\partial_{y_k}}+ c_\eps(y,\xi),
\end{equation}
where the coefficients $(a_{j,\eps})_\eps$ belong to $S^0[\Om\times\R^p]$ and $(b_{k,\eps})_\eps$, $(c_\eps)_\eps$ are elements of $S^{-1}[\Om\times\R^p]$. The following technical lemma is crucial in proving Proposition \ref{prop_operator}.   
\begin{lem}
\label{lemma_1}
\leavevmode
\begin{trivlist}
\item[(i)] Let $\varphi_{\phi_\eps}(y,\xi):=|\nabla\phi_\eps(y,\xi/|\xi|)|^{-2}$. If $(\phi_\eps)_\eps\in\MPhi(\Om\times\R^p)$ then
\[
(\varphi_{\phi_\eps})_\eps\in\mM_{S^0_{\rm{hg}}(\Om\times\R^p\setminus 0)}.
\]
\item[(ii)] If $(\phi_\eps)_\eps, (\omega_\eps)_\eps\in\MPhi(\Om\times\R^p)$ and $(\phi_\eps)_\eps\sim(\omega_\eps)_\eps$ then
\[
\big(({\partial_{\xi_j}\phi_\eps})\varphi_{\phi_\eps}-({\partial_{\xi_j}\omega_\eps})\varphi_{\omega_\eps}\big)_\eps\in\Neg_{S^0_{\rm{hg}}(\Om\times\R^p\setminus 0)}
\]
for all $j=1,...,p$ and
\[
\big(({\partial_{y_k}\phi_\eps}){|\xi|^{-2}\varphi_{\phi_\eps}}-({\partial_{y_k}\omega_\eps}){|\xi|^{-2}\varphi_{\omega_\eps}}\big)_\eps\in\Neg_{S^{-1}_{\rm{hg}}(\Om\times\R^p\setminus 0)}
\]
for all $k=1,...,n$.
\end{trivlist}
\end{lem} 
\begin{prop}
\label{prop_operator}
\leavevmode
\begin{trivlist}
\item[(i)] If $(\phi_\eps)_\eps\in\MPhi(\Om\times\R^p)$ then $(a_{j,\eps})_\eps\in\mM_{S^0(\Om\times\R^p)}$ for all $j=1,...,p$, $(b_{k,\eps})_\eps\in\mM_{S^{-1}(\Om\times\R^p)}$ for all $k=1,...,n$, and $(c_\eps)_\eps\in\mM_{S^{-1}(\Om\times\R^p)}$.
\item[(ii)] If $(\phi_\eps)_\eps, (\omega_\eps)_\eps\in\MPhi(\Om\times\R^p)$ and $(\phi_\eps)_\eps\sim(\omega_\eps)_\eps$ then
\[
L_{\phi_\eps}-L_{\omega_\eps}=\sum_{j=1}^p a'_{j,\eps}(y,\xi)\frac{\partial}{\partial_{\xi_j}} +\sum_{k=1}^n b'_{k,\eps}(y,\xi)\frac{\partial}{\partial_{y_k}}+ c'_\eps(y,\xi),
\]
where $(a'_{j,\eps})_\eps\in\Neg_{S^{0}(\Om\times\R^p)}$, $(b'_{k,\eps})_\eps\in\Neg_{S^{-1}(\Om\times\R^p)}$ and $(c'_\eps)_\eps\in\Neg_{S^{-1}(\Om\times\R^p)}$ for all $j=1,...,p$ and $k=1,...,n$.
\end{trivlist}
\end{prop} 
As a consequence of Propositions \ref{prop_operator} we can claim that any generalized phase function $\phi\in\wt{\Phi}(\Om\times\R^p)$ defines a generalized partial differential operator
\[
L_\phi(y,\xi,\partial_y,\partial_\xi)=\sum_{j=1}^p a_{j}(y,\xi)\frac{\partial}{\partial_{\xi_j}} +\sum_{k=1}^n b_{k}(y,\xi)\frac{\partial}{\partial_{y_k}}+ c(y,\xi)
\]
whose coefficients $\{a_j\}_{j=1}^p$ and $\{b_k\}_{k=1}^n$, $c$ are generalized symbols in $\wt{\mathcal{S}}^{0}(\Om\times\R^p)$ and $\wt{\mathcal{S}}^{-1}(\Om\times\R^p)$, respectively. By construction, $L_\phi$ maps $\wt{\mathcal{S}}^m_{\rho,\delta}(\Om\times\R^p)$ continuously into $\wt{\mathcal{S}}^{m-s}_{\rho,\delta}(\Om\times\R^p)$, where $s=\min\{\rho,1-\delta\}$. Hence $L^k_\phi$ is continuous from $\wt{\mathcal{S}}^m_{\rho,\delta}(\Om\times\R^p)$ to $\wt{\mathcal{S}}^{m-ks}_{\rho,\delta}(\Om\times\R^p)$.

Before stating the next proposition we recall a classical lemma valid any symbol $\phi\in S^1(\Om\times\R^p\setminus 0)$. 
\begin{lem}
\label{lemma_esp_classic}
For all $\alpha\in\N^p$ and $\beta\in\N^n$,
\[
\partial^\alpha_\xi\partial^\beta_y\esp^{i\phi(y,\xi)}=\sum_{\substack{k\le|\alpha+\beta|,\\ \alpha_1+\alpha_2+...+\alpha_k=\alpha\\ \beta_1+\beta_2+...+\beta_k=\beta}}c_{\alpha_1,...,\alpha_k,\beta_1,...,\beta_k}\,\partial^{\alpha_1}_\xi\partial^{\beta_1}_y\phi(y,\xi)...\partial^{\alpha_k}_\xi\partial^{\beta_k}_y\phi(y,\xi).
\]
It follows that 
\[
\partial^\alpha_\xi\partial^\beta_y\esp^{i\phi(y,\xi)}= \esp^{i\phi(y,\xi)}a_{\alpha,\beta}(y,\xi),
\]
where $a_{\alpha,\beta}\in S^{|\beta|}(\Om\times\R^p\setminus 0)$ and 
\begin{equation}
\label{coeff_a}
|a_{\alpha,\beta}|^{(|\beta|)}_{K,j}\le c\sup_{y\in K,\xi\neq 0}\sup_{|\gamma+\delta|\le|\alpha+\beta|+j}\lara{\xi}^{-1+|\gamma|}|\partial^\gamma_\xi\partial^\delta_y\phi(y,\xi)|,
\end{equation}
where the constant $c$ depends only on $\alpha$, $\beta$, and $j$. 
\end{lem}
From \eqref{coeff_a} we have that  
\[
(\phi_\eps)_\eps\in\mM_{S^1(\Om\times\R^p\setminus 0)}\quad \Rightarrow \quad (a_{\alpha,\beta,\eps})_\eps\in \mM_{S^{|\beta|}(\Om\times\R^p\setminus 0)}
\]
or more in general that the net $(a_{\alpha,\beta,\eps})_\eps$ has the ``$\eps$-scale properties'' of $(\phi_\eps)_\eps$.
\begin{prop}
\label{prop_exp}
Let $\phi\in\wt{\Phi}(\Om\times\R^p)$. The exponential $$\esp^{i\phi(y,\xi)}$$ is a well-defined element of $\wt{\mathcal{S}}^1_{0,1}(\Om\times\R^p\setminus 0)$.
\end{prop}
\begin{proof}
From Lemma \ref{lemma_esp_classic} we have that if $(\phi_\eps)_\eps\in\mM_\Phi(\Om\times\R^p)$ then $(\esp^{i\phi_\eps(y,\xi)})_\eps\in\mM_{S^0_{0,1}(\Om\times\R^p\setminus 0)}$. When $(\phi_\eps)_\eps\sim(\omega_\eps)_\eps$, the equality
\begin{multline*}
\esp^{i\omega_\eps(y,\xi)}-\esp^{i\phi_\eps(y,\xi)}=\esp^{i\omega_\eps(y,\xi)}\big(1-\esp^{i(\phi_\eps-\omega_\eps)(y,\xi)}\big)\\
=\esp^{i\omega_\eps(y,\xi)}\sum_{j=1}^p  \esp^{i(\phi_\eps-\omega_\eps)(y,\theta\xi)}\partial_{\xi_j}(\phi_\eps-\omega_\eps)(y,\theta\xi)i\xi_j,
\end{multline*}
with $\theta\in(0,1)$, implies that
\begin{equation}
\label{neg_esp}
\sup_{y\in K,\xi\in\R^p\setminus 0}|\xi|^{-1}\big|\esp^{i\omega_\eps(y,\xi)}-\esp^{i\phi_\eps(y,\xi)}\big|=O(\eps^q)
\end{equation}
for all $q\in\N$. At this point writing $\partial^\alpha_\xi\partial^\beta_y( \esp^{i\omega_\eps(y,\xi)}-\esp^{i\phi_\eps(y,\xi)})$ as
\begin{multline*}
\partial^\alpha_\xi\partial^\beta_y \esp^{i\om_\eps(y,\xi)}\big(1-\esp^{i(\phi_\eps-\omega_\eps)(y,\xi)}\big)+\\
+\sum_{\alpha'<\alpha,\beta'<\beta}\binom{\alpha}{\alpha'}\binom{\beta}{\beta'}\partial^{\alpha'}_\xi\partial^{\beta'}_y \esp^{i\omega_\eps(y,\xi)}\big(-\partial^{\alpha-\alpha'}_\xi\partial^{\beta-\beta'}_y \esp^{i(\phi_\eps-\omega_\eps)(y,\xi)}\big)
\end{multline*}
we obtain the characterizing estimate of a net in $\Neg_{S^1_{0,1}(\Om\times\R^p\setminus 0)}$, using \eqref{neg_esp} the moderateness of $(\esp^{i\om_\eps(y,\xi)})_\eps$ and Lemma \ref{lemma_esp_classic}.
\end{proof}
By construction of the operator $L_\phi$ the equality ${\ }^t L_\phi \esp^{i\phi}=\esp^{i\phi}$ holds in $\wt{\mathcal{S}}^1_{0,1}(\Om\times\R^p\setminus 0)$. In addition, Proposition \ref{prop_exp} and the properties of $L^k_\phi$ allow to conclude that $$\esp^{i\phi(y,\xi)}L^k_\phi(a(y,\xi)u(y))$$ is a generalized symbol in $\wt{\mathcal{S}}^{m-ks+1}_{0,1}(\Om\times\R^p)$ which is integrable on $\Om\times\R^p$ in the sense of Section \ref{section_basic} when $m-ks+1<-p$. From now on we assume that $\rho>0$ and $\delta<1$.
\begin{defn}
\label{def_gen_osc}
Let $\phi\in\wt{\Phi}(\Om\times\R^p)$, $a\in\wt{\mathcal{S}}^m_{\rho,\delta}(\Om\times\R^p)$ and $u\in\Gc(\Om)$. The \emph{generalized oscillatory integral}
\[
\int_{\Om\times\R^p}\esp^{i\phi(y,\xi)}a(y,\xi)u(y)\, dy\,\dslash\xi
\]
is defined as
\[
\int_{\Om\times\R^p}\esp^{i\phi(y,\xi)}L^k_\phi(a(y,\xi)u(y))\, dy\,\dslash\xi
\]
where $k$ is chosen such that $m-ks+1<-p$.  
\end{defn}
The functional 
\[
I_\phi(a):\Gc(\Om)\to\wt{\C}:u\to\int_{\Om\times\R^p}\esp^{i\phi(y,\xi)}a(y,\xi)u(y)\, dy\, \dslash\xi
\]
belongs to the dual $\LL(\Gc(\Om),\wt{\C})$. Indeed, by \eqref{bil_product}, the continuity of $L^k_\phi$ and of the product between generalized symbols we have that the map
\[
\Gc(\Om)\to\wt{\mathcal{S}}^{m-ks+1}_{0,1}(\Om\times\R^p):u\to \esp^{i\phi(y,\xi)}L^k_\phi(a(y,\xi)u(y))
\]
is continuous and thus, by an application of the integral on $\Om\times\R^p$, the resulting functional $I_\phi(a)$ is continuous.
  
\section{Generalized Fourier integral operators}
\label{gen_sec}
\subsection*{Definition and mapping properties}
We now study oscillatory integrals where an additional parameter $x$, varying in an open subset $\Om'$ of $\R^{n'}$, appears in the phase function $\phi$ and in the symbol $a$. The dependence on $x$ is investigated in the Colombeau context. We denote by $\Phi[\Om';\Om\times\R^p]$ the set of all nets $(\phi_\eps)_{\eps\in(0,1]}$ of continuous functions on $\Om'\times\Om\times\R^p$ which are smooth on $\Om'\times\Om\times\R^p\setminus\{0\}$ and such that $(\phi_\eps(x,\cdot,\cdot))_\eps\in\Phi[\Om\times\R^p]$ for all $x\in\Om'$.
\begin{defn}
\label{def_phase_x_moderate}
An element of $\mM_{\Phi}(\Om';\Om\times\R^p)$ is a net $(\phi_\eps)_\eps\in\Phi[\Om';\Om\times\R^p]$ satisfying the conditions:
\begin{itemize}
\item[(i)] $(\phi_\eps)_\eps\in\mM_{S^1_{\rm{hg}}(\Om'\times\Om\times\R^p\setminus 0)}$,
\item[(ii)] for all $K'\Subset\Om'$ and $K\Subset\Om$ the net
\begin{equation}
\label{net_FIO}
\biggl(\inf_{x\in K',y\in K,\xi\in\R^p\setminus 0}\biggl|\nabla_{y,\xi} \phi_\eps\biggl(x,y,\frac{\xi}{|\xi|}\biggr)\biggr|^2\biggr)_\eps
\end{equation}
is strictly nonzero.
\end{itemize}
On $\mM_{\Phi}(\Om';\Om\times\R^p)$ we introduce the equivalence relation $\sim$ as follows: $(\phi_\eps)_\eps\sim(\omega_\eps)_\eps$ if and only if $(\phi_\eps-\omega_\eps)_\eps\in\Neg_{S^1_{\rm{hg}}(\Om'\times\Om\times\R^p\setminus 0)}$. The elements of the factor space 
\[
\wt{\Phi}(\Om';\Om\times\R^p):=\mM_{\Phi}(\Om';\Om\times\R^p) / \sim.
\]
are called \emph{generalized phase functions with respect to the variables in $\Om\times\R^p$}.  
\end{defn}
Lemma \ref{lemma_1} as well as Proposition \ref{prop_operator} can be adapted to nets in $\mM_{\Phi}(\Om';\Om\times\R^p)$. More precisely, the operator
\begin{equation}
\label{smilla}
L_{\phi_\eps}(x;y,\xi,\partial_y,\partial_\xi)=\sum_{j=1}^p a_{j,\eps}(x,y,\xi)\frac{\partial}{\partial_{\xi_j}} +\sum_{k=1}^n b_{k,\eps}(x,y,\xi)\frac{\partial}{\partial_{y_k}}+ c_\eps(x,y,\xi)
\end{equation}
defined for any value of $x$ by \eqref{def_L_phi_cl}, has the property ${\ }^tL_{\phi_\eps(x,\cdot,\cdot)}\esp^{i\phi_\eps(x,\cdot,\cdot)}=\esp^{i\phi_\eps(x,\cdot,\cdot)}$ for all $x\in\Om'$ and $\eps\in(0,1]$ and its coefficients depend smoothly on $x\in\Om'$.
\begin{lem}
\label{lemma_1_x}
\leavevmode
\begin{trivlist}
\item[(i)] Let 
\begin{equation}
\label{def_varphi}
\varphi_{\phi_\eps}(x,y,\xi):=|\nabla_{y,\xi}\phi_\eps(x,y,\xi/|\xi|)|^{-2}.
\end{equation}
If $(\phi_\eps)_\eps\in\mM_{\Phi}(\Om';\Om\times\R^p)$ then $(\varphi_{\phi_\eps})_\eps\in\mM_{S^0_{\rm{hg}}(\Om'\times\Om\times\R^p\setminus 0)}$.
\item[(ii)] If $(\phi_\eps)_\eps, (\omega_\eps)_\eps\in\MPhi(\Om';\Om\times\R^p)$ and $(\phi_\eps)_\eps\sim(\omega_\eps)_\eps$ then
\[
\big(({\partial_{\xi_j}\phi_\eps})\varphi_{\phi_\eps}-({\partial_{\xi_j}\omega_\eps})\varphi_{\omega_\eps}\big)_\eps\in\Neg_{S^0_{\rm{hg}}(\Om'\times\Om\times\R^p\setminus 0)}
\]
for all $j=1,...,p$ and
\[
\big(({\partial_{y_k}\phi_\eps}){|\xi|^{-2}\varphi_{\phi_\eps}}-({\partial_{y_k}\omega_\eps}){|\xi|^{-2}\varphi_{\omega_\eps}}\big)_\eps\in\Neg_{S^{-1}_{\rm{hg}}(\Om'\times\Om\times\R^p\setminus 0)}
\]
for all $k=1,...,n$.
\end{trivlist}
\end{lem}
\begin{prop}
\label{prop_operator_x}
\leavevmode
\begin{trivlist}
\item[(i)] If $(\phi_\eps)_\eps\in\mM_{\Phi}(\Om';\Om\times\R^p)$ then the coefficients occurring in \eqref{smilla} satisfy the following: $(a_{j,\eps})_\eps\in\mM_{S^0(\Om'\times\Om\times\R^p)}$ for all $j=1,...,p$, $(b_{k,\eps})_\eps\in\mM_{S^{-1}(\Om'\times\Om\times\R^p)}$ for all $k=1,...,n$, and $(c_\eps)_\eps\in\mM_{S^{-1}(\Om'\times\Om\times\R^p)}$.
\item[(ii)] If $(\phi_\eps)_\eps, (\omega_\eps)_\eps\in\mM_{\Phi}(\Om';\Om\times\R^p)$ and $(\phi_\eps)_\eps\sim(\omega_\eps)_\eps$ then 
\begin{equation}
\label{L_phi_eps-L_omega_eps_x}
L_{\phi_\eps}-L_{\omega_\eps}=\sum_{j=1}^p a'_{j,\eps}(x,y,\xi)\frac{\partial}{\partial_{\xi_j}} +\sum_{k=1}^n b'_{k,\eps}(x,y,\xi)\frac{\partial}{\partial_{y_k}}+ c'_\eps(x,y,\xi),
\end{equation}
where $(a'_{j,\eps})_\eps\in\Neg_{S^{0}(\Om'\times\Om\times\R^p)}$, $(b'_{k,\eps})_\eps\in\Neg_{S^{-1}(\Om'\times\Om\times\R^p)}$ and $(c'_\eps)_\eps\hskip-2pt\in\Neg_{S^{-1}(\Om'\times\Om\times\R^p)}$ for all $j=1,...,p$ and $k=1,...,n$.
\end{trivlist}
\end{prop}
Proposition \ref{prop_operator_x} yields that any generalized phase function $\phi$ in $\wt{\Phi}(\Om';\Om\times\R^p)$ defines a partial differential operator
\begin{equation}
\label{def_L_phi_amp}
L_{\phi}(x;y,\xi,\partial_y,\partial_\xi)=\sum_{j=1}^p a_{j}(x,y,\xi)\frac{\partial}{\partial_{\xi_j}} +\sum_{k=1}^n b_{k}(x,y,\xi)\frac{\partial}{\partial_{y_k}}+ c(x,y,\xi)
\end{equation}
with coefficients $a_j\in\wt{\mathcal{S}}^0(\Om'\times\Om\times\R^p)$, $b_k,c\in\wt{\mathcal{S}}^{-1}(\Om'\times\Om\times\R^p)$ such that ${\ }^tL_\phi \esp^{i\phi}=\esp^{i\phi}$ holds in $\wt{\mathcal{S}}^1_{0,1}(\Om'\times\Om\times\R^p\setminus 0)$.
Arguing as in Proposition \ref{prop_exp} we obtain that $\esp^{i\phi(x,y,\xi)}$ is a well-defined element of $\wt{\mathcal{S}}^1_{0,1}(\Om'\times\Om\times\R^p\setminus 0)$. The usual composition argument implies that the map 
\[
\Gc(\Om)\to\wt{\mathcal{S}}^{m-ks+1}_{0,1}(\Om'\times\Om\times\R^p): u\to \esp^{i\phi(x,y,\xi)}L^k_{\phi}(a(x,y,\xi)u(y))
\]
is continuous.
 
The oscillatory integral  
\begin{multline*}
I_\phi(a)(u)(x)=\int_{\Om\times\R^p}\esp^{i\phi(x,y,\xi)}a(x,y,\xi)u(y)\, dy\, \dslash\xi\\
:=\int_{\Om\times\R^p}\esp^{i\phi(x,y,\xi)}L^k_{\phi}(a(x,y,\xi)u(y))\, dy\, \dslash\xi,
\end{multline*}
where $\phi\in\wt{\Phi}(\Om';\Om\times\R^p)$ and $a\in\wt{\mathcal{S}}^m_{\rho,\delta}(\Om'\times\Om\times\R^p)$ is an element of $\wt{\C}$ for fixed $x\in\Om'$. In particular, $I_\phi(a)(u)$ is the integral on $\Om\times\R^p$ of a generalized amplitude in $\wt{\mathcal{S}}^{l}_{0,1}(\Om'\times\Om\times\R^p)$ having compact support in $y$. The order $l$ can be chosen arbitrarily low. 
\begin{thm}
\label{theorem_map}
Let $\phi\in\wt{\Phi}(\Om';\Om\times\R^p)$, $a\in\wt{\mathcal{S}}^m_{\rho,\delta}(\Om'\times\Om\times\R^p)$ and $u\in\Gc(\Om)$. The generalized oscillatory integral
\begin{equation}
\label{oscillatory_x}
I_{\phi}(a)(u)(x)=\int_{\Om\times\R^p}\esp^{i\phi(x,y,\xi)}a(x,y,\xi)u(y)\, dy\,\dslash\xi  
\end{equation}
defines a generalized function in $\G(\Om')$ and the map
\begin{equation}
\label{def_A_fourier}
A:\Gc(\Om)\to\G(\Om'):u\to I_{\phi}(a)(u)
\end{equation}
is continuous.
\end{thm}
The operator $A$ defined in \eqref{def_A_fourier} is called \emph{generalized Fourier integral operator} with amplitude $a\in\wt{\mathcal{S}}^m_{\rho,\delta}(\Om'\times\Om\times\R^p)$ and phase function $\phi\in\wt{\Phi}(\Om';\Om\times\R^p)$.
 
\begin{ex}
Our outline of a basic theory of Fourier integral operators with Co\-lom\-be\-au generalized
amplitudes and phase functions is motivated to a large extent by potential applications
in regularity theory for generalized solutions to hyperbolic partial (or pseudo-)
differential equations with distributional or Colombeau-type coefficients (or symbols)
and data (cf.\ \cite{HdH:01,LO:91,O:89}). To illustrate the typical situation we
consider here the following simple model: let $u\in\G(\R^2)$ be the solution of the
generalized Cauchy-problem
\begin{align}
\label{G_example1}
\d_t u + c\,\d_x u  + b\,u &= 0\\
u \mid_{t=0} &= g,
\end{align}
where $g$ belongs to $\Gc(\R)$ and the coefficients $b$, $c\in\G(\R^2)$. 
Furthermore, $b$, $c$, as well as $\d_x c$ are
assumed to be of local $L^\infty$-log-type (concerning growth with respect to the
regularization parameter, cf.\ \cite{O:89}), $c$ being generalized real-valued and globally bounded in
addition. Let $\gamma \in \G(\R^3)$ be the unique (global) solution of the
corresponding generalized characteristic ordinary differential equation
\begin{align*}
\diff{s} \gamma(x,t;s) &= c(\gamma(x,t;s),s)\\
\gamma(x,t;t) &= x.
\end{align*}
Then $u$ is given in terms of $\gamma$ by $ u(x,t) = g(\gamma(x,t;0)) \exp(-\int_0^t
b(\gamma(x,t;r),r)\, dr)$. Writing $g$ as the inverse of its Fourier transform we
obtain the Fourier integral representation
\begin{equation}\label{hypsolu}
  u(x,t) = \iint \esp^{i(\gamma(x,t;0)-y) \xi}\; a(x,t,y,\xi)\, g(y)\, dy\, \dslash\xi,
\end{equation}
where $a(x,t,y,\xi) := \exp(-\int_0^t b(\gamma(x,t;r),r)\, dr)$ is a generalized
amplitude of order $0$. The phase function $\phi(x,t,y,\xi) :=
(\gamma(x,t;0)-y) \xi$ has (full) gradient
$$(\d_x\gamma(x,t;0),\d_t\gamma(x,t;0),-\xi,\gamma(x,t;0)-y)$$ and thus defines a generalized phase
function $\phi$. Therefore (\ref{hypsolu}) reads $u = A g$ where  $A : \Gc(\R) \to
\G(\R^2)$ is a generalized Fourier integral operator.
\end{ex}

\subsection*{Regularity properties}
We now investigate the regularity properties of the \emph{generalized Fourier integral operator} $A$. We will prove that for appropriate generalized phase functions and generalized amplitudes, $A$ maps $\Gcinf(\Om)$ into $\Ginf(\Om')$. The following example shows that a $\Ginf$-kind of regularity assumption for the net $(\phi_\eps)_\eps$ with respect to the parameter $\eps$ does not entail the desired mapping property.
\begin{ex}
Let $n=n'=p=1$ and $\Om=\Om'=\R$ and $\phi_\eps(x,y,\xi)=(x-\eps y)\xi$. Then $(\phi_\eps)_\eps\in\mM_{\Phi}(\R;\R\times\R)$ and in particular we have $N=0$ in all moderateness estimates (see Definition \ref{def_phase_x_moderate}$(i)$)) and $|\nabla_{y,\xi}\phi_\eps(x,y,\xi/|\xi|)|^2\ge\eps^2$. Choose the amplitude $a$ identically equal to $1$. The corresponding generalized operator $A$ does not map $\Gcinf(\R)$ into $\Ginf(\R)$. Indeed, for $0\neq f\in\Cinfc(\R)$ we have that
\[
A[(f)_\eps]=\biggl[\biggl(\int_{\R\times\R} \esp^{i(x-\eps y)\xi}f(y)\, dy\, \dslash\xi\biggr)_\eps\biggr]=[(\eps^{-1}f(x/\eps))_\eps]\in\G(\R)\setminus\Ginf(\R).
\]
\end{ex}
This example suggests that a stronger notion of regularity on generalized phase functions has to be designed. Such is provided by the concept of \emph{slow scale net}. 
\begin{defn}
\label{def_slow_phase}
We say that $\phi\in\wt{\Phi}(\Om';\Om\times\R^p)$ is a \emph{slow scale generalized phase function in the variables of $\Om\times\R^p$} if it has a representative $(\phi_\eps)_\eps$ fulfilling the conditions
\begin{itemize}
\item[(i)] $(\phi_\eps)_\eps\in\mM^\ssc_{S^1_{\rm{hg}}(\Om'\times\Om\times\R^p\setminus 0)}$, 
\item[(ii)] for all $K'\Subset\Om'$ and $K\Subset\Om$ the net \eqref{net_FIO} is slow scale-strictly nonzero.
\end{itemize}
\end{defn}
In the sequel the set of all $(\phi_\eps)_\eps\in\Phi[\Om';\Om\times\R^p]$ fulfilling $(i)$ and $(ii)$ in Definition \ref{def_slow_phase} will be denoted by $\mM^\ssc_{\Phi}(\Om';\Om\times\R^p)$ while we use $\wt{\Phi}^\ssc(\Om';\Om\times\R^p)$ for the set of slow scale generalized functions as above. Similarly, using $\nabla_{x,y,\xi}$ in place of $\nabla_{y,\xi}$ in $(ii)$ we define the space $\wt{\Phi}^{\ssc}(\Om'\times\Om\times\R^p)$ of slow scale generalized phase functions on $\Om'\times\Om\times\R^p$.  
We refer to \cite[Section 3]{GHO:06} for the proof of the following theorem.

\begin{thm}
\label{theorem_Ginf_map}
Let $\phi\in\wt{\Phi}^{\ssc}(\Om';\Om\times\R^p)$.
\begin{itemize}
\item[(i)] If $a\in\wt{\mathcal{S}}^{m,\ssc}_{\rho,\delta}(\Om'\times\Om\times\R^p)$ the corresponding generalized Fourier integral operator
\[
A:u\to\int_{\Om\times\R^p}\esp^{i\phi(x,y,\xi)}a(x,y,\xi)u(y)\, dy\, \dslash\xi
\]
maps $\Gcinf(\Om)$ continuously into $\Ginf(\Om')$.
\item[(ii)] If $a\in\wt{\mathcal{S}}^{-\infty,\ssc}(\Om'\times\Om\times\R^p)$ then $A$ maps $\Gc(\Om)$ continuously into $\Ginf(\Om')$.
\end{itemize}
\end{thm}
 
\subsection*{Extension to the dual}
Finally, we prove that under suitable hypotheses on the generalized phase function $\phi\in \wt{\Phi}(\Om'\times\Om\times\R^p)$, the definition of the generalized Fourier integral operator $A$ can be extended to the dual $\LL(\G(\Om),\wt{\C})$.
\begin{defn}
\label{definition_phaseop}
We say that $\phi\in\wt{\Phi}(\Om'\times\Om\times\R^p)$ is a generalized operator phase function if it has a representative $(\phi_\eps)_\eps$ of operator phase functions satisfying the conditions (i) and (ii) of Definition \ref{def_phase_x_moderate} and such that 
\begin{itemize}
\item[(iii)] for all $K'\Subset\Om'$ and $K\Subset\Om$ the net
\[
\biggl(\inf_{x\in K',y\in K,\xi\in\R^p\setminus 0}\biggl|\nabla_{x,\xi} \phi_\eps\biggl(x,y,\frac{\xi}{|\xi|}\biggr)\biggr|^2\biggr)_\eps
\]
is strictly nonzero.
\end{itemize}
\end{defn}
It is clear that when $\phi$ is a generalized operator phase function then by Theorem \ref{theorem_map} the oscillatory integral 
\begin{equation}
\label{eq_transposed}
\int_{\Om'\times\R^p} \esp^{i\phi(x,y,\xi)}a(x,y,\xi)v(x)\, dx\, \dslash\xi,
\end{equation}
where $a\in\wt{\mathcal{S}}^m_{\rho,\delta}(\Om'\times\Om\times\R^p)$ and $v\in\Gc(\Om')$, defines a generalized function in $\G(\Om)$ and a continuous $\wt{\C}$-linear operator from $\Gc(\Om')$ to $\G(\Om)$. More precisely, we have the following result.
\begin{prop}
\label{prop_extension}
Let $\phi$ be a generalized operator phase function on $\Om'\times\Om\times\R^p$, $a\in\wt{\mathcal{S}}^m_{\rho,\delta}(\Om'\times\Om\times\R^p)$ and $A:\Gc(\Om)\to\G(\Om')$ the generalized Fourier integral operator given by \eqref{oscillatory_x}-\eqref{def_A_fourier}. Then, 
\begin{itemize}
\item[(i)] the transposed ${\,}^tA$ of $A$ is the generalized Fourier integral operator given by \eqref{eq_transposed};
\item[(ii)] the operator $A$ can be extended to a continuous $\wt{\C}$-linear map acting from $\LL(\G(\Om),\wt{\C})$ to $\LL(\Gc(\Om'),\wt{\C})$. 
\end{itemize}
\end{prop}
\begin{proof}
Working at the level of representatives, the proof of the first assertion is a simple application of the corresponding classical result. It follows that $A$ can be extended to a $\wt{\C}$-linear map from $\LL(\G(\Om),\wt{\C})$ to $\LL(\Gc(\Om'),\wt{\C})$ by setting
\[
A(T)(u)= T({\,}^tAu),
\] 
for all $T\in\LL(\G(\Om),\wt{\C})$ and $u\in\Gc(\Om')$. Finally, let $B$ a bounded subset of $\Gc(\Om')$. From the continuity of ${\,}^tA$ and $T$ we have that
\[
\sup_{u\in B}|A(T)(u)|_\esp=\sup_{u\in B}|T({\,}^tAu)|=\sup_{v\in {\,}^tA(B)}|T(v)|,
\]
where ${\,}^tA(B)$ is a bounded subset of $\G(\Om)$. This shows that $A:\LL(\G(\Om),\wt{\C})\to\LL(\Gc(\Om'),\wt{\C})$ is continuous.
\end{proof}

\section{Composition of a generalized Fourier integral operator with a generalized pseudodifferential operator}
\label{section_comp}
\subsection*{Generalized Fourier integral operators of the type $F_\omega(b)$}
Let $\Om$ and $\Om'$ be open subsets of $\R^n$ and $\R^{n'}$ respectively. We now focus on operators of the form
\begin{equation}
\label{def_F_om}
F_\omega(b)(u)(x)=\int_{\R^n}\esp^{i\omega(x,\eta)}b(x,\eta)\widehat{u}(\eta)\, \dslash\eta,
\end{equation}
where $\omega\in\wt{\mathcal{S}}^1_{\rm{hg}}(\Om'\times\R^n\setminus 0)$, $b\in\wt{\mathcal{S}}^m(\Om'\times\R^n)$ and $u\in\Gc(\Om)$.\\
Note that $\phi(x,y,\eta):=\omega(x,\eta)-y\eta$ is a well-defined generalized phase function belonging to $\wt{\Phi}(\Om';\Om\times\R^n)$. Indeed, for any $(\omega_\eps)_\eps$ representative of $\omega$ we have that  $(\omega_\eps(x,\eta)-y\eta)_\eps\in\mM_{S^1_{\rm{hg}}(\Om'\times\Om\times\R^n)}$, if $(\omega_\eps-\omega'_\eps)_\eps\in\Neg_{S^1_{\rm{hg}}(\Om'\times\R^n)}$ then $(\omega_\eps-y\eta-\omega'_\eps+y\eta)_\eps\in\Neg_{S^1_{\rm{hg}}(\Om'\times\Om\times\R^n)}$ and $|\nabla_{y,\eta}\phi(x,y,\eta)|=|(-\eta,\nabla_{\eta}\omega-y)|\ge |\eta|$. In particular it follows that for any representative $\phi_\eps:=\omega_\eps(x,\eta)-y\eta$ and any $K'\Subset\Om'$, $K\Subset\Om$, the net $\displaystyle\inf_{x\in K',y\in K,\eta\in\R^n\setminus 0}|\nabla_{y,\eta}\phi_\eps\big(x,y,\frac{\eta}{|\eta|}\big)|$ is slow scale-strictly non-zero.

We recall that by Lemma \ref{lemma_esp_classic}, the estimate \eqref{coeff_a} and Proposition \ref{prop_exp}
\begin{trivlist}
\item[-] if $\omega\in\wt{\mathcal{S}}^1_{\rm{hg}}(\Om'\times\R^n\setminus 0)$ then $\esp^{i\omega(x,\eta)}\in\wt{\mathcal{S}}^1_{0,1}(\Om'\times\R^n)$ and 
\[
\partial^\alpha_\eta\partial^\beta_x\esp^{i\omega(x,\eta)}=\esp^{i\omega(x,\eta)} a_{\alpha,\beta}(x,\eta),
\]
where $a_{\alpha,\beta}\in \wt{\mathcal{S}}^{\,|\beta|}(\Om'\times\R^n\setminus 0)$ and the equality is intended in the space $\wt{\mathcal{S}}^{\,1+|\beta|}_{0,1}(\Om'\times\R^n\setminus 0)$;
\item[-] if $\omega\in\wt{\mathcal{S}}^{\,1,\ssc}_{\rm{hg}}(\Om'\times\R^n\setminus 0)$ then $a_{\alpha,\beta}\in \wt{\mathcal{S}}^{\,|\beta|,\ssc}(\Om'\times\R^n\setminus 0)$. 
\end{trivlist}

An immediate application of Theorem \ref{theorem_map} and Proposition \ref{prop_extension} yields the following mapping properties.
\begin{prop}
\label{prop_F_map}
\leavevmode
\begin{itemize}
\item[(i)] If $\omega\in\wt{\mathcal{S}}^1_{\rm{hg}}(\Om'\times\R^n\setminus 0)$ and $b\in\wt{\mathcal{S}}^m(\Om'\times\R^n)$ then $F_\omega(b)$ maps $\Gc(\Om)$ continuously into $\G(\Om')$.
\item[(ii)] If $\omega\in\wt{\mathcal{S}}^1_{\rm{hg}}(\Om'\times\R^n\setminus 0)$ has a representative $(\omega_\eps)_\eps\in\Phi[\Om'\times\R^n]$ such that for all $K'\Subset\Om'$
\[
\biggl(\inf_{x\in K',\eta\in\R^n\setminus 0}|\nabla_{x}\omega_\eps\big(x,\frac{\eta}{|\eta|}\big)|\biggr)_\eps
\]
is strictly non-zero, then $F_\omega(b)$ can be extended to a continuous $\wt{\C}$-linear map from $\LL(\G(\Om),\wt{\C})$ to $\LL(\Gc(\Om'),\wt{\C})$. 
\item[(iii)] If $\omega\in\wt{\mathcal{S}}^{\,1,\ssc}_{\rm{hg}}(\Om'\times\R^n\setminus 0)$ and $b\in\wt{\mathcal{S}}^{m,\ssc}(\Om'\times\R^n)$ then $F_\omega(b)$ maps $\Gcinf(\Om)$ continuously into $\Ginf(\Om')$. 
\item[(iv)] If $\supp_x b\Subset\Om'$ then $F_\omega(b)$ maps $\Gc(\Om)$ into $\Gc(\Om')$ and under the assumptions of {\rm{(ii)}} maps $\LL(\G(\Om),\wt{\C})$ into $\LL(\G(\Om'),\wt{\C})$.
\end{itemize}
\end{prop}
\begin{proof}
The first assertion is clear from Theorem \ref{theorem_map} and the second one from Proposition \ref{prop_extension}(ii).\\ 
(iii) Lemma \ref{lemma_esp_classic} and the considerations which precede this proposition entail 
\[
\partial^\beta_x F_{\omega}(b)(u)(x)=\sum_{\beta'\le\beta}\binom{\beta}{\beta'}\int_{\R^n}\esp^{i\omega(x,\eta)}a_{\beta'}(x,\eta)\partial^{\beta-\beta'}b(x,\eta)\widehat{u}(\eta),\ \dslash\eta,
\]
where $a_{\beta'}\in\wt{\mathcal{S}}^{|\beta'|,\ssc}(\Om'\times\R^n\setminus 0)$. Hence, if $b\in\wt{\mathcal{S}}^{\,m,\ssc}(\Om'\times\R^n)$ and $u\in\Gcinf(\Om)$ then $F_\omega(b)\in\Ginf(\Om')$. Moreover, since for all $\beta\in\N^{n'}$ and $K'\Subset\Om'$ there exists $h\in\N$ and $c>0$ such that for all $g\in\Cinf_{K}(\Om)$ and $\eps\in(0,1]$ the estimate
\[
\sup_{x\in K'}|\partial^\beta F_{\omega_\eps}(b_\eps)(g)(x)|\le c\max_{\beta'\le\beta}|a_{\beta',\eps}|^{(|\beta'|)}_{K',0} |b_\eps|^{(m)}_{K',|\beta|}\sup_{y\in K,|\gamma|\le h}|\partial^\gamma g(y)|,
\]
holds, we conclude that when $[(a_{\beta',\eps})_\eps]$ and $[(b_\eps)_\eps]$ are symbols of slow scale type then the map $F_\omega(b):\Gcinf(\Om)\to\Ginf(\Om')$ is continuous.\\
(iv) If $\supp_x b\Subset\Om'$ from the first assertion we have that $F_\omega(b)\in\Gc(\Om')$. Under the assumptions of $(ii)$ for the phase function $\omega$ we have that ${\,}^tF_\omega(b)$ maps $\G(\Om')$ continuously into $\G(\Om)$ and therefore $F_\omega(b)$ can be extended to a map from $\LL(\G(\Om),\wt{\C})$ to $\LL(\G(\Om'),\wt{\C})$.

\end{proof}
\begin{rem}
Taking $\Om=\R^n$ and noting that $\Gc(\Om')\subseteq\Gc(\R^n)$, it is clear that $F_\omega(b)$ maps $\Gc(\R^n)$ into $\Gc(\R^n)$ when $\supp_x b\Subset\Om'$. In addition, ${\,}^tF_\omega(b):\G(\R^n)\to\G(\R^n)$ and $F_\omega(b):\LL(\G(\R^n),\wt{\C})\to\LL(\G(\R^n),\wt{\C})$.
\end{rem}
In the sequel we assume $\Om=\Om'\subseteq\R^n$. Our main purpose is to investigate the composition $a(x,D)\circ F_\omega(b)$, where $a(x,D)$ is a generalized pseudodifferential operator and $F_\omega(b)$ a generalized Fourier integral operator as in \eqref{def_F_om}. This requires some technical preliminaries.

\subsection*{Technical preliminaries}
The proof of the following lemma can be found in \cite[Lemmas A.11, A.12]{Coriasco:98}.
\begin{lem}
\label{lemma_sandro_1}
Let $a\in\Cinf(\Om\times\R^n\setminus 0)$ and $\omega\in\Cinf(\Om\times\R^n\setminus 0)$. Then,
\begin{multline*}
\partial^\alpha_x\partial^\sigma_\eta( a(x,\nabla_x\omega(x,\eta))=\sum_{\sigma'\le\sigma}\binom{\sigma}{\sigma'}\sum_{|\beta+\gamma|\le|\alpha|}\sum_{|\sigma''|\le|\sigma'|}\partial^\beta_x\partial^{\gamma+\sigma''}_\eta a(x,\nabla_x\omega(x,\eta))\, \cdot\\
\cdot P^{\sigma'}_{\eta,\sigma''}(x,\eta)\partial^{\sigma-\sigma'}_\eta P^\alpha_{x\beta\gamma}(x,\eta),
\end{multline*}
where
\[
\begin{split}
P^{\sigma'}_{\eta,\sigma''}(x,\eta)&=1\quad\qquad\qquad\qquad\qquad\qquad\qquad\qquad\qquad\qquad\quad \text{if $\sigma'=0$},\\
P^{\sigma'}_{\eta,\sigma''} &=\sum_{\substack{{\delta_1,...,\delta_q}\\ s_1,...,s_q}}c^{s_1,...,s_q}_{\delta_1,...,\delta_q}\,\partial^{\delta_1}_\eta \partial_{x_{s_1}}\omega(x,\eta)...\partial^{\delta_q}_\eta\partial_{x_{s_q}}\omega(x,\eta)\quad \text{otherwise},
\end{split}
\]
with $q=|\sigma''|$, $\sum_{j=1}^q|\delta_j|=|\sigma'|$ and
\[
\begin{split}
P^\alpha_{x\beta\gamma}&=1\quad\qquad\qquad\qquad\qquad\qquad\qquad\qquad\qquad\qquad\quad \text{if $\gamma=0$},\\
P^\alpha_{x\beta\gamma} &=\sum_{\substack{{\delta_1,...,\delta_r}\\ s_1,...,s_r}}d^{s_1,...,s_r}_{\delta_1,...,\delta_r}\,\partial^{\delta_1}_x\partial_{x_{s_1}}\omega(x,\eta)...\partial^{\delta_r}_x\partial_{x_{s_r}}\omega(x,\eta)\quad \text{otherwise},
\end{split}
\]
with $|\gamma|=r$ and $\sum_{j=1}^r|\delta_j|+|\beta|=|\alpha|$.
\end{lem}
\begin{prop}
\label{prop_a_eps}
\leavevmode
\begin{itemize}
\item[(h1)] Let $(\omega_\eps)_\eps\in\mM_{S^1_{\rm{hg}}}(\Om\times\R^n\setminus 0)$ such that $\nabla_x\omega_\eps\neq 0$ for all $\eps\in(0,1]$ and for all $K\Subset\Om$
\[
\biggl(\inf_{x\in K,\eta\in\R^n\setminus 0}\biggl|\nabla_{x}\omega_\eps\big(x,\frac{\eta}{|\eta|}\big)\biggr|\biggr)_\eps
\]
is strictly non-zero.
\item[(i)] If $(a_\eps)_\eps\in \mM_{S^m(\Om\times\R^n\setminus 0)}$ then $(a_\eps(x,\nabla_x\omega_\eps(x,\eta)))_\eps\in \mM_{S^m(\Om\times\R^n\setminus 0)}$;
\item[(ii)] if $(a_\eps)_\eps\in \Neg_{S^m(\Om\times\R^n\setminus 0)}$ then $(a_\eps(x,\nabla_x\omega_\eps(x,\eta)))_\eps\in \Neg_{S^m(\Om\times\R^n\setminus 0)}$.\\
\item[(h2)] Let $(\omega_\eps)_\eps\in\mM^{\ssc}_{S^1_{\rm{hg}}}(\Om\times\R^n\setminus 0)$ such that $\nabla_x\omega_\eps\neq 0$ for all $\eps\in(0,1]$ and for all $K\Subset\Om$
\[
\biggl(\inf_{x\in K,\eta\in\R^n\setminus 0}\biggl|\nabla_{x}\omega_\eps\big(x,\frac{\eta}{|\eta|}\big)\biggr|\biggr)_\eps
\]
is slow scale strictly non-zero.
\item[(iii)] If $(a_\eps)_\eps\in \mM^{\ssc}_{S^m(\Om\times\R^n\setminus 0)}$ then $(a_\eps(x,\nabla_x\omega_\eps(x,\eta)))_\eps\in \mM^{\ssc}_{S^m(\Om\times\R^n\setminus 0)}$.
\item[(h3)] Finally, let $(\omega_\eps-\omega'_\eps)_\eps\in\Neg_{S^1_{\rm{hg}}}(\Om\times\R^n\setminus 0)$ with $(\omega_\eps)_\eps$ and $(\omega'_\eps)_\eps$ satisfying the hypothesis $(h1)$ above. 
\item[(iv)] If $(a_\eps)_\eps\in \mM_{S^m(\Om\times\R^n\setminus 0)}$ then 
\[
(a_\eps(x,\nabla_x\omega_\eps(x,\eta))-a_\eps(x,\nabla_x\omega'_\eps(x,\eta)))_\eps\in \Neg_{S^m(\Om\times\R^n\setminus 0)}.
\]
\end{itemize}
\end{prop}
\begin{proof}
From Lemma \ref{lemma_sandro_1} it follows that $\partial^\alpha_x\partial^\sigma_\eta( a_\eps(x,\nabla_x\omega_\eps(x,\eta))$ is a finite sum of terms of the type
\[
\partial^{\alpha'}_x\partial^{\sigma'}_\eta a_\eps(x,\nabla_x\omega_\eps(x,\eta))g_{\alpha',\sigma',\eps}(x,\eta),
\]
where $(g_{\alpha',\sigma',\eps})_\eps$ is a net of symbols in $S^{|\sigma'|-|\sigma|}(\Om\times\R^n\setminus 0)$. Note that $(g_{\alpha',\sigma',\eps})_\eps$ depends on $(\omega_\eps)_\eps$ and is actually a finite sum of products of derivatives of $(\omega_\eps)_\eps$. One can easily prove that 
\begin{equation}
\label{g}
\begin{split}
(\omega_\eps)_\eps\in\mM_{S^1_{\rm{hg}}}(\Om\times\R^n\setminus 0)\quad &\Rightarrow \quad (g_{\alpha',\sigma',\eps})_\eps\in\mM_{S^{|\sigma'|-|\sigma|}(\Om\times\R^n\setminus 0)},\\
(\omega_\eps)_\eps\in\mM^\ssc_{S^1_{\rm{hg}}}(\Om\times\R^n\setminus 0)\quad &\Rightarrow \quad (g_{\alpha',\sigma',\eps})_\eps\in\mM^{\ssc}_{S^{|\sigma'|-|\sigma|}(\Om\times\R^n\setminus 0)}.
\end{split}
\end{equation}
and that the following 
\begin{multline}
\label{a}
\forall\alpha',\sigma'\in\N^n\, \forall K\Subset\Om\, \exists (\lambda_\eps)_\eps\in\R^{(0,1]}\, \forall x\in K\, \forall\eta\in\R^n\setminus 0\, \forall\eps\in(0,1]\\
|\partial^{\alpha'}_x\partial^{\sigma'}_\eta a_\eps(x,\nabla_x\omega_\eps(x,\eta))|\le\lambda_\eps\lara{\nabla_x\omega_\eps}^{m-|\sigma'|}
\end{multline}
holds, with $(\lambda_\eps)_\eps\in\EM$ if $(a_\eps)_\eps\in\mM_{S^m(\Om\times\R^n\setminus 0)}$, $(\lambda_\eps)_\eps$ slow scale net if $(a_\eps)_\eps\in\mM^{\ssc}_{S^m(\Om\times\R^n\setminus 0)}$ and $(\lambda_\eps)_\eps\in\Neg$ if $(a_\eps)_\eps\in\Neg_{S^m(\Om\times\R^n\setminus 0)}$. 
Now, let us consider $(\nabla_x\omega_\eps(x,\eta))_\eps$. We have that
\begin{equation*}
\label{nabla}
\begin{array}{cc}
(h1)\quad \Rightarrow\quad &\forall K\Subset\Om\, \exists r>0\, \exists c_1,c_2>0\, \exists\eta\in(0,1]\, \forall x\in K\, \forall|\eta|\ge 1\, \forall\eps\in(0,\eta]\\[0.2cm]
&\lara{\eta}c_1\eps^r\le|\nabla_x\omega_\eps(x,\eta)|\le c_2\eps^{-r}\lara{\eta}, \\[0.3cm]
(h2)\quad \Rightarrow\quad &\forall K\Subset\Om\, \exists (\mu_\eps)_\eps\, {\rm{s.s.n}}\, \exists\eta\in(0,1]\, \forall x\in K\, \forall|\eta|\ge 1\, \forall\eps\in(0,\eta]\\[0.2cm]
&\lara{\eta}\mu_\eps^{-1}\le|\nabla_x\omega_\eps(x,\eta)|\le \mu_\eps\lara{\eta}.
\end{array}
\end{equation*}
Under the hypothesis $(h1)$, combining \eqref{g} with \eqref{a} we obtain the assertions $(i)$ and $(ii)$. Moreover, from the second implications of \eqref{g} and \eqref{a} we see that $(h2)$ yields $(iii)$. It remains to prove that if $(h3)$ holds and $(a_\eps)_\eps$ is a moderate net of symbols then $(a_\eps(x,\nabla_x\omega_\eps(x,\eta))-a_\eps(x,\nabla_x\omega'_\eps(x,\eta)))_\eps\in \Neg_{S^m(\Om\times\R^n\setminus 0)}$. If suffices to write $\partial^\alpha_x\partial^\sigma_\eta(a_\eps(x,\nabla_x\omega_\eps(x,\eta))-a_\eps(x,\nabla_x\omega'_\eps(x,\eta)))$ as the finite sum
\begin{multline}
\label{sum_compl}
\sum_{\alpha',\sigma'}\partial^{\alpha'}_x\partial^{\sigma'}_\eta a_\eps(x,\nabla_x\omega_\eps(x,\eta))(g_{\alpha',\sigma'}(\omega_\eps)-g_{\alpha',\sigma'}(\omega'_\eps))\\
+\sum_{\alpha',\sigma'}[\partial^{\alpha'}_x\partial^{\sigma'}_\eta a_\eps(x,\nabla_x\omega_\eps(x,\eta))-\partial^{\alpha'}_x\partial^{\sigma'}_\eta a_\eps(x,\nabla_x\omega'_\eps(x,\eta))]g_{\alpha',\sigma'}(\omega'_\eps).
\end{multline}
An inspection of Lemma \ref{lemma_sandro_1} shows that the net $(g_{\alpha',\sigma'}(\omega_\eps)-g_{\alpha',\sigma'}(\omega'_\eps))_\eps$ belongs to $\Neg_{S^{|\sigma'|-|\sigma|}(\Om\times\R^n\setminus 0)}$ and from the hypothesis $(h1)$ on $(\omega_\eps)_\eps$ it follows that the first summand in \eqref{sum_compl} is an element of $\Neg_{S^{m-|\sigma|}(\Om\times\R^n\setminus 0)}$. We use Taylor's formula on the second summand of \eqref{sum_compl}. Therefore, for $x$ varying in a compact set $K$ and for $\eps$ small enough we can estimate
\[
|\partial^{\alpha'}_x\partial^{\sigma'}_\eta a_\eps(x,\nabla_x\omega_\eps(x,\eta))-\partial^{\alpha'}_x\partial^{\sigma'}_\eta a_\eps(x,\nabla_x\omega'_\eps(x,\eta))|
\]
by means of 
\begin{multline*}
\sum_{j=1}^n\eps^{-N}\lara{\nabla_x\omega'_\eps(x,\eta)+\theta(\nabla_x\omega_\eps(x,\eta)-\nabla_x\omega'_\eps(x,\eta))}^{m-|\sigma'|-1}|\partial_{x_j}(\omega_\eps-\omega'_\eps)(x,\eta)|\\
\le \eps^q\lara{\nabla_x\omega'_\eps(x,\eta)+\theta(\nabla_x\omega_\eps(x,\eta)-\nabla_x\omega'_\eps(x,\eta))}^{m-|\sigma'|-1}\lara{\eta},
\end{multline*}
where $\theta\in[0,1]$. Since, taking $\eps$ small and $|\eta|\ge 1$ the following inequalities
\begin{multline*}
|{\nabla_x\omega'_\eps(x,\eta)+\theta(\nabla_x\omega_\eps(x,\eta)-\nabla_x\omega'_\eps(x,\eta))}|\ge \eps^r\lara{\eta}-\eps^{r+1}\lara{\eta}\ge\frac{\eps^{r}}{2}\lara{\eta},\\
|{\nabla_x\omega'_\eps(x,\eta)+\theta(\nabla_x\omega_\eps(x,\eta)-\nabla_x\omega'_\eps(x,\eta))}|\le \eps^{-r}\lara{\eta}
\end{multline*}
hold for some $r>0$, we conclude that 
\[
\big(\partial^{\alpha'}_x\partial^{\sigma'}_\eta a_\eps(x,\nabla_x\omega_\eps(x,\eta))-\partial^{\alpha'}_x\partial^{\sigma'}_\eta a_\eps(x,\nabla_x\omega'_\eps(x,\eta))\big)_\eps\in\Neg_{S^{m-|\sigma'|}(\Om\times\R^n\setminus 0)}.
\]
Thus, from $(g_{\alpha',\sigma'}(\omega'_\eps))_\eps\in\mM_{S^{|\sigma'|-|\sigma|}(\Om\times\R^n\setminus 0)}$ we have that the second summand of \eqref{sum_compl} belongs to $\Neg_{S^{m-|\sigma|}(\Om\times\R^n\setminus 0)}$ and the proof is complete.
\end{proof}
\begin{cor}
\label{coroll_a}
If $a\in\wt{\mathcal{S}}^m(\Om\times\R^n\setminus 0)$ and $\omega\in\wt{\mathcal{S}}^1_{\rm{hg}}(\Om\times\R^n\setminus 0)$ has a representative satisfying condition $(h1)$ of Proposition \ref{prop_a_eps} then $a(x,\nabla_x\omega(x,\eta))\in \wt{\mathcal{S}}^m(\Om\times\R^n\setminus 0)$.\\
If $a\in\wt{\mathcal{S}}^{m,\ssc}(\Om\times\R^n\setminus 0)$ and $\omega\in\wt{\mathcal{S}}^{1,\ssc}_{{\rm{hg}}}(\Om\times\R^n\setminus 0)$ has a representative satisfying condition $(h2)$ of Proposition \ref{prop_a_eps}, then $a(x,\nabla_x\omega(x,\eta))\in \wt{\mathcal{S}}^{m,\ssc}(\Om\times\R^n\setminus 0)$.
\end{cor}
Let $\omega\in\wt{\mathcal{S}}^1_{\rm{hg}}(\Om\times\R^n\setminus 0)$ have a representative satisfying $(h1)$. We want to investigate the properties of 
\begin{equation}
\label{esp}
D^\beta_z\big(\esp^{i\overline{\omega}(z,x,\eta)}\big)|_{z=x},
\end{equation}
where $\overline{\omega}(z,x,\eta):=\omega(z,\eta)-\omega(x,\eta)-\lara{\nabla_x\omega(x,\eta),z-x}$. We make use of the following technical lemma, whose proof can be found in \cite[Proposition 15]{Coriasco:99}.
\begin{lem}
\label{lemma_sandro_2}
Let $\omega\in\Cinf(\Om\times\R^n\setminus 0)$ and $\overline{\omega}(z,x,\eta)$ as above. Then, for $|\beta|\neq 0$,we have
\begin{multline*}
D^\beta_z\esp^{i\overline{\omega}(z,x,\eta)}=\esp^{i\overline{\omega}(z,x,\eta)}\biggl[(\nabla_z\omega(z,\eta)-\nabla_x\omega(x,\eta))^\beta\\
+\sum_{j_1}c_{j_1}(\nabla_z\omega(z,\eta)-\nabla_x\omega(x,\eta))^{\theta_{j_1}}\prod_{j_2=1}^{n_{1,j_1}}\partial^{\gamma_{j_1,j_2}}_z\omega(z,\eta)\\
+\sum_{j_1}c'_{j_1}\prod_{j_2=1}^{n_{2,j_1}}\partial^{\delta_{j_1,j_2}}_z\omega(z,\eta)\biggr],
\end{multline*}
where $c_{j_1}$, $c'_{j_1}$ are suitable constants, $|\gamma_{j_1,j_2}|\ge 2$, $|\delta_{j_1,j_2}|\ge 2$ and 
\[
\theta_{j_1}+\sum_{j_2=1}^{n_{1,j_1}}\gamma_{j_1,j_2}=\sum_{j_2=1}^{n_{2,j_1}}\delta_{j_1,j_2}=\beta.
\]
\end{lem}
It follows that
\begin{equation}
\label{formula_esp}
D^\beta_z\big(\esp^{i\overline{\omega}(z,x,\eta)}\big)|_{z=x}= \sum_{j_1}c_{j_1}\prod_{j_2=1}^{n_{1,j_1}}\partial^{\gamma_{j_1,j_2}}_x\omega(x,\eta)
+\sum_{j_1}c'_{j_1}\prod_{j_2=1}^{n_{2,j_1}}\partial^{\delta_{j_1,j_2}}_x\omega(x,\eta),
\end{equation}
with
\[
\sum_{j_2=1}^{n_{1,j_1}}\gamma_{j_1,j_2}=\sum_{j_2=1}^{n_{2,j_1}}\delta_{j_1,j_2}=\beta.
\]
Moreover, from $|\gamma_{j_1,j_2}|\ge 2$, $|\delta_{j_1,j_2}|\ge 2$ we have $|\beta|\ge 2n_{1,j_1}$ and $|\beta|\ge 2n_{2,j_1}$. 

Since the constants $c_{j_1}$, $c'_{j_1}$ do not depend on $\omega$, we can use the formula \eqref{formula_esp} in estimating the net $(D^\beta_z\big(\esp^{i\overline{\omega_\eps}(z,x,\eta)}\big)|_{z=x})_\eps$. 
\begin{prop}
\label{prop_esp}
\leavevmode
\begin{trivlist}
\item[(i)] If $\omega\in\wt{\mathcal{S}}^1_{\rm{hg}}(\Om\times\R^n\setminus 0)$ then \eqref{esp} is a well-defined element of $\wt{\mathcal{S}}^{|\beta|/2}(\Om\times\R^n\setminus 0)$.
\item[(ii)] If $\omega\in\wt{\mathcal{S}}^{1,\ssc}_{{\rm{hg}}}(\Om\times\R^n\setminus 0)$ then \eqref{esp} is a well-defined element of $\wt{\mathcal{S}}^{|\beta|/2,\ssc}(\Om\times\R^n\setminus 0)$.
\end{trivlist}
\end{prop}
\begin{proof}
From \eqref{formula_esp} we have that 
\[
\begin{split}
(\omega_\eps)_\eps\in\mM_{S^1_{\rm{hg}}}(\Om\times\R^n\setminus 0)\quad &\Rightarrow \quad (D^\beta_z\big(\esp^{i\overline{\omega_\eps}(z,x,\eta)}\big)|_{z=x})_\eps\in\mM_{S^{|\beta|/2}(\Om\times\R^n\setminus 0)},\\
(\omega_\eps)_\eps\in\mM^\ssc_{S^1_{\rm{hg}}}(\Om\times\R^n\setminus 0)\quad &\Rightarrow \quad (D^\beta_z\big(\esp^{i\overline{\omega_\eps}(z,x,\eta)}\big)|_{z=x})_\eps\in\mM^\ssc_{S^{|\beta|/2}(\Om\times\R^n\setminus 0)}.
\end{split}
\]
Noting that $(\omega_\eps-\omega'_\eps)_\eps\in\Neg_{S^1_{\rm{hg}}(\Om\times\R^n\setminus 0)}$ entails 
\[
\begin{split}
&\biggl(\prod_{j_2=1}^{n_{1,j_1}}\partial^{\gamma_{j_1,j_2}}_x\omega_\eps(x,\eta)-\prod_{j_2=1}^{n_{1,j_1}}\partial^{\gamma_{j_1,j_2}}_x\omega'_\eps(x,\eta)\biggr)_\eps\in\Neg_{S^{|\beta|/2}(\Om\times\R^n\setminus 0)},\\
&\biggl(\prod_{j_2=1}^{n_{2,j_1}}\partial^{\delta_{j_1,j_2}}_x\omega_\eps(x,\eta)-\prod_{j_2=1}^{n_{2,j_1}}\partial^{\delta_{j_1,j_2}}_x\omega'_\eps(x,\eta)\biggr)_\eps\in\Neg_{S^{|\beta|/2}(\Om\times\R^n\setminus 0)},
\end{split}
\]
we conclude that the net $\big(D^\beta_z\big(\esp^{i\overline{\omega_\eps}(z,x,\eta)}\big)|_{z=x}-D^\beta_z\big(\esp^{i\overline{\omega'_\eps}(z,x,\eta)}\big)|_{z=x}\big)_\eps$ belongs to $\Neg_{S^{|\beta|/2}(\Om\times\R^n\setminus 0)}$.
\end{proof}
By combining Corollary \ref{coroll_a} with Proposition \ref{prop_esp} we obtain the following statement.
\begin{prop}
\label{prop_h_alpha}
Let $\alpha\in\N^n$ and
\begin{equation}
\label{h_alpha}
h_\alpha(x,\eta)=\frac{\partial^\alpha_\xi a(x,\nabla_x\omega(x,\eta))}{\alpha!}D^\alpha_z\big(\esp^{i\overline{\omega}(z,x,\eta)}b(z,\eta)\big)|_{z=x}.
\end{equation}
\begin{trivlist}
\item[(i)] If $a\in\wt{\mathcal{S}}^m(\Om\times\R^n\setminus 0)$, $\omega\in\wt{\mathcal{S}}^1_{\rm{hg}}(\Om\times\R^n\setminus 0)$ has a representative satisfying condition $(h1)$ and $b\in\wt{\mathcal{S}}^l(\Om\times\R^n\setminus 0)$, then $h_\alpha\in\wt{\mathcal{S}}^{\,l+m-|\alpha|/2}(\Om\times\R^n\setminus 0)$ for all $\alpha$.
\item[(ii)] If $a\in\wt{\mathcal{S}}^{m,\ssc}(\Om\times\R^n\setminus 0)$, $\omega\in\wt{\mathcal{S}}^{\,1,\ssc}_{\rm{hg}}(\Om\times\R^n\setminus 0)$ has a representative satisfying condition $(h2)$ and $b\in\wt{\mathcal{S}}^{\,l,\ssc}(\Om\times\R^n\setminus 0)$, then $h_\alpha\in\wt{\mathcal{S}}^{\,l+m-|\alpha|/2,\ssc}(\Om\times\R^n\setminus 0)$ for all $\alpha$.
\end{trivlist}
\end{prop}
Our next task is to give a closer look to $\esp^{i\omega(x,\eta)}$.  
\begin{prop}
\label{prop_id_esp_1}
Let $\omega\in\wt{\mathcal{S}}^1_{\rm{hg}}(\Om\times\R^n\setminus 0)$ have a representative satisfying condition $(h1)$. Then for any positive integer $N$ there exists $p_N\in\wt{\mathcal{S}}^{-2N}(\Om\times\R^n\setminus 0)$ such that
\begin{equation}
\label{id_esp_1}
\esp^{i\omega(x,\eta)}=\biggl(p_N(x,\eta)\Delta_x^N+r(x,\eta)\biggr)\esp^{i\omega(x,\eta)},
\end{equation}
where $r\in\wt{\mathcal{S}}^{-\infty}(\Om\times\R^n\setminus 0)$.\\
If $\omega\in\wt{\mathcal{S}}^1_{\rm{hg}}(\Om\times\R^n\setminus 0)$ is of slow scale type and has a representative satisfying condition $(h2)$ then $p_N$ and $r$ are of slow scale type.
\end{prop}
\begin{proof}
Let $(\omega_\eps)_\eps$ be a representative of $\omega$ satisfying $(h1)$. We leave to the reader to prove by induction that 
\[
\Delta_x^N\big(\esp^{i\omega_\eps(x,\eta)}\big)=a_\eps(x,\eta)\esp^{i\omega_\eps(x,\eta)},
\]
where $(a_\eps)_\eps\in\mM_{S^{2N}(\Om\times\R^n\setminus 0)}$ with principal part given by 
\[
a_{2N,\eps}=(-1)^N|\nabla_x\omega_\eps(x,\eta)|^{2N}.
\]
From $(h1)$ we have that $\nabla_x\omega_\eps\neq 0$ for all $\eps\in(0,1]$ and for all $K\Subset\Om$ there exist $r>0$ and $\eps_0\in(0,1]$ such that $|\nabla_x\omega_\eps(x,\eta)|\ge\eps^r|\eta|$ for all $x\in K$, $\eta\neq 0$ and $\eps\in(0,\eps_0]$. Hence,
\[
|\nabla_x\omega_\eps(x,\eta)|\ge\frac{\eps^r}{2}\lara{\eta},
\]
for $|\eta|\ge 1$, $x\in K$ and $\eps\in(0,\eps_0]$. It follows from Proposition \ref{prop_ellip}$(iii)$ that $(a_\eps)_\eps$ is a net of elliptic symbols of $S^{2N}(\Om\times\R^n\setminus 0)$ such that for all $K\Subset\Om$ there exist $s\in \R$, $(R_\eps)_\eps$ strictly nonzero and $\eps_0\in(0,1]$ such that 
\[
|a_\eps(x,\eta)|\ge \eps^s\lara{\eta}^{2N},
\]
for $x\in K$, $|\eta|\ge R_\eps$ and $\eps\in(0,\eps_0]$. By Proposition \ref{prop_ellip_2}$(i)$ we find $(p_{N,\eps})_\eps\in\mM_{S^{-2N}(\Om\times\R^n\setminus 0)}$ and $(r_\eps)_\eps\in\mM_{S^{-\infty}(\Om\times\R^n\setminus 0)}$ such that
\begin{equation}
\label{eq_repr}
p_{N,\eps}a_\eps=1-r_\eps
\end{equation}
for all $\eps$. Therefore, 
\[
\esp^{i\omega_\eps(x,\eta)}=\biggl(p_{N,\eps}(x,\eta)\Delta_x^N+r_\eps(x,\eta)\biggr)\esp^{i\omega_\eps(x,\eta)}.
\]
This equality at the representatives'level implies the equality \eqref{id_esp_1} between equivalence classes of $\wt{\mathcal{S}}^{1}_{0,1}(\Om\times\R^n\setminus 0)$.

Now, let $\omega$ be a slow scale symbol with a representative $(\omega_\eps)_\eps$ satisfying condition $(h2)$. From Proposition \ref{prop_ellip}$(vi)$ we have that $(a_\eps)_\eps\in\mM^\ssc_{S^{2N}(\Om\times\R^n\setminus 0)}$ is a net of elliptic symbols such that for some $(s_\eps)_\eps$ inverse of a slow scale net, $(R_\eps)_\eps$ slow scale net and $\eps_0\in(0,1]$ the inequality
\[
|a_\eps(x,\eta)|\ge s_\eps\lara{\eta}^{2N},
\]
holds for all $x\in K$, $|\eta|\ge R_\eps$ and $\eps\in(0,\eps_0]$. Proposition \ref{prop_ellip_2}$(ii)$ shows that \eqref{eq_repr} is true for some $(p_{N,\eps})_\eps\in\mM^\ssc_{S^{-2N}(\Om\times\R^n\setminus 0)}$ and $(r_\eps)_\eps\in\mM^\ssc_{S^{-\infty}(\Om\times\R^n\setminus 0)}$.
\end{proof}
\subsection*{Main theorems}
The make use of the previous propositions in proving the main theorems of this section: Theorems \ref{theo_comp} and \ref{theo_comp_ssc}.
\begin{thm}
\label{theo_comp}
Let $\omega\in\wt{\mathcal{S}}^1_{\rm{hg}}(\Om\times\R^n\setminus 0)$ have a representative satisfying condition $(h1)$. Let $a\in\wt{\mathcal{S}}^m(\Om\times\R^n)$ and $b\in\wt{\mathcal{S}}^l(\Om\times\R^n\setminus 0)$ with $\supp_x\, b\Subset\Om$. Then, the operator $a(x,D)F_\omega(b)$ has the following properties:
\begin{itemize}
\item[(i)] maps $\Gc(\Om)$ into $\G(\Om)$ and $\LL(\G(\Om),\wt{\C})$ into $\LL(\G_c(\Om),\wt{\C})$;
\item[(ii)] is of the form 
\[
\int_{\R^n}\esp^{i\omega(x,\eta)}h(x,\eta)\widehat{u}(\eta)\, \dslash\eta +r(x,D)u,
\]
where $h\in\wt{\mathcal{S}}^{\,l+m}(\Om\times\R^n\setminus 0)$ has asymptotic expansion given by the symbols $h_\alpha$ defined in \eqref{h_alpha} and $r\in\wt{\mathcal{S}}^{-\infty}(\Om\times\R^n\setminus 0)$.
\end{itemize} 
\end{thm}
\begin{proof}
From Proposition \ref{prop_F_map}$(iv)$ is clear that $F_\om(b)$ maps $\Gc(\Om)$ and $\LL(\G(\Om),\wt{\C})$ into themselves respectively. We obtain $(i)$ combining this results with the usual mapping properties of a generalized pseudodifferential operator. We now have to investigate the composition 
\begin{multline*}
a(x,D)F_\om(b)u(x)=\int_{\Om\times\R^n}\esp^{i(x-z)\theta}a(x,\theta){F_\om(b)u}(z)\, dz\,\dslash\theta\\
= \int_{\Om\times\R^n}\esp^{i(x-z)\theta}a(x,\theta)\biggl(\int_{\R^n}\esp^{i\omega(z,\eta)}b(z,\eta)\widehat{u}(\eta)\, \dslash\eta\biggr)\, dz\,\dslash\theta\\
= \int_{\R^n}\int_{\Om\times\R^n}\esp^{i((x-z)\theta+\omega(z,\eta))}a(x,\theta)b(z,\eta)\, dz\, \dslash\theta\, \widehat{u}(\eta)\, \dslash\eta,
\end{multline*}
for $u\in\Gc(\Om)$. The last integral in $dz$ and $\dslash\theta$ is regarded as the oscillatory integral
\begin{equation}
\label{int_1_sum}
\int_{\Om\times\R^n}\esp^{i(x-z)\theta}a(x,\theta)b(z,\eta)\esp^{i\omega(z,\eta)}\, dz\, \dslash\theta,
\end{equation}
with $b(z,\eta)\esp^{i\omega(z,\eta)}\in\Gc(\Om_z)$.

In the sequel we will work at the level of representatives and we will follow the proof of Theorem 4.1.1 in \cite{MR:97}.\\
\bf{Step 1.}\rm\, Let$(\sigma_\eps)_\eps$ such that $\sigma_\eps\ge c\eps^s$ for some $c,s>0$ and for all $\eps\in(0,1]$. We take $\varphi\in\Cinf(\R^n)$ such that $\varphi(y)=1$ for $|y|\le 1/2$ and $\varphi(y)=0$ for $|y|\ge 1$ and we set
\[
b_\eps(z,\eta)= b'_\eps(z,x,\eta)+b''_\eps(z,x,\eta)= \varphi\big(\frac{x-z}{\sigma_\eps}\big)b_\eps(z,\eta)+\big(1-\varphi\big(\frac{x-z}{\sigma_\eps}\big)\big)b_\eps(z,\eta).
\]
We now write the integral in $dz$ and $\dslash\theta$ of \eqref{int_1_sum} as
\begin{multline*}
\int_{\Om\times\R^n}\esp^{i((x-z)\theta+\omega_\eps(z,\eta))}a_\eps(x,\theta)b'_\eps(z,x,\eta)\, dz\, \dslash\theta\\ +\int_{\Om\times\R^n}\esp^{i((x-z)\theta+\omega_\eps(z,\eta))}a_\eps(x,\theta)b''_\eps(z,x,\eta)\, dz\, \dslash\theta
:= I_{1,\eps}(x,\eta)+I_{2,\eps}(x,\eta)
\end{multline*}
and we begin to investigate the properties of $(I_{2,\eps})_\eps$. Proposition \ref{prop_id_esp_1} provides the identity 
\[
\esp^{i\omega_\eps(z,\eta)}=\biggl(p_{N,\eps}(z,\eta)\Delta_z^N+r_\eps(z,\eta)\biggr)\esp^{i\omega_\eps(z,\eta)},
\]
where $(p_{N,\eps})_\eps\in \mM_{S^{-2N}(\Om\times\R^n\setminus 0)}$ and $(r_\eps)_\eps\in\mM_{S^{-\infty}(\Om\times\R^n\setminus 0)}$, and allows us to write $(I_{2,\eps})_\eps$ as
\begin{multline*}
\int_{\Om\times\R^n}\esp^{i\omega_\eps(z,\eta)}\Delta_z^N\biggl(\esp^{i(x-z)\theta}p_{N,\eps}(z,\eta)b''_\eps(z,x,\eta)\biggr)a_\eps(x,\theta)\, dz\, \dslash\theta\\
+\int_{\Om\times\R^n}\esp^{i(x-z)\theta}a_\eps(x,\theta)b''_\eps(z,x,\eta)r_\eps(z,\eta)\esp^{i\omega_\eps(z,\eta)}\, dz\, \dslash\theta
:= I^1_{2,\eps}(x,\eta)+I^2_{2,\eps}(x,\eta)
\end{multline*}
The net $(I^2_{2,\eps})_\eps\in\mM_{S^{-\infty}}(\Om\times\R^n\setminus 0)$. Indeed, $I^2_{2,\eps}(x,\eta)=\int_{\R^n}\int_{\R^n}g_{\eps}(x,\eta,z,\theta)\, dz\, \dslash\theta$, where 
\[
g_\eps(x,\eta,z,\theta)=\esp^{i(x-z)\theta}a_\eps(x,\theta)b''_\eps(z,x,\eta)r_\eps(z,\eta)\esp^{i\omega_\eps(z,\eta)}
\]
and the following holds: for all $K\Subset\Om$, for all $\alpha\in\N^n$ and $d>0$ exist $N\in\N$ and $\eps_0\in(0,1]$ such that 
\[
\biggl|(i\theta)^\alpha\int_{\R^n}g_{\eps}(x,\eta,z,\theta)\, dz\biggr|\le \eps^{-N}\lara{\theta}^m\lara{\eta}^{-d}.
\]
This is due to the fact that $\supp_z b_\eps\subseteq K_b\Subset\Om$ for all $\eps$ and $(r_\eps)_\eps\in\mM_{S^{-\infty}(\Om\times\R^n\setminus 0)}$.\\
\bf{Step 2.}\rm\, By construction $b''_\eps(z,x,\eta)=0$ if $|x-z|\le\sigma_\eps/2$ for all $\eps\in(0,1]$. By making use of the identity
\[
\esp^{i(x-z)\theta}=|x-z|^{-2k}(-\Delta_\theta)^k\biggl(\esp^{i(x-z)\theta}\biggr)
\]
we have
\begin{multline*}
I^1_{2,\eps}(x,\eta)\\
= \int_{\Om\times\R^n}\hskip-15pt\esp^{i\omega_\eps(z,\eta)}\Delta_z^N\biggl(|x-z|^{-2k}(-\Delta_\theta)^k\biggl(\esp^{i(x-z)\theta}\biggr)p_{N,\eps}(z,\eta)b''_\eps(z,x,\eta)\biggr)a_\eps(x,\theta)\, dz\, \dslash\theta\\
=
\int_{\Om\times\R^n}\hskip-10pt\esp^{i\omega_\eps(z,\eta)}(-\Delta_\theta)^k a_\eps(x,\theta)\Delta_z^N\biggl(\esp^{i(x-z)\theta}|x-z|^{-2k}p_{N,\eps}(z,\eta)b''_\eps(z,x,\eta)\biggr)\, dz\, \dslash\theta.
\end{multline*}
It follows that for $x\in K\Subset\Om$ and $\eps$ small enough
\begin{multline*}
|I^1_{2,\eps}(x,\eta)|\le c\eps^{-N_a-N_p-N_b}(\sigma_\eps)^{-2k}\sum_{|\gamma|\le 2N}c_\gamma\sigma_\eps^{-|\gamma|}\int_{\R^n}\lara{\theta}^{m-2k+2N}\, d\theta\,\lara{\eta}^{-2N+l} \\
\le \eps^{-N'}\int_{\R^n}\lara{\theta}^{m-2k+2N}\, d\theta\,\lara{\eta}^{-2N+l}.
\end{multline*}
Hence, given $d>0$ and taking $N,k$ such that $-2N+l<-d$ and $m-2k+2N<-n$ we obtain that $(I^1_{2,\eps})_\eps$ is a moderate net of symbols of order $-\infty$ on $\Om\times\R^n\setminus 0$. Summarizing,
\[
\int_{\Om\times\R^n}\esp^{i((x-z)\theta+\omega_\eps(z,\eta))}a_\eps(x,\theta)b_\eps(z,\eta)\, dz\, \dslash\theta =I_{1,\eps}(x,\eta)+I^1_{2,\eps}(x,\eta)+I^2_{2,\eps}(x,\eta),
\]
where $(I^1_{2,\eps})_\eps$ and $(I^2_{2,\eps})_\eps$ belong to $\mM_{S^{-\infty}(\Om\times\R^n\setminus 0)}$.\\
\bf{Step 3.}\rm\, It remains to study 
\[
I_{1,\eps}(x,\eta)=\int_{\Om\times\R^n}\esp^{i((x-z)\theta+\omega_\eps(z,\eta))}a_\eps(x,\theta)b'_\eps(z,x,\eta)\, dz\, \dslash\theta.
\]
We expand $a_\eps(x,\theta)$ with respect to $\theta$ at $\theta=\nabla_x\omega_\eps(x,\eta)$ and we observe that 
\[
(\theta-\nabla_x\omega_\eps(x,\eta))^\alpha \esp^{i(x-z)(\theta-\nabla_x\omega_\eps(x,\eta))}=(-1)^{|\alpha|}D^\alpha_z\esp^{i(x-z)(\theta-\nabla_x\omega_\eps(x,\eta))}.
\]
By integrating by parts we obtain 
\[
\begin{split}
\esp^{-i\omega_\eps(x,\eta)}&I_{1,\eps}(x,\eta)\\
&=\int_{\Om\times\R^n}\hskip-10pt\esp^{i\overline{\omega}_\eps(z,x,\eta)}\esp^{i(x-z)(\theta-\nabla_x\omega_\eps(x,\eta))}b'_\eps(z,x,\eta)a_\eps(x,\theta)\, dz\, \dslash\theta\\
&=\sum_{|\alpha|< k}\frac{1}{\alpha!}\partial^\alpha_\xi a_\eps(x,\nabla_x\omega_\eps(x,\eta))\cdot\\
&\cdot\int_{\Om\times\R^n} D^\alpha_z\biggl(\esp^{i\overline{\omega}_\eps(z,x,\eta)}b'_\eps(z,x,\eta)\biggr)\esp^{i(x-z)(\theta-\nabla_x\omega_\eps(x,\eta))}\, dz\, \dslash\theta\\
&+\sum_{|\alpha|=k}\frac{k}{\alpha!}\int_{\Om\times\R^n}D^\alpha_z\biggl(\esp^{i\overline{\omega}_\eps(z,x,\eta)}b'_\eps(z,x,\eta)\biggr)\esp^{-i(x-z)\theta}r_{\alpha,\eps}(x,\eta,\theta)\, dz\, \dslash\theta,
\end{split}
\]
where $\overline{\omega}_\eps(z,x,\eta):=\omega_\eps(z,\eta)-\omega_\eps(x,\eta)-(\nabla_x\omega_\eps(x,\eta))(z-x)$ and
\[
r_{\alpha,\eps}(x,\eta,\theta)=\int_{0}^1(1-t)^{k-1}\partial^\alpha_\xi a_\eps(x,\nabla_x\omega_\eps(x,\eta)-t\theta)\, dt.
\]
Since, $b'_\eps(z,x,\eta)=b_\eps(x,\eta)$ if $|x-z|\le\frac{\sigma_\eps}{2}$ we have that 
\begin{multline*}
\int_{\Om\times\R^n}D^\alpha_z\biggl(\esp^{i\overline{\omega}_\eps(z,x,\eta)}b'_\eps(z,x,\eta)\biggr)\esp^{i(x-z)(\theta-\nabla_x\omega_\eps(x,\eta))}\, dz\, \dslash\theta\\ = \int_{\Om\times\R^n}D^\alpha_z\biggl(\esp^{i\overline{\omega}_\eps(z,x,\eta)}b'_\eps(z,x,\eta)\biggr)\esp^{i(x-z)\theta}\, dz\, \dslash\theta\\
=D^\alpha_z\biggl(\esp^{i\overline{\omega}_\eps(z,x,\eta)}b_\eps(z,\eta)\biggr)|_{x=z}
\end{multline*}
This means that 
\[
\esp^{-i\omega_\eps(x,\eta)}I_{1,\eps}(x,\eta)=\sum_{|\alpha|< k}h_{\alpha,\eps}(x,\eta)+\sum_{|\alpha|=k}\frac{k}{\alpha!}R_{\alpha,\eps}(x,\eta),
\]
where $(h_{\alpha,\eps})_\eps$ is defined in \eqref{h_alpha} and 
\[
R_{\alpha,\eps}(x,\eta):=\int_{\Om\times\R^n}D^\alpha_z\biggl(\esp^{i\overline{\omega}_\eps(z,x,\eta)}b'_\eps(z,x,\eta)\biggr)\esp^{-i(x-z)\theta}r_{\alpha,\eps}(x,\eta,\theta)\, dz\, \dslash\theta.
\]
\bf{Step 4.}\rm\, Our next task is to prove moderate symbol estimates for the net $(R_{\alpha,\eps})_\eps$. Let $\chi\in\Cinfc(\R^n)$ such that $\chi(\theta)=1$ for $|\theta|\le 1$ and $\chi(\theta)=0$ for $|\theta|\ge 3/2$. Let us take a positive net $(\tau_\eps)_\eps$ such that $\tau_\eps\ge c\eps^r$ for some $c>0$ and $r>0$. We define the sets
\[
W^1_{\tau_\eps,\eta}=\{\theta\in\R^n:\, |\theta|<\tau_\eps|\eta|\},\qquad W^2_{\tau_\eps,\eta}=\R^n\setminus W^1_{\tau_\eps,\eta}.
\]
Set now $\chi_\eps(\theta):=\chi(\theta/\tau_\eps)$. By construction we have that $\chi_\eps(\theta/|\eta|)=1$ on $W^1_{\tau_\eps,\eta}$, $\supp\, \chi_\eps(\cdot/|\eta|)\subseteq W^1_{2\tau_\eps,\eta}$ and $\supp(1-\chi_\eps(\cdot/|\eta|))\subseteq W^2_{\tau_\eps,\eta}$. We write $R_{\alpha,\eps}(x,\eta)$ as
\begin{multline*}
\int_{\Om\times\R^n}D^\alpha_z\biggl(\esp^{i\overline{\omega}_\eps(z,x,\eta)}b'_\eps(z,x,\eta)\biggr)\esp^{-i(x-z)\theta}r_{\alpha,\eps}(x,\eta,\theta)\chi_\eps(\theta/|\eta|)\, dz\, \dslash\theta\\
+ \int_{\Om\times\R^n}\hskip-10pt D^\alpha_z\biggl(\esp^{i\overline{\omega}_\eps(z,x,\eta)}b'_\eps(z,x,\eta)\biggr)\esp^{-i(x-z)\theta}r_{\alpha,\eps}(x,\eta,\theta)(1-\chi_\eps(\theta/|\eta|))\, dz\, \dslash\theta\\
:= R_{\alpha,\eps}^1(x,\eta)+R^2_{\alpha,\eps}(x,\eta).
\end{multline*}
We begin by estimating the net $R^1_{\alpha,\eps}$. We make use of the identity
\[
\esp^{-i(x-z)\theta}=(1+|\eta|^2|x-z|^2)^{-N}(1-|\eta|^2\Delta_\theta)^N\esp^{-i(x-z)\theta}
\]
which yields
\begin{multline*}
R_{\alpha,\eps}^1(x,\eta)=\int_{\Om\times\R^n}D^\alpha_z\biggl(\esp^{i\overline{\omega}_\eps(z,x,\eta)}b'_\eps(z,x,\eta)\biggr)\esp^{-i(x-z)\theta}\cdot\\
\cdot(1+|\eta|^2|x-z|^2)^{-N}(1-|\eta|^2\Delta_\theta)^N\big(r_{\alpha,\eps}(x,\eta,\theta)\chi_\eps(\theta/|\eta|)\big)\, dz\, \dslash\theta. 
\end{multline*}
By the moderateness of the net $(\omega_\eps)$ and Taylor's formula we have the inequality
\begin{equation}
\label{form_ruly}
|\nabla_x\omega_\eps(x,\eta)-\nabla_z\omega_\eps(z,\eta)|\le c\eps^{-M}|{\eta}|^1|x-z|,
\end{equation}
valid for $z\in K_b$, $|x-z|\le\sigma_\eps$ and $\sigma_\eps$ small enough such that $\cup_{\eps\in(0,1]}\{z+\lambda(x-z):\, z\in K_b,\, |x-z|\le\sigma_\eps,\, \lambda\in[0,1]\}\subseteq K'\Subset\Om$. Clearly $M$ depends on the compact set $K'$. By Lemma \ref{lemma_sandro_2} we have that for all $x$ and $z$ as above, $|\eta|\ge 1$ and $\eps\in(0,\eps_0]$, the estimate
\[
\biggl|D^\beta_z\esp^{i\overline{\omega}_\eps(z,x,\eta)}\biggr|
\le c'\eps^{-M'}\lara{\eta}^{\frac{|\beta|}{2}}(1+|\eta|^2|x-z|^2)^{L_\beta}
\]
holds for some $L_\beta\in\N$ and $M'\in\N$. Hence, recalling that $\sigma_\eps\ge c\eps^s$ for some $c,s>0$, we are led from the previous considerations to 
\begin{equation}
\label{ineq_D_alpha}
\biggl|D^\alpha_z\biggl(\esp^{i\overline{\omega}_\eps(z,x,\eta)}b'_\eps(z,x,\eta)\biggr)\biggr|\le C\eps^{-N'}\lara{\eta}^{l+\frac{|\alpha|}{2}}(1+|\eta|^2|x-z|^2)^{L_\alpha},
\end{equation}
valid for $|\eta|\ge 1$, $\eps$ small enough and $N'$ depending on $K_b$, $\alpha$ and the bound $c\eps^s$ of $\sigma_\eps$. Before considering $(1-|\eta|^2\Delta_\theta)^N\big(r_{\alpha,\eps}(x,\eta,\theta)\chi_\eps(\theta/|\eta|)$ it is useful to investigate the quantity $|\nabla_x\omega_\eps(x,\eta)-t\theta|$ for $x\in K\Subset\Om$ and $\theta\in W^1_{2\tau_\eps,\eta}$. We recall that there exists $r>0$, $c_0,c_1>0$ and $\eps_0\in(0,1]$ such that 
\[
c_0\eps^r|\eta|\le |\nabla_x\omega_\eps(x,\eta)|\le c_1\eps^{-r}|\eta|,
\]
for all $x\in K$, $\eta\neq 0$ and $\eps\in(0,\eps_0]$. Since, if $\theta\in W^1_{2\tau_\eps,\eta}$ then $|\theta|\le 2\tau_\eps|\eta|$, we obtain, for all $x\in K$, $\theta\in W^1_{2\tau_\eps,\eta}$ and $t\in[0,1]$, the following estimates:
\begin{multline*}
|\nabla_x\omega_\eps(x,\eta)-t\theta|\le |\nabla_x\omega_\eps(x,\eta)|+|\theta|\le (1+2\tau_\eps c_0^{-1}\eps^{-r})|\nabla_x\omega_\eps(x,\eta)|\\
|\nabla_x\omega_\eps(x,\eta)-t\theta|\ge (1-2\tau_\eps c_0^{-1}\eps^{-r})|\nabla_x\omega_\eps(x,\eta)|.
\end{multline*}
It follows that assuming $\tau_\eps\le\frac{\eps^r}{4c_0^{-1}}$ the inequality
\begin{equation}
\label{est_tau}
\frac{c_0}{2}\eps^r|\eta|\le \frac{1}{2}|\nabla_x\omega_\eps(x,\eta)|\le |\nabla_x\omega_\eps(x,\eta)-t\theta|\le \frac{3}{2}|\nabla_x\omega_\eps(x,\eta)|\le c_1\frac{3}{2}\eps^{-r}|\eta|
\end{equation}
holds for $x\in K$, $\eta\neq 0$, $\theta\in W^1_{2\tau_\eps,\eta}$, $t\in[0,1]$ and $\eps$ small enough. We make use of \eqref{est_tau} in estimating $(1-|\eta|^2\Delta_\theta)^N\big(r_{\alpha,\eps}(x,\eta,\theta)\chi_\eps(\theta/|\eta|)\big)$ and we conclude that for all $N\in\N$ there exists $N''$ such that 
\begin{equation}
\label{est_Delta_N}
|(1-|\eta|^2\Delta_\theta)^N\big(r_{\alpha,\eps}(x,\eta,\theta)\chi_\eps(\theta/|\eta|)\big)|\le \eps^{-N''}\lara{\eta}^{m-|\alpha|},
\end{equation}
for all $x\in K$, $\eta\neq 0$, $\theta\in W^1_{2\tau_\eps,\eta}$ and $\eps\in(0,\eps_0]$. A combination of \eqref{ineq_D_alpha} with \eqref{est_Delta_N} entails
\[
|R^1_{\alpha,\eps}(x,\eta)|\le \eps^{-N'-N''}\lara{\eta}^{m+l-\frac{|\alpha|}{2}}\int_{W^1_{2\tau_\eps,\eta}}d\theta \int_{\R^n}(1+|\eta|^2|y|^2)^{L_\alpha-N}\, dy.
\]
Therefore, choosing $N\ge L_\alpha+\frac{n+1}{2}$ we obtain
\begin{multline*}
|R^1_{\alpha,\eps}(x,\eta)|\le c\eps^{-N'-N''}(\tau_\eps)^n\lara{\eta}^{m+l-\frac{|\alpha|}{2}}|\eta|^n\int_{\R^n}\lara{z}^{-n-1}\, dz |\eta|^{-n}\\
\le c_1\eps^{-N_1}\lara{\eta}^{m+l-\frac{|\alpha|}{2}},
\end{multline*}
for $x\in K$ and $|\eta|\ge 1$.\\
The case $|\eta|\le 1$ requires less precise estimates. More precisely, it is enough to see that from Lemma \ref{lemma_sandro_2} we have that for all $\alpha$ there exists some $d\in\R$ such that 
\[
\biggl|D^\alpha_z\biggl(\esp^{i\overline{\omega}_\eps(z,x,\eta)}b'_\eps(z,x,\eta)\biggr)\biggr|\le C\eps^{-M'}\lara{\eta}^{d}
\]
for all $x\in K$, $z\in K_b$ and $|x-z|\le \sigma_\eps$. Thus, 
\[
|R^1_{\alpha,\eps}(x,\eta)|\le c\eps^{N_2}\lara{\eta}^{d-l-\frac{|\alpha|}{2}}\lara{\eta}^{m+l-\frac{|\alpha|}{2}}\le c_2\eps^{N_2}\lara{\eta}^{m+l-\frac{|\alpha|}{2}},
\]
when $|\eta|\le 1$. In conclusion, there exists $(C_{K,\eps})_\eps\in\EM$ such that
\[
|R^1_{\alpha,\eps}(x,\eta)|\le C_{K,\eps}\lara{\eta}^{m+l-\frac{|\alpha|}{2}}
\]
for all $x\in K$, $\eta\in\R^n\setminus 0$ and $\eps\in(0,1]$.\\
\bf{Step 5.}\rm\, Finally, we consider $R^2_{\alpha,\eps}(x,\eta)$. By Lemma \ref{lemma_sandro_2} we can write 
\[
D^\alpha_z\biggl(\esp^{i\overline{\omega}_\eps(z,x,\eta)}b'_\eps(z,x,\eta)\biggr)
\]
as the finite sum
\[
\esp^{i\overline{\omega}_\eps(z,x,\eta)}\sum_{\beta}b_{\alpha,\beta,\eps}(z,x,\eta),
\]
where, by making use of the hypotheses on $b'_{\eps}$ and $\sigma_\eps$, the following holds:
\begin{multline*}
\forall \beta\in\N^n\, \exists m_\beta\in\R\, \forall \gamma\in\N^n\, \exists (\mu_{\beta,\gamma,\eps})_\eps\in\EM\, \forall x\in\Om\, \forall z\in\Om\, \forall \eta\in\R^n\setminus 0\, \forall\eps\in(0,1]\\
|\partial^\gamma_z b_{\alpha,\beta,\eps}(z,x,\eta)|\le \mu_{\beta,\gamma,\eps}\lara{\eta}^{m_\beta},
\end{multline*}
with $b_{\alpha,\beta,\eps}(z,x,\eta)=0$ for $|x-z|\ge \sigma_\eps$.
Hence, we have
\begin{multline*}
R^2_{\alpha,\eps}(x,\eta)\\
=\sum_\beta \int_{\Om\times\R^n}\esp^{i\overline{\omega}_\eps(z,x,\eta)}b_{\alpha,\beta,\eps}(z,x,\eta)\esp^{-i(x-z)\theta}r_{\alpha,\eps}(x,\eta,\theta)(1-\chi_\eps(\theta/|\eta|))\, dz\, \dslash\theta\\
=\sum_\beta \int_{\R^n}\esp^{-ix\theta}r_{\alpha,\eps}(x,\eta,\theta)(1-\chi_\eps(\theta/|\eta|))\int_\Om \esp^{i\rho_\eps(z,x,\eta,\theta)}b_{\alpha,\beta,\eps}(z,x,\eta)\, dz\, \dslash\theta,
\end{multline*}
where $\rho_\eps(z,x,\eta,\theta)=\overline{\omega}_\eps(z,x,\eta)+z\theta$. Since $\chi_\eps(\theta/|\eta|)=1$ for $\theta\in W^1_{\tau_\eps,\eta}$, we may limit ourselves to consider $\theta\in W^2_{\tau_\eps,\eta}$, i.e., $|\theta|\ge\tau_\eps|\eta|$. We investigate now the properties of the net $(\rho_\eps)_\eps$. We have
\[
\nabla_z\rho_\eps(z,x,\eta,\theta)=\theta+\nabla_z\omega_\eps(z,\eta)-\nabla_x\omega_\eps(x,\eta),
\]
and therefore \eqref{form_ruly} yields
\[
|\theta+\nabla_z\omega_\eps(z,\eta)-\nabla_x\omega_\eps(x,\eta)|\le |\theta|+\eps^{-M}\sigma_\eps|\eta|\le |\theta|(1+\eps^{-M}\sigma_\eps\tau_\eps^{-1}) 
\]
for $\theta\in W^2_{\tau_\eps,\eta}$, $|x-z|<\sigma_\eps$, $z\in K_b$ and $\eps$ small enough. We now take $\sigma_\eps$ so small that $\eps^{-M}\sigma_\eps\le\frac{\tau_\eps}{2}$. From \eqref{form_ruly} and the previous assumptions we obtain
\[
|\theta+\nabla_z\omega_\eps(z,\eta)-\nabla_x\omega_\eps(x,\eta)|\ge |\theta|-\eps^{-M}\sigma_\eps|\eta|\ge |\theta|-\eps^{-M}\sigma_\eps\tau_\eps^{-1}|\theta|\ge \frac{1}{2}|\theta|.
\]
In other words, there exists $(\lambda_{1,\eps})_\eps\in\EM$ strictly nonzero and $(\lambda_{2,\eps})_\eps\in\EM$ such that 
\[
\lambda_{1,\eps}|\theta|\le|\theta+\nabla_z\omega_\eps(z,\eta)-\nabla_x\omega_\eps(x,\eta)|\le\lambda_{2,\eps}|\theta|,
\]
for $\theta\in W^2_{\tau_\eps,\eta}$, $|x-z|<\sigma_\eps$, $z\in K_b$ and $\eps\in(0,1]$.\\
Consider now
\[
p_{N,\eps}(z,x,\eta,\theta)=\esp^{-i\rho_\eps(z,x,\eta,\theta)}\Delta_z^N\esp^{i\rho_\eps(z,x,\eta,\theta)}.
\]
Noting that $\partial^\gamma_z \rho_\eps(z,x,\eta,\theta)=\partial^\gamma_z\omega_\eps(z,\eta)$ for $|\gamma|\ge 2$, and making use of the previous estimates on $|\nabla_z\rho_\eps(z,x,\eta,\theta)|$, one can prove by induction that 
\[
\Delta^N_z\esp^{i\rho_\eps(z,x,\eta,\theta)}=\esp^{i\rho_\eps(z,x,\eta,\theta)}\biggl((-1)^N|\nabla_z\rho_\eps(z,x,\eta,\theta)|^{2N}+s_{N,\eps}(z,x,\eta,\theta)\biggr),
\]
where $(s_{N,\eps})_\eps$ has the following property:
\begin{equation}
\label{prop_1_6}
\exists l\in [0,2N)\, \forall\gamma\in\N^n\, \exists (s'_{\gamma,N,\eps})_\eps\in\EM\qquad |\partial^\gamma_z s_{N,\eps}(z,x,\eta,\theta)|\le s'_{\gamma,N,\eps}|\theta|^l,
\end{equation}
for $|\eta|\ge 1$, $\theta\in W^2_{\tau_\eps,\eta}$, $|x-z|<\sigma_\eps$ and $z\in K_b$. It follows that
\begin{multline}
\label{prop_2_6}
|p_{N,\eps}(z,x,\eta,\theta)|\ge \frac{1}{2^{2N}}|\theta|^{2N}-s'_{0,N,\eps}|\theta|^l=|\theta|^{2N}\big(\frac{1}{2^{2N}}-s'_{0,N,\eps}|\theta|^{l-2N}\big)\\
\ge \frac{1}{2^{2N+1}}|\theta|^{2N},
\end{multline}
for $\theta\in W^2_{\tau_\eps,\eta}$, $|x-z|<\sigma_\eps$, $z\in K_b$ and $|\eta|\ge \lambda_{N,\eps}:= \tau_\eps^{-1}(2^{2N+1}s'_{0,N,\eps})^{\frac{1}{2N-l}}$. Moreover, we have that for all $\gamma\in\N^n$ there exists $(a_{\gamma,N,\eps})_\eps\in\EM$ such that
\begin{equation}
\label{prop_3_6}
|\partial^\gamma_z|\nabla_z\rho_\eps(z,x,\eta,\theta)|^{2N}|\le a_{\gamma,N,\eps}|\theta|^{2N},
\end{equation}
for $|\eta|\ge 1$, $\theta\in W^2_{\tau_\eps,\eta}$, $|x-z|<\sigma_\eps$ and $z\in K_b$. This allows us to prove by induction that 
\begin{equation}
\label{prop_4_6}
\forall\gamma\in\N^n\, \exists (b_{\gamma,N,\eps})_\eps\in\EM\, \exists (\lambda_{\gamma,N,\eps})_\eps\in\EM\,\qquad |\partial^\gamma_z p^{-1}_{N,\eps}(z,x,\eta,\theta)|\le b_{\gamma,N,\eps}|\theta|^{-2N},
\end{equation}
for $\theta\in W^2_{\tau_\eps,\eta}$, $|x-z|<\sigma_\eps$, $z\in K_b$ and $|\eta|\ge \lambda_{\gamma,N,\eps}$. The assertion \eqref{prop_4_6} is clear for $\gamma=0$ by \eqref{prop_2_6}. Assume now that \eqref{prop_4_6} holds for $|\gamma'|\le N$ and take $|\gamma|=N$. From $p_{N,\eps}^{-1}p_{N,\eps}=1$ we obtain 
\[
\partial^{\gamma}_zp^{-1}_{N,\eps}(z,x,\eta,\theta)p_{N,\eps}(z,x,\eta,\theta)=-\hskip-4pt\sum_{\gamma'<\gamma}\binom{\gamma}{\gamma'}\partial^{\gamma'}_z p^{-1}_{N,\eps}(z,x,\eta,\theta)\partial^{\gamma-\gamma'}_zp_{N,\eps}(z,x,\eta,\theta)
\] 
and therefore
\begin{multline*}
|\partial^{\gamma}_zp^{-1}_{N,\eps}(z,x,\eta,\theta)|\le \sum_{\gamma'<\gamma}b_{\gamma',N,\eps}|\theta|^{-2N}(a_{\gamma-\gamma',N,\eps}|\theta|^{2N}+s'_{\gamma-\gamma',N,\eps}|\theta|^l)|\theta|^{-2N}\\
\le b_{\gamma,N,\eps}|\theta|^{-2N},
\end{multline*}
for $\theta\in W^2_{\tau_\eps,\eta}$, $|x-z|<\sigma_\eps$, $z\in K_b$ and $|\eta|\ge\lambda_{\gamma,N,\eps}:= \max_{\gamma'<\gamma}\lambda_{\gamma',N,\eps}$. We make use of the identity
\[
\esp^{i\rho_\eps(z,x,\eta,\theta)}=\Delta^N_z\esp^{i\rho_\eps(z,x,\eta,\theta)}p^{-1}_{N,\eps}(z,x,\eta,\theta)
\]
in the integral
\[
\int_\Om \esp^{i\rho_\eps(z,x,\eta,\theta)}b_{\alpha,\beta,\eps}(z,x,\eta)\, dz.
\]
Since, $\supp_z b_{\alpha,\beta,\eps}(z,x,\eta)\subseteq K_b$, $b_{\alpha,\beta,\eps}(z,x,\eta)=0$ for $|x-z|\ge\sigma_\eps$ and $1-\chi_\eps(\theta/|\eta|)=0$ for $\theta\not\in W^2_{\tau_\eps,\eta}$, we can write
\begin{multline*}
\int_{\R^n}\esp^{-ix\theta}r_{\alpha,\eps}(x,\eta,\theta)(1-\chi_\eps(\theta/|\eta|))\int_\Om \esp^{i\rho_\eps(z,x,\eta,\theta)}b_{\alpha,\beta,\eps}(z,x,\eta)\, dz\, \dslash\theta\\
= \int_{\R^n}\esp^{-ix\theta}r_{\alpha,\eps}(x,\eta,\theta)(1-\chi_\eps(\theta/|\eta|))\int_\Om \esp^{i\rho_\eps(z,x,\eta,\theta)}	\cdot\\
\cdot\Delta^N_z\big(p_{N,\eps}^{-1}(z,x,\eta,\theta)b_{\alpha,\beta,\eps}(z,x,\eta)\big)\, dz\, \dslash\theta,
\end{multline*}
where
\begin{multline*}
\biggl|r_{\alpha,\eps}(x,\eta,\theta)(1-\chi_\eps(\theta/|\eta|))\int_\Om \esp^{i\rho_\eps(z,x,\eta,\theta)}\Delta^N_z\big(p_{N,\eps}^{-1}(z,x,\eta,\theta)b_{\alpha,\beta,\eps}(z,x,\eta)\big)\, dz\biggr|\\
\le c|r_{\alpha,\eps}(x,\eta,\theta)||1-\chi_\eps(\theta/|\eta|)|b_{N,\eps}\mu_{\beta,N,\eps}|\theta|^{-2N}\lara{\eta}^{m_\beta}, 
\end{multline*}
for $|\eta|\ge \lambda_{N,\eps}$ and $m_\beta$ independent of $N$. We take $2N=N_1+N_2$ such that $-N_2+m_\beta\le 0$. Hence, from $|\theta|\ge \tau_\eps|\eta|$ we have, for some $(c_{N,\eps})_\eps\in\EM$ and $|\eta|\ge\lambda_{N,\eps}$, the following estimate:
\begin{multline*}
\biggl|r_{\alpha,\eps}(x,\eta,\theta)(1-\chi_\eps(\theta/|\eta|))\int_\Om \esp^{i\rho_\eps(z,x,\eta,\theta)}\Delta^N_z\big(p_{N,\eps}^{-1}(z,x,\eta,\theta)b_{\alpha,\beta,\eps}(z,x,\eta)\big)\, dz\biggr|\\
\le c_{N,\eps}|r_{\alpha,\eps}(x,\eta,\theta)||1-\chi_\eps(\theta/|\eta|)||\theta|^{-N_1}.
\end{multline*}
By definition of $r_{\alpha,\eps}$ we easily see that for all $K\Subset\Om$ there exists $(d_\eps)_\eps, (d'_\eps)_\eps\in\EM$ such that 
\[
|r_{\alpha,\eps}(x,\eta,\theta)(1-\chi_\eps(\theta/|\eta|))|\le d_\eps\lara{\theta}^{m_+}\lara{\nabla_x\omega_\eps(x,\eta)}^{m_+}\le d'_\eps\lara{\theta}^{2m_+}.
\]
Hence for all $h\ge 0$ there exists $2N=N_1+N_2$ large enough such that, for $x\in K$ and $|\eta|\ge\lambda_{N,\eps}$
\begin{multline*}
\biggl|r_{\alpha,\eps}(x,\eta,\theta)(1-\chi_\eps(\theta/|\eta|))\int_\Om \esp^{i\rho_\eps(z,x,\eta,\theta)}\Delta^N_z\big(p_{N,\eps}^{-1}(z,x,\eta,\theta)b_{\alpha,\beta,\eps}(z,x,\eta)\big)\, dz\biggr|\\
\le \nu_{K,\eps} \lara{\theta}^{-h}\le \nu_{K,\eps}\tau_\eps^{-h}\lara{\eta}^{-h},
\end{multline*}
with $(\nu_{K,\eps})_\eps\in\EM$. This means that for all $h\ge 0$ there exists $(\lambda_\eps)_\eps\in\EM$ such that
\[
|R^2_{\alpha,\eps}(x,\eta)|\le \nu_{K,\eps}\lara{\eta}^{-h}
\]
when $x\in K$, $|\eta|\ge\lambda_{\eps}$ and $\eps\in(0,1]$. A simple investigation of the oscillatory integral which defines $R^2_{\alpha,\eps}(x,\eta)$ shows that there exists some $h'\ge 0$ and some $\nu'_{K,\eps}\in\EM$ such that the estimate
\[
|R^2_{\alpha,\eps}(x,\eta)|\le \nu'_{K,\eps}\lara{\eta}^{h'}
\]
holds for all $x\in K$, $\eta\in\R^n\setminus 0$ and $\eps\in(0,1]$. This yields for $|\eta|\le\lambda_\eps$
\[
|R^2_{\alpha,\eps}(x,\eta)|\le \nu'_{K,\eps}\lara{\eta}^{-h}\lara{\eta}^{h+h'}\le \nu'_{K,\eps}\lara{\lambda_\eps}^{h+h'}\lara{\eta}^{-h}.
\]
In conclusion, we have that for all $h\ge 0$ there exists $(C_{h,\eps}(K))_\eps\in\EM$ such that
\[
|R^2_{\alpha,\eps}(x,\eta)|\le C_{h,\eps}(K)\lara{\eta}^{-h}
\]
for all $x\in K$, $\eta\in\R^n\setminus 0$ and $\eps\in(0,1]$.\\
\bf{Step 6.}\rm\, Finally, we combine all the results of the previous steps. We have that
\begin{equation}
\label{form_finale}
a_\eps(x,D)F_{\om_\eps}(b_\eps)u_\eps(x)=\int_{\R^n}I_{1,\eps}(x,\eta)\widehat{u_\eps}(\eta)\, \dslash\eta + \int_{\R^n}I_{2,\eps}(x,\eta)\widehat{u_\eps}(\eta)\, \dslash\eta,
\end{equation}
where $(I_{2,\eps})_\eps\in\mM_{S^{-\infty}(\Om\times\R^n\setminus 0)}$. From Theorem \ref{theo_asymp_expan}$(i)$ and Proposition \ref{prop_h_alpha} there exists $(h_\eps)_\eps\in\mM_{S^{m+l}(\Om\times\R^n\setminus 0)}$ such that $h_\eps(x,\eta)\sim \sum_\alpha h_{\alpha,\eps}(x,\eta)$. We write the first integral in \eqref{form_finale} as
\[
\int_{\R^n}\esp^{i\omega_\eps(x,\eta)}h_\eps(x,\eta)\widehat{u_\eps}(\eta)\, \dslash\eta + \int_{\R^n}\esp^{i\omega_\eps(x,\eta)}\biggl(\esp^{-i\omega_\eps(x,\eta)}I_{1,\eps}(x,\eta)-h_\eps(x,\eta)\biggr)\widehat{u_\eps}(\eta)\, \dslash\eta
\]
and we concentrate on 
\[
\esp^{-i\omega_\eps(x,\eta)}I_{1,\eps}(x,\eta)-h_\eps(x,\eta).
\]
From the previous computations we have that for all $k\ge 1$ and $K\Subset\Om$ there exists $(C_{k,\eps}(K))_\eps\in\EM$ such that 
\[
\biggl|\esp^{-i\omega_\eps(x,\eta)}I_{1,\eps}(x,\eta)-\sum_{|\alpha|<k}h_{\alpha,\eps}(x,\eta)\biggr|\le C_{k,\eps}(K)\lara{\eta}^{m+l-\frac{k}{2}}
\]
for all $x\in K$, $\eta\in\R^n\setminus 0$ and $\eps\in(0,1]$. Moreover, regarding $I_{1,\eps}(x,\eta)$ as the oscillatory integral 
\[
\int_{\Om\times\R^n}\esp^{-iz\theta}\esp^{ix\theta+i\omega_\eps(z,\eta)}a_\eps(x,\theta)b'_\eps(z,x,\eta)\, dz\, \dslash\theta,
\]
from Theorem 3.1 in \cite{Garetto:04}, we obtain that for all $\alpha,\beta\in\N^n$ there exists $d\in\R$ and for all $K\Subset\Om$ there exists $(c_{\alpha,\beta,\eps}(K))_\eps\in\EM$ such that for all $\eta\in\R^n\setminus 0$ and $\eps\in(0,1]$,
\[
\sup_{x\in K}|\partial^\alpha_\eta\partial^\beta_x I_{1,\eps}(x,\eta)|\le c_{\alpha,\beta,\eps}(K)\lara{\eta}^{d}.
\]
Recalling that 
\[
\partial^\alpha_\eta\partial^\beta_x \esp^{-i\omega_\eps{(x,\eta)}}=\esp^{-i\omega_\eps(x,\eta)}a_{\alpha,\beta,\eps}(x,\eta),
\]
with $(a_{\alpha,\beta,\eps})_\eps\in \mM_{S^{|\beta|}(\Om\times\R^n\setminus 0)}$, we conclude that the net $(\esp^{-i\omega_\eps(x,\eta)}I_{1,\eps}(x,\eta))_\eps$ satisfies the hypothesis of Proposition \ref{prop_asym_Shubin}$(i)$. It follows that
\[
(\esp^{-i\omega_\eps(x,\eta)}I_{1,\eps}(x,\eta))_\eps\sim \sum_\alpha (h_{\alpha,\eps})_\eps.
\]
Hence, by Theorem \ref{theo_asymp_expan}$(i)$ we conclude 
\[
(\esp^{-i\omega_\eps(x,\eta)}I_{1,\eps}(x,\eta)-h_\eps(x,\eta))_\eps\in\mM_{S^{-\infty}(\Om\times\R^n\setminus 0)}.
\]
Going back to \eqref{form_finale} we have that there exists $(r_\eps)_\eps\in\mM_{S^{-\infty}(\Om\times\R^n\setminus 0)}$ such that
\[
a_\eps(x,D)F_{\om_\eps}(b_\eps)u_\eps(x)= \int_{\R^n}\esp^{i\omega_\eps(x,\eta)}h_\eps(x,\eta)\widehat{u_\eps}(\eta)\, \dslash\eta + r_\eps(x,D)(u_\eps)(x).
\]
\end{proof}
\begin{thm}
\label{theo_comp_ssc}
Let $\omega\in\wt{\mathcal{S}}^{\,1,\ssc}_{\rm{hg}}(\Om\times\R^n\setminus 0)$ have a representative satisfying condition $(h2)$. Let $a\in\wt{\mathcal{S}}^{\,m,\ssc}(\Om\times\R^n)$ and $b\in\wt{\mathcal{S}}^{\,l,\ssc}(\Om\times\R^n\setminus 0)$ with $\supp_x\, b\Subset\Om$. Then, the operator $a(x,D)F_\omega(b)$ has the following properties:
\begin{itemize}
\item[(i)] maps $\Gcinf(\Om)$ into $\Ginf(\Om)$;
\item[(ii)] is of the form 
\[
\int_{\R^n}\esp^{i\omega(x,\eta)}h(x,\eta)\widehat{u}(\eta)\, \dslash\eta +r(x,D)u,
\]
where $h\in\wt{\mathcal{S}}^{\,l+m,\ssc}(\Om\times\R^n\setminus 0)$ has asymptotic expansion given by the symbols $h_\alpha$ defined in \eqref{h_alpha} and $r\in\wt{\mathcal{S}}^{-\infty,\ssc}(\Om\times\R^n\setminus 0)$.
\end{itemize} 
\end{thm}
\begin{proof}
Combining Proposition \ref{prop_F_map}$(iii)$ with the usual mapping properties of generalized pseudodifferential operators we have that $(i)$ holds. Concerning assertion $(ii)$, we argue as in the proof of Theorem \ref{theo_comp} by taking the nets $(\sigma_\eps)_\eps$ and $(\tau_\eps)_\eps$ slow scale strictly nonzero. From the assumptions of slows scale type on $\omega$, $a$ and $b$ we have that all the moderate nets involved are of slow scale type. This leads to the desired conclusion.
\end{proof}

\section{Generalized Fourier integral operators and microlocal analysis}
Concluding, we present some first results of microlocal analysis for generalized Fourier integral operators provided in \cite[Section 4]{GHO:06}. A deeper investigation of the microlocal properties of
\[
A:\Gc(\Om)\to\Gc(\Om'):u\to\int_{\Om\times\R^p}\esp^{i\phi(x,y,\xi)}a(x,y,\xi)u(y)\, dy\, \dslash\xi 
\]
is current topic of research.
\subsection*{Generalized singular supports of the functional $I_\phi(a)$}
We begin with the functional 
\[
I_{\phi}(a):\Gc(\Om)\to\wt{\C}:u\to\int_{\Om\times\R^p}\esp^{i\phi(y,\xi)}a(y,\xi)u(y)\, dy\, \dslash\xi 
\]
Before defining specific regions depending on the generalized phase function $\phi$, we observe that any $\phi\in\wt{\Phi}(\Om\times\R^p)$ can be regarded as an element of ${\widetilde{\mathcal{S}}}^1_{\rm{hg}}(\Om\times\R^p\setminus 0)$ and consequently $|\nabla_\xi\phi|^2\in{\widetilde{\mathcal{S}}}^0_{\rm{hg}}(\Om\times\R^p\setminus 0)$.

Let $\Om_1$ be an open subset of $\Om$ and $\Gamma\subseteq\R^p\setminus 0$. We say that $b\in\wt{\mathcal{S}}^0(\Om\times\R^p\setminus 0)$ is \emph{invertible on $\Om_1\times\Gamma$} if for all relatively compact subsets $U$ of $\Om_1$ there exists a representative $(b_\eps)_\eps$ of $b$, a constant $r\in\R$ and $\eta\in(0,1]$ such that 
\begin{equation}
\label{est_inv_sym}
\inf_{y\in U,\xi\in\Gamma}|b_\eps(y,\xi)|\ge \eps^r
\end{equation}
for all $\eps\in(0,\eta]$. In an analogous way we say that $b\in\wt{\mathcal{S}}^0(\Om\times\R^p\setminus 0)$ is \emph{slow scale-invertible} on $\Om_1\times\Gamma$ if \eqref{est_inv_sym} holds with the inverse of some slow scale net $(s_\eps)_\eps$ in place of $\eps^r$. This kind of bounds from below hold for all representatives of the symbol $b$ once they are known to hold for one.

In the sequel $\pi_\Om$ denotes the projection of $\Om\times\R^p$ on $\Om$.
\begin{defn}
\label{def_C_phi}
Let $\phi\in\wt{\Phi}(\Om\times\R^p)$. We define $C_\phi\subseteq\Om\times\R^p\setminus 0$ as the complement of the set of all $(x_0,\xi_0)\in\Om\times\R^p\setminus 0$ with the property that there exist a relatively compact open neighborhood $U(x_0)$ of $x_0$ and a conic open neighborhood $\Gamma(\xi_0)\subseteq\R^p\setminus 0$ of $\xi_0$ such that $|\nabla_\xi\phi|^2$ is invertible on $U(x_0)\times\Gamma(\xi_0)$.
We set $\pi_\Om(C_{\phi})=S_{\phi}$ and $R_{\phi}=(S_{\phi})^{{\rm{c}}}$. 
\end{defn}
By construction $C_{\phi}$ is a closed conic subset of $\Om\times\R^p\setminus 0$ and $R_{\phi}\subseteq\Om$ is open. It is routine to check that the region $C_\phi$ coincides with the classical one when $\phi$ is classical.
\begin{prop}
\label{prop_R_phi}
The generalized symbol $|\nabla_\xi\phi|^2$ is invertible on $R_\phi\times\R^p\setminus 0$. 
\end{prop}
The more specific assumption of slow scale-invertibility concerning the generalized symbol $|\nabla_\xi\phi|^2$ is employed in the definition of the following sets.
\begin{defn}
\label{def_C_phi_ssc}
Let $\phi\in\wt{\Phi}(\Om\times\R^p)$. We define $C^\ssc_{\phi}\subseteq\Om\times\R^p\setminus 0$ as the complement of the set of all $(x_0,\xi_0)\in\Om\times\R^p\setminus 0$ with the property that there exist a relatively compact open neighborhood $U(x_0)$ of $x_0$ and a conic open neighborhood $\Gamma(\xi_0)\subseteq\R^p\setminus 0$ of $\xi_0$ such that $|\nabla_\xi\phi|^2$ is on $U(x_0)\times\Gamma(\xi_0)$. We set $\pi_\Om(C^\ssc_{\phi})=S^\ssc_{\phi}$ and $R^\ssc_{\phi}=(S^\ssc_{\phi})^{{\rm{c}}}$. 
\end{defn}
By construction $C^\ssc_{\phi}$ is a conic closed subset of $\Om\times\R^p\setminus 0$ and $R^\ssc_{\phi}\subseteq R_{\phi}\subseteq\Om$ is open. In analogy with Proposition \ref{prop_R_phi} we can prove that $|\nabla_\xi\phi|^2$ is slow scale-invertible on $R^\ssc_\phi\times\R^p\setminus 0$.  
\begin{thm}
\label{theorem_R_phi}
Let $\phi\in\wt{\Phi}(\Om\times\R^p)$ and $a\in\wt{\mathcal{S}}^m_{\rho,\delta}(\Om\times\R^p)$. 
\begin{itemize}
\item[(i)] The restriction $I_{\phi}(a)|_{R_\phi}$ of the functional $I_\phi(a)$ to the region $R_\phi$ belongs to $\G(R_\phi)$.
\item[(ii)] If $\phi\in\wt{\Phi}^\ssc(\Om\times\R^p)$ and $a\in\wt{\mathcal{S}}^{m,\ssc}_{\rho,\delta}(\Om\times\R^p)$ then $I_{\phi}(a)|_{{R^\ssc_\phi}}\in\Ginf(R^\ssc_\phi)$.
\end{itemize}
\end{thm}
Theorem \ref{theorem_R_phi} means that
\[
\singsupp_\G\, I_\phi(a)\subseteq S_\phi
\]
if $\phi\in\wt{\Phi}(\Om\times\R^p)$ and $a\in\wt{\mathcal{S}}^m_{\rho,\delta}(\Om\times\R^p)$ and that
\[
\singsupp_{\Ginf} I_\phi(a)\subseteq S^\ssc_\phi
\]
if $\phi\in\wt{\Phi}^\ssc(\Om\times\R^p)$ and $a\in\wt{\mathcal{S}}^{m,\ssc}_{\rho,\delta}(\Om\times\R^p)$. 
\begin{ex}
Returning to the first example in Section \ref{gen_sec} we are now in the position to analyze the regularity
properties of the generalized kernel functional $I_\phi(a)$ of the solution operator
$A$ corresponding to the hyperbolic Cauchy-problem. For any $v\in \Gc(\R^3)$ we have
\begin{equation}\label{hypsol}
  I_\phi(a) (v) = \int \esp^{i\phi(x,t,y,\xi)}\; a(x,t,y,\xi)\, v(x,t,y)\, dx\, dt\, dy\, \dslash\xi,
\end{equation}
where $a$ and $\phi$ are as in Section \ref{gen_sec}. Note that in the case of partial differential operators with smooth coefficients and distributional initial values the wave front set of the distributional kernel of $A$
determines the propagation of singularities from the initial data. When the
coefficients are non-differentiable functions, or even distributions or generalized
functions, matters are not yet understood in sufficient generality. Nevertheless, the
above results allow us to identify regions where the generalized kernel functional
agrees with a generalized function or is even guaranteed to be a $\Ginf$-regular
generalized function. To identify the set $C_\phi$ in this situation one simply has to study invertibility of $\d_\xi \phi (x,t,y,\xi) =
\gamma(x,t;0)- y$ as a generalized function in a neighborhood of any given point
$(x_0,t_0,y_0)$.

Under the assumptions on $c$ of the example in Section \ref{gen_sec}, the representing nets
$(\gamma_\eps(.,.;0))_{\eps\in(0,1]}$ of $\gamma$ are uniformly bounded on compact sets (e.g., when $c$ is
a bounded generalized constant). For given $(x_0,t_0)$ define the generalized domain of
dependence $D(x_0,t_0)\subseteq\R$ to be the set of accumulation points of the net
$(\gamma_\eps(x_0,t_0;0))_{\eps\in(0,1]}$. Then we have that
$$
   \{(x_0,t_0,y_0) \in\R^3 : y_0 \not\in D(x_0,t_0) \} \subseteq R_\phi.
$$

When $c\in\wt{\R}$ this may be proved by showing that if $(x_0,t_0,y_0)\in C_\phi$ then there exists an accumulation point $c'$ of a representative $(c_\eps)_\eps$ of $c$ such that $y_0=x_0-c't_0$. 
\end{ex}
\begin{ex}
As an illustrative example concerning the regions involving the regularity of the functional $I_\phi(a)$ we consider the generalized phase function on $\R^2\times\R^2$ given by $\phi_\eps(y_1,y_2,\xi_1,\xi_2)=-\eps y_1\xi_1-s_\eps y_2\xi_2$ where $(s_\eps)_\eps$ is bounded and $(s_\eps^{-1})_\eps$ is a slow scale net. Clearly $\phi:=[(\phi_\eps)_\eps]\in\wt{\Phi}^\ssc(\R^2\times\R^2)$. Simple computations show that $R_\phi=\R^2\setminus(0,0)$ and $R^\ssc_\phi=\R^2\setminus\{y_2=0\}$. We leave it to the reader to check that the oscillatory integral
\[
\int_{\R^2}\esp^{i\phi(y,\xi)}(1+\xi_1^2+\xi_2^2)^{\frac{1}{2}}\, \dslash\xi =\biggl[\biggl(\int_{\R^2}\esp^{-i\eps y_1\xi_1-is_\eps y_2\xi_2}(1+\xi_1^2+\xi_2^2)^{\frac{1}{2}}\, \dslash\xi_1\, \dslash\xi_2\biggr)_\eps\biggr]
\]
defines a generalized function in $\R^2\setminus(0,0)$ whose restriction to $\R^2\setminus\{y_2=0\}$ is regular.
\end{ex}

The Colombeau-regularity of the functional $I_\phi(a)$ is easily proved in the case of generalized symbols of order $-\infty$.  
\begin{prop}
\label{prop_smooth_sing}
\begin{itemize}
\item[{\, }]
\item[(i)] If $\phi\in\wt{\Phi}(\Om\times\R^p)$ and $a\in\wt{\mathcal{S}}^{-\infty}(\Om\times\R^p)$ then $\singsupp_{\G}I_\phi(a)=\emptyset$.
\item[(ii)] If $\phi\in\wt{\Phi}^\ssc(\Om\times\R^p)$ and $a\in\wt{\mathcal{S}}^{-\infty,\ssc}(\Om\times\R^p)$ then $\singsupp_{\Ginf}I_\phi(a)=\emptyset$.
\end{itemize}
\end{prop}
Proposition \ref{prop_smooth_sing} leads to the following result.
\begin{prop}
\label{prop_smooth_sing_cone}
\begin{itemize}
\item[{\, }]
\item[(i)] If $\phi\in\wt{\Phi}(\Om\times\R^p)$ and $a\in\wt{\mathcal{S}}^m_{\rho,\delta}(\Om\times\R^p)$ then $$\singsupp_{\G}\,I_\phi(a)\subseteq\pi_\Om(C_\phi\cap {\rm{cone\, supp}}\, a).$$
\item[(ii)] If $\phi\in\wt{\Phi}^\ssc(\Om\times\R^p)$ and $a\in\wt{\mathcal{S}}^{m,\ssc}_{\rho,\delta}(\Om\times\R^p)$ then $$\singsupp_{\Ginf}I_\phi(a)\subseteq\pi_\Om(C^\ssc_\phi\cap {\rm{cone\, supp}}\, a).$$
\end{itemize}
\end{prop}
\subsection*{Generalized wave front sets of the functional $I_\phi(a)$}
The next theorem investigates the $\G$-wave front set and the $\Ginf$-wave front set of the functional $I_\phi(a)$ under suitable assumptions on the generalized symbol $a$ and the phase function $\phi$. 
\begin{thm} 
\label{theorem_WF}
\leavevmode
\begin{itemize}
\item[(i)] Let $\phi\in\wt{\Phi}(\Om\times\R^p)$ and $a\in\wt{\mathcal{S}}^m_{\rho,\delta}(\Om\times\R^p)$. The generalized wave front set $\WF_\G I_\phi(a)$ is contained in the set $W_{\phi,a}$ of all points $(x_0,\xi_0)\in\CO{\Om}$ with the property that for all relatively compact open neighborhoods $U(x_0)$ of $x_0$, for all open conic neighborhoods $\Gamma(\xi_0)\subseteq \R^n\setminus 0$ of $\xi_0$, for all open conic neighborhoods $V$ of {\rm{cone\,supp}}\,$a\cap C_\phi$ such that $V\cap (U(x_0)\times\R^p\setminus 0)\neq\emptyset$ the generalized number 
\begin{equation}
\label{gen_num_inv}
\mathop{{\rm{Inf}}}\limits_{\substack{y\in U(x_0), \xi\in \Gamma(\xi_0)\\ (y,\theta)\in V\cap( U(x_0)\times\R^p\setminus 0)}}\frac{|\xi-\nabla_y\phi(y,\theta)|}{|\xi|+|\theta|}
\end{equation}
is not invertible.
\item[(ii)] If $\phi\in\wt{\Phi}^\ssc(\Om\times\R^p)$ and $a\in\wt{\mathcal{S}}^{m,\ssc}_{\rho,\delta}(\Om\times\R^p)$ then $\WF_{\Ginf}I_\phi(a)$ is contained in the set $W^\ssc_{\phi,a}$ of all points $(x_0,\xi_0)\in\CO{\Om}$ with the property that for all relatively compact open neighborhoods $U(x_0)$ of $x_0$, for all open conic neighborhoods $\Gamma(\xi_0)\subseteq \R^n\setminus 0$ of $\xi_0$, for all open conic neighborhoods $V$ of {\rm{cone\,supp}}\,$a\cap C^\ssc_\phi$ such that $V\cap (U(x_0)\times\R^p\setminus 0)\neq\emptyset$ the generalized number \eqref {gen_num_inv}
is not slow scale-invertible.
\end{itemize}
\end{thm}
Note that when $\phi$ is a classical phase function the set $W_{\phi,a}$ as well as the set $W^\ssc_{\phi,a}$ coincide with 
\begin{equation}
\label{set_class}
\{(x,\nabla_x\phi(x,\theta)):\, (x,\theta)\in{\rm{cone\, supp}}\,a\cap C_\phi\}.
\end{equation}
For more details see \cite[Remark 4.13]{GHO:06}.
\begin{ex}
\label{C_example3}
Theorem \ref{theorem_WF} can be employed for investigating the generalized wave front sets of the kernel $K_A:=I_\phi(a)$ of the Fourier integral operator introduced in the first example of Section \ref{gen_sec}. For simplicity we assume that $c$ is a bounded generalized constant in $\wt{\R}$ and that $a=1$. Let $((x_0,t_0,y_0),\xi_0)\in\WF_\G K_A$. From the first assertion of Theorem \ref{theorem_WF} we know that the generalized number given by 
\begin{equation}
\label{non_inv_ex}
\inf_{\substack{(x,t,y)\in U, \xi\in \Gamma\\ ((x,t,y),\theta)\in V\cap(U\times\R\setminus 0)}}\frac{|\xi-(\theta,-c_\eps\theta,-\theta)|}{|\xi|+|\theta|}
\end{equation}
is not invertible, for every choice of neighborhoods $U$ of $(x_0,t_0,y_0)$, $\Gamma$ of $\xi_0$ and $V$ of $C_\phi$. Note that it is not restrictive to assume that $|\theta|=1$. We fix some sequences $(U_n)_n$, $(\Gamma_n)_n$ and $(V_n)_n$ of neighborhoods shrinking to $(x_0,t_0,y_0)$, $\{\xi_0\lambda:\lambda>0\}$ and $C_\phi$ respectively. By \eqref{non_inv_ex} we find a sequence $\eps_n$ tending to $0$ such that for all $n\in\N$ there exists $\xi_n\in\Gamma_n$, $(x_n,t_n,y_n,\theta_n)\in V_n$ with $|\theta_n|=1$ and $(x_n,t_n,y_n)\in U_n$ such that
\[
|\xi_n-(\theta_n,-c_{\eps_n}\theta_n,-\theta_n)|\le \eps_n(|\xi_n|+1).
\]
In particular, $\xi_n$ remains bounded. Passing to suitable subsequences we obtain that there exist $\theta$ such that $(x_0,t_0,y_0,\theta)\in C_\phi$, an accumulation point $c'$ of $(c_\eps)_\eps$ and a multiple $\xi'$ of $\xi_0$ such that $\xi'=(\theta,-c'\theta,-\theta)$. It follows that
\[
\frac{\xi_0}{|\xi_0|}=
\frac{\xi'}{|\xi'|}=\frac{1}{\sqrt{2+(c')^2}|\theta|}\,(\theta,-c'\theta,-\theta).
\]

In other words the $\G$-wave front set of the kernel $K_A$ is contained in the set of points of the form $((x_0,t_0,y_0),(\theta_0,-c'\theta_0,-\theta_0))$ where $(x_0,t_0,y_0,\theta_0)\in C_\phi$ and $c'$ is an accumulation point of a net representing $c$. Since in the classical case (when $c\in\R$) the distributional wave front set of the corresponding kernel is the set $\{((x_0,t_0,y_0),(\theta_0,-c\theta_0,-\theta_0)):\, (x_0,t_0,y_0,\theta_0)\in C_\phi\}$, the result obtained above for $\WF_\G K_A$ is a generalization in line with what we deduced about the regions $R_\phi$ and $C_\phi$.
\end{ex}
\subsection*{Particular case: generalized pseudodifferential operators}
Finally, we consider a generalized pseudodifferential operator $a(x,D)$ on $\Om$ and its kernel $K_{a(x,D)}\in\LL(\Gc(\Om\times\Om),\wt{\C})$. By Remark 4.15 in \cite{GHO:06}, we have that $\WF_\G(K_{a(x,D)})$ is contained in the normal bundle of the diagonal in $\Om\times\Om$ when $a\in\wt{\mathcal{S}}^{m}_{\rho,\delta}(\Om\times\R^n)$ and that $\WF_{\Ginf}(K_{a(x,D)})$ is a subset of the normal bundle of the diagonal in $\Om\times\Om$ when $a$ is of slow scale type. We define the sets
\[
\WF_\G(a(x,D))=\{(x,\xi)\in\CO{\Om}:\, (x,x,\xi,-\xi)\in\WF_\G(K_{a(x,D)})\}
\]
and
\[
\WF_{\Ginf}(a(x,D))=\{(x,\xi)\in\CO{\Om}:\, (x,x,\xi,-\xi)\in\WF_{\Ginf}(K_{a(x,D)})\}.
\]
From Theorem \ref{theorem_WF} one deduces the following.
\begin{prop}
\label{prop_pseudo}
Let $a(x,D)$ be a generalized pseudodifferential operator.
\begin{itemize}
\item[(i)] If $a\in\wt{\mathcal{S}}^{m}_{\rho,\delta}(\Om\times\R^n)$ then $\WF_\G(a(x,D))\subseteq\mu\, \supp_\G(a)$.
\item[(ii)]If $a\in\wt{\mathcal{S}}^{m,\ssc}_{\rho,\delta}(\Om\times\R^n)$ then $\WF_{\Ginf}(a(x,D))\subseteq\mu\, \supp_{\Ginf}(a)$.
\end{itemize}
\end{prop}


\subsection*{Acknowledgment}
The author would like to express her gratitude to Professor L. Rodino and Professor M. W. Wong for the kind invitation to the session on Pseudodifferential Operators of the 2007 ISAAC Congress at METU in Ankara.
\end{document}